\documentclass[oneside,11pt,reqno,a4paper]{amsart}
\usepackage[T1]{fontenc}
\usepackage[utf8]{inputenc}
\usepackage[top=1cm,bottom=2cm,right=2cm,left=2cm,includehead]{geometry}
\usepackage{amsmath,amsfonts,amssymb,amscd,amsthm}
\usepackage[usenames,dvipsnames,svgnames]{xcolor}
\usepackage[numbers,sort&compress]{natbib}
\usepackage{doi}
\usepackage{graphicx}
\usepackage[font={small}]{caption}
\usepackage[font={small}]{subcaption}
\usepackage{autonum} % Only enumerate what is referred
\usepackage{hyperref} % hyperref insertions
\usepackage{enumitem}

\newcommand{\thetitle}{\uppercase{The aggregated unfitted finite element method
\\ on parallel tree-based adaptive meshes}}

\newcommand{\theauthors}{Santiago Badia\textsuperscript{a,b}, Alberto F.
Mart\'in\textsuperscript{a}, Eric
Neiva\textsuperscript{b,c,}\footnote{Corresponding author. \\ Emails:
\theemailsanddate} and Francesc Verdugo\textsuperscript{b}}
\newcommand{\theauths}{S. Badia, A. F. Mart\'in, E. Neiva and F. Verdugo}

\newcommand{\theaffiliations}{
	\textsuperscript{a} School of Mathematics, Monash University, Clayton,
	Victoria, 3800, Australia. \\ [0.5em]
	\textsuperscript{b} CIMNE – Centre Internacional de M\`etodes Num\`erics en
	Enginyeria, \\ Edifici C1, Campus Nord UPC, C. Gran Capit\`a S/N, 08034
	Barcelona, Spain. \\ [0.5em] \textsuperscript{c} Department of Civil and
	Environmental Engineering, Universitat Polit\`ecnica de Catalunya, \\ Edifici
	C2, Campus Nord UPC, C. Jordi Girona 1-3, 08034 Barcelona, Spain. \\ [0.5em]
}

\newcommand{\theemailsanddate}{
  \texttt{\{santiago.badia,alberto.martin\}@monash.edu}, 
  \texttt{\{eneiva,fverdugo\}@cimne.upc.edu} \\
	\today
}

\newcommand{\thethanks}{Financial support from the European Commission under the
FET-HPC ExaQUte project (Grant agreement ID: 800898) within the Horizon 2020
Framework Programme and the project RTI2018-096898-B-I00 from the
``FEDER/Ministerio de Ciencia e Innovación – Agencia Estatal de Investigación'' is
gratefully acknowledged. F. Verdudo acknowledges support from the Spanish Ministry
of Economy and Competitiveness through the ``Severo Ochoa Programme for Centers of
Excellence in R\&D (CEX2018-000797-S)" and \emph{Secretaria d'Universitats i
Recerca} of the Catalan Government in the framework of the Beatriu Pinós Program
(Grant Id.: 2016 BP 00145). E. Neiva gratefully acknowledges the support received
from the Catalan Government through a FI fellowship (2019 FI-B2-00090; 2018
FI-B1-00095; 2017 FI-B-00219). Financial support to CIMNE via the CERCA Programme
/ Generalitat de Catalunya is also acknowledged. The authors thankfully
acknowledge the computer resources at Marenostrum-IV and the technical support
provided by the Barcelona Supercomputing Center (RES-ActivityID: FI-2019-1-0007,
IM-2019-2-0007, IM-2019-3-0008). This work was also supported by computational
resources provided by the Australian Government through NCI under the National
Computational Merit Allocation Scheme.}

\hypersetup{
    bookmarks=true,
    breaklinks=true,
    bookmarksopen=true,
    pdftitle={\thetitle},  % title
    pdfauthor={\theauths}, % author
    colorlinks=true,       % false: boxed links; true: colored links
    linkcolor=black,       % color of internal links (change box color with
    citecolor=blue,        % color of links to bibliography
    filecolor=black,       % color of file links
    urlcolor=blue          % color of external links
}

\usepackage{fancyhdr}

\usepackage{newtxtext} % Restyling of text fonts
\usepackage{newtxmath} % Restyling of math fonts

\usepackage{acronym}
\usepackage[most]{tcolorbox}

\usepackage{booktabs}

\usepackage{tikz}
\definecolor{myellow}{RGB}{255,230,128}
\definecolor{gray20}{RGB}{204,204,204}
\definecolor{mygray}{RGB}{204,204,204}
\definecolor{mygreen}{RGB}{138,203,95}
\definecolor{myblue}{RGB}{77,151,214}

\definecolor{lstgrey}{rgb}{0.95,0.95,0.95}
\usepackage{listings}
\usepackage[vlined]{algorithm2e}
\usepackage{xcolor}

\SetAlFnt{\small}
\SetArgSty{textnormal}

\SetAlCapSty{xAlCapSty}

\SetCommentSty{xCommentSty}

\SetNlSty{mynlfont}{}{} 

\LinesNumbered

\SetSideCommentRight

\DontPrintSemicolon

\RestyleAlgo{algoruled}

\SetInd{0.5em}{0.5em}

\acrodef{pde}[PDE]{partial differential equation}
\acrodef{bvp}[BVP]{boundary value problem}
\acrodef{amr}[AMR]{Adaptive mesh refinement and coarsening}
\acrodef{ls}[LS]{Level-Set}
\acrodef{dof}[DOF]{degree of freedom}
\acrodefplural{dof}[DOFs]{degrees of freedom}
\acrodef{vef}[VEF]{vertex, edge, and face}
\acrodefplural{vef}[VEFs]{vertices, edges, and faces}
\acrodef{cg}[CG]{continuous Galerkin}
\acrodef{dg}[DG]{discontinuous Galerkin}
\acrodef{vms}[VMS]{variational multiscale}
\acrodef{sps}[SPS]{symmetric projection stabilization}
\acrodef{fe}[FE]{finite element}
\acrodef{fem}[FEM]{finite element method}
\acrodef{xfem}[XFEM]{extended finite element method}
\acrodef{agfe}[agFE]{aggregated finite element}
\acrodef{agfem}[AgFEM]{aggregated finite element method}
\acrodef{cgm}[CG]{conjugate gradient}
\acrodef{amg}[AMG]{algebraic multigrid}
\acrodef{lb}[LB]{Li and Bettess}
\acrodef{ob}[OB]{O\~{n}ate and Bugeda}
\acrodef{dd}[DD]{domain decomposition}
\acrodef{mpi}[MPI]{message passing interface}
\acrodef{sfc}[SFC]{space filling curve}
\acrodefplural{sfc}[SFCs]{space filling curves}

\newtheorem{theorem}{Theorem}[section]
\newtheorem{lemma}[theorem]{Lemma}
\newtheorem{proposition}[theorem]{Proposition}
\newtheorem{corollary}[theorem]{Corollary}
\newtheorem{definition}[theorem]{Definition}

\newtheorem{assumption}[theorem]{Assumption}
\newtheorem{remark}[theorem]{Remark}

\newcommand{\myadded}[1]{{\leavevmode #1}}

\newcommand{\closure}[2][3]{%
{}\mkern#1mu\overline{\mkern-#1mu#2}}

\graphicspath{{figures/}{figures/geometries_and_meshes/}{figures/cheese_Q1_WS/}{figures/cheese_Q2_WS_TLC7k/}{figures/popcorn_Q1_WS/}{figures/eta_0_sensitivity/}{figures/error_vs_dofs_serial/}{figures/error_vs_dofs_parallel/}{figures/condition_numbers/}{figures/iterations/}{figures/shock_on_cube/}}

\def\fempar{{\texttt{FEMPAR}}}

\def\p4est{{\texttt{p4est}}}
\def\t8code{{\texttt{t8code}}}

\def\petsc{{\texttt{PETSc}}}

\def\gamg{{\texttt{GAMG}}}

\def\cell{T}
\def\T{{\mathcal{T}}}
\def\V{{\mathcal{V}}}
\def\M{{\mathcal{M}}}

\def\S{{\mathcal{S}}}
\def\x{{\boldsymbol{x}}}
\def\u{{\underline{\mathbf{u}}}}

\def\H{{\mathrm{H}}}
\def\F{{\mathrm{F}}}
\def\A{{\mathrm{A}}}
\def\C{{\mathrm{C}}}
\def\L{{\mathrm{L}}}

\def\X{{\mathrm{X}}}
\def\Y{{\mathrm{Y}}}

\def\TG{{\mathrm{TG}}}
\def\RG{{\mathrm{RG}}}

\def\art{{\mathrm{art}}}
\def\ls{{\mathrm{ls}}}
\def\wp{{\mathrm{W}}}
\def\out{{\mathrm{O}}}
\def\ip{{\mathrm{I}}}
\def\ag{{\mathrm{ag}}}
\def\act{{\mathrm{act}}}
\def\ncf{{\mathrm{ncf}}}
\def\std{{\mathrm{std}}}

\makeatletter
\newcommand{\opnorm}{\@ifstar\@opnorms\@opnorm}
\newcommand{\@opnorms}[1]{%
  \left|\mkern-1.5mu\left|\mkern-1.5mu\left|
   #1
  \right|\mkern-1.5mu\right|\mkern-1.5mu\right|
}
\newcommand{\@opnorm}[2][]{%
  \mathopen{#1|\mkern-1.5mu#1|\mkern-1.5mu#1|}
  #2
  \mathclose{#1|\mkern-1.5mu#1|\mkern-1.5mu#1|}
}
\makeatother

\def\opnormh#1{\opnorm{#1}_h}
\def\opnormvh#1{\opnorm{#1}_{\V(h)}}

\newcommand{\restrict}[2]{{\left. #1 \right|_{#2}}}

\begin{document}

\thispagestyle{empty}

\renewcommand*{\thefootnote}{\fnsymbol{footnote}}

\begin{center}
{ \bf {\thetitle}}

\vspace*{1em}

\theauthors

\vspace*{1em}

\theaffiliations

\end{center}

\setcounter{footnote}{0}
\renewcommand*{\thefootnote}{\arabic{footnote}}

\begin{center}

{\bf Abstract}

\vspace*{1em}

\begin{minipage}{0.9\textwidth}
\begin{small}

In this work, we present an adaptive unfitted finite element scheme that
combines the aggregated finite element method with parallel adaptive mesh
refinement. We introduce a novel scalable distributed-memory implementation of
the resulting scheme on locally-adapted Cartesian forest-of-trees meshes. We
propose a two-step algorithm to construct the finite element space at hand 
by means of a discrete extension operator that
carefully mixes aggregation constraints of problematic degrees of freedom, which
get rid of the small cut cell problem, and standard hanging degree of freedom
constraints, which ensure trace continuity on non-conforming meshes. Following
this approach, we derive a finite element space that can be expressed as the
original one plus well-defined linear constraints. Moreover, it requires minimum
parallelization effort, using standard functionality available in existing
large-scale finite element codes. Numerical experiments demonstrate its optimal
mesh adaptation capability, robustness to cut location and parallel efficiency,
on classical Poisson $hp$-adaptivity benchmarks. Our work opens the path to
functional and geometrical error-driven dynamic mesh adaptation with the
aggregated finite element method in large-scale realistic scenarios. Likewise,
it can offer guidance for bridging other scalable unfitted methods and parallel
adaptive mesh refinement.

\end{small}

\end{minipage}
\end{center}

\vspace*{1em}

%\end{abstract}

%\maketitle

%\noindent{\bf 2010 Mathematics Subject Classification:} 65N12, 65N15, 65N30


%\tableofcontents

\section{Introduction} \label{sec:int}

\ac{amr} using adaptive tree-based meshes is attracting growing interest in
large-scale simulations of physical problems modelled with \acp{pde}. Research
over the past few years has demonstrated that tree-based \ac{amr} enables
efficient data storage and mesh traversal, fast computation of mesh hierarchy
and cell adjacency and extremely scalable partitioning and dynamic load
balancing. Although several cell topologies have been
studied~\cite{Holke2018,Burstedde2016}, attention has centred around
quadrilateral (2D) or hexahedral (3D) adaptive meshes endowed with standard
isotropic 1:4 (2D) and 1:8 (3D) refinement rules. They form tree structures that
are commonly known as quadtrees %~\cite{Finkel1974} or octrees~\cite{Meagher1982}
or \emph{forest-of-quadtrees} or \emph{-octrees}, when the former are patched
together. There is ample literature concerning single-octree
meshes %~\cite{Sundar2007,TiankaiTu2005,Sundar2008} 
and extensions to
forest-of-octrees~\cite{burstedde_p4est_2011,Isaac2014}. State-of-the art in
these techniques is available at the open source parallel forest-of-octrees
meshing engine \texttt{p4est}~\cite{burstedde_p4est_2011}.

In the context of parallel adaptive \ac{fe} solvers, forest-of-trees have been
an essential component in many large-scale application
problems~\cite{Olm2018,Rudi2015,Burstedde2008,neiva2019scalable}.
As they provide multi-resolution by local mesh adaptation, they are convenient,
among others, in the following three scenarios: (1) a \emph{priori} mesh
refinement, when the \ac{bvp} exhibits local features that must be captured with
high resolution, but are known in advance, see
e.g.~\cite{Olm2018,neiva2019scalable}; (2) a \emph{posteriori} mesh refinement,
driven by error estimators~\cite{Ainsworth2000}, for solutions of \acp{bvp}
whose local features are not known or spatially evolve over
time~\cite{Rudi2015}; and (3) to control geometric approximation errors of
static or moving boundaries and interfaces, in combination with unfitted \ac{fe}
methods~\cite{burman2019posteriori}.

In spite of their scalable multi-resolution capability, practical integration
of forest-of-trees in large-scale \ac{fe} codes is hindered by the fact that,
in general, they are non-conforming meshes. In particular, they contain the
widely known \emph{hanging} \acp{vef}, occurring at the interface of
neighbouring cells with different refinement levels. Mesh non-conformity
increases implementation complexity of \ac{fe} methods, especially, when they
are conforming. In this case, \acp{dof} lying on \emph{hanging} \acp{vef}
cannot have an arbitrary value, they must be constrained to guarantee trace
continuity across cell interfaces. Set up (during \ac{fe} space construction)
and application (during \ac{fe} assembly) of hanging \ac{dof} constraints have
been thoroughly studied~\cite{Rheinboldt1980,Shephard1984}. Several large-scale
\ac{fe} software packages also provide state-of-the-art treatment of hanging
\acp{dof}~\cite{bangerth_algorithms_2012,Badia2019b}. They accommodate to
standard practice of constraining the processor-local portion of the mesh to
the cells the processor owns and a single layer of adjacent off-processor
cells, the so-called ghost cells; it is well-established that hanging \ac{dof}
constraints do not expand beyond a single layer of ghost cells, see
e.g.~\cite{Badia2019b} for comprehensive and rigorous demonstration.

While research is mature on generic parallel tree-based adaptive \ac{fe}
methods, enabling applications in arbitrarily complex geometries has been
vastly overlooked. Usage of \emph{body-fitted} meshes (i.e. those whose faces
conform to the domain boundary) is not a choice in large-scale parallel
computations, due to the bottleneck in generating and partitioning large
unstructured meshes. On the other hand, unfitted (also known as embedded or
immersed) \ac{fe} methods blend exceptionally well with adaptive tree-based
meshes. However, to the authors' best knowledge, this line of research has been
barely explored. The main advantage of unfitted methods is that, instead of
requiring body-fitted meshes, they embed the domain of interest in a
geometrically simple background grid (usually a uniform or an adaptive
Cartesian grid), which can be generated much more efficiently. Unfortunately,
unfitted \ac{fe} methods also suffer from well-known drawbacks, above all, the
so-called \emph{small cut cell problem}. The intersection of a background cell
with the physical domain can be arbitrarily small, with unbounded aspect
ratios. This leads to severely ill-conditioned systems of algebraic linear
equations, if no specific strategy alleviates this issue~\cite{DePrenter2017}.

Many different unfitted methodologies have emerged that cope with the small cut
cell problem (see, e.g., the cutFEM method~\cite{burman_cutfem_2015}, the Finite 
Cell Method~\cite{Schillinger2015}, the AgFEM method~\cite{Badia2018}, and some 
variants of the XFEM method~\cite{sukumar_modeling_2001}). They have also been 
useful for many multi-phase and multi-physics applications with moving interfaces 
(e.g.\ fracture mechanics~\cite{Sukumar2000}, fluid–structure
interaction~\cite{Massing2015}, free surface flows~\cite{Sauerland2011}), in
applications with varying domain topologies (e.g. shape or topology
optimization~\cite{Burman2018}, or in applications where the geometry is not
described by CAD data (e.g. medical simulations based on CT-scan
images~\cite{Nguyen2017}). However, fewer works have addressed scalable parallel
unfitted methods, which are essential for realistic large-scale applications.
Notable exceptions are the works in~\cite{badia_robust_2017,Jomo2018}, that
design tailored preconditioners for unfitted methods.
%In~\cite{Jomo2018} a
%single-level Additive Schwarz preconditioner scales up to 4 million \acp{dof},
%whereas~\cite{badia_robust_2017} considers a two-level balancing domain
%decomposition preconditioner~\cite{Mandel2003} and reports results up to almost
%6 million \acp{dof}. 
Recent parallelization strategies~\cite{Verdugo2019} have
taken a different path, by considering enhanced \ac{fe} formulations that lead
to well-conditioned system  matrices, regardless of cut location. As a result,
they are amenable to resolution with state-of-the-art large-scale iterative
linear solvers such as \ac{amg}, %~\cite{Ruge1987,Vanek2001}, 
for which there are
highly-scalable parallel implementations in renowned scientific computing
packages such as %\trilinos~\cite{Heroux2005} or
\petsc~\cite{petsc-user-ref}. This approach yields superior
scalability, e.g.\ in~\cite{Verdugo2019}, a distributed-memory implementation of
the aggregated \ac{fem}, referred to as Ag\ac{fem}, scales up to 16K cores and
up to nearly 300M \acp{dof}, on the Poisson equation in complex 3D domains,
discretised with uniform meshes.

This paper aims to fill the gap between parallel adaptive tree-based meshing and
robust and scalable unfitted \ac{fe} techniques. We restrict the scope of our
work to Ag\ac{fem}~\cite{Badia2018}, although other enhanced unfitted
formulations, such as the CutFEM method~\cite{burman_cutfem_2015}, could also be
considered. Ag\ac{fem} is based 
on a discrete extension operator from well-posed to ill-posed \acp{dof}.
The definition of this operator relies on aggregating cells on the boundary to remove
basis functions associated with badly cut cells and, thus, eliminate
ill-conditioning issues. The formulation enjoys good numerical properties, such
as stability, condition number bounds, optimal convergence, and continuity with
respect to data; detailed mathematical analysis of the method is included
in~\cite{Badia2018} for elliptic problems and in~\cite{Badia2018a} for the
Stokes equation. Conversely, cell aggregation locally increases the
characteristic size of the resulting aggregated mesh, which has an impact on the
constant (not order) in the convergence of the method, even though such constant
has experimentally been observed to be similar to the one of the non-aggregated
\ac{fem}~\cite{Badia2018}. In this work, we demonstrate that Ag\ac{fem} is also
amenable to parallel tree-based meshes and optimal error-driven $h$-adaptivity
in practical large-scale \ac{fe} applications. We refer to the resulting method
as $h$-Ag\ac{fem}. Furthermore, since $h$-Ag\ac{fem} is capable of adding mesh
resolution wherever it is needed, it is not hindered by the local accuracy issue
mentioned above.

The outline of this work is as follows. We detail first, in
Section~\ref{sec:agfem}, a possible way to construct conforming Ag\ac{fe} spaces
on top of non-conforming (adaptive) meshes. The main challenge is to combine the
linear constraints arising from both hanging and problematic \acp{dof}. We propose
a two-tier approach that generates first the hanging \ac{dof} constraints and then
modifies them with the Ag\ac{fem} constraints. We show that this technique yields
unified linear constraints that have no circular dependencies. Furthermore,
distributed-memory extension of the method can be implemented using common
functionality of large-scale \ac{fe} software packages. In our case, we have
implemented the method in the large-scale FE software package
\fempar~\cite{badia-fempar}, which exploits the highly-scalable forest-of-tree
mesh engine \p4est{}. In the numerical tests of Section~\ref{sec:numericals}, we
consider the Poisson equation as model problem on several complex geometries and
$hp$-\ac{fem} standard benchmarks. We demonstrate similar accuracy and optimal
convergence as with standard body-fitted $h$-\ac{fem} and consistent robustness
and scalability, using \emph{out-of-the-box} \ac{amg} solvers from the \petsc{}
project. We draw the main conclusions of our work in
Section~\ref{sec:conclusions}. \myadded{Finally, we supplement the paper contents
with an exhaustive step-by-step derivation of Ag\ac{fe} spaces, in
Appendix~\ref{appendix:agfespace}, and with the proof that Ag\ac{fe} spaces on
nonconforming meshes retain the good numerical properties ensured on uniform
meshes, in Appendix~\ref{appendix:proofs}.}

\section{The aggregated unfitted finite element method on non-conforming
adaptive meshes}\label{sec:agfem}

Our goal is to define conforming, \ac{cg}, Ag\ac{fe} spaces on top of
non-conforming adaptive meshes. In this section, we introduce notation and
concepts necessary to construct such spaces. We start with a typical immersed
boundary setup on a non-conforming mesh in Section~\ref{sec:geometry}; for
scalability reasons, we restrict ourselves to the particular case of
(non-conforming) forest-of-trees meshes. We continue with the description of the
cell aggregation scheme in Section~\ref{sec:cell-aggr}, which is the cornerstone
of Ag\ac{fem}. As stated in Section~\ref{sec:int}, our two-level strategy to
construct Ag\ac{fe} spaces is (1) generation of \ac{dof} constraints enforcing
conformity on hanging \acp{vef}, followed by (2) generation of \ac{dof}
aggregation constraints, judiciously combined with the previous ones. To mirror
our approach in this text, we define first standard conforming Lagrangian \ac{fe}
spaces in Section~\ref{sec:std-fe-spaces}, then we lay out aggregated counterparts
in Section~\ref{sec:ag-fe-spaces}. At first, we look at the sequential version of
these spaces; distributed-memory extension is covered in
Section~\ref{sec:dm-impl}.

\subsection{Embedded boundary setup}\label{sec:geometry}

Let $\Omega \subset \mathbb{R}^d$ be an open bounded polygonal domain, with $d \in
\{2,3\}$ the number of spatial dimensions, in which our \ac{pde} problem is posed.
As usual, in the context of embedded boundary methods, let $\Omega^\art$ be an
\emph{artificial} or \emph{background} domain with a simple shape that includes
the \emph{physical} one, i.e.\ $\Omega \subset \Omega^\art$, \myadded{as in
Figure~\ref{fig:immersed-setup-a}}. We assume that $\Omega^\art$ can be easily
meshed using, e.g.~Cartesian grids or unstructured $d$-simplexes. Let $\T_h$
represent a partition of $\Omega^\art$ into cells, with $h_\cell$ the
characteristic size of a cell $\cell \in \T_h$ and $h \doteq \max_{\cell \in \T_h}
h_\cell$. Any $\cell \in \T_h$ is the image of a differentiable homeomorphism
$\Phi_\cell$  over a set of admissible open reference
$d$-polytopes~\cite{badia-fempar}, such as $d$-simplexes or $d$-cubes. Let
$\mathcal{F}_\cell$ denote the \emph{disjoint} $d-1$-skeleton of $\cell \in \T_h$,
e.g.~$\mathcal{F}_\cell$ is composed of vertices, edges and faces for $d = 3$.
Hereafter, we abuse terminology and refer to $\mathcal{F}_\cell$ as the set of
\acp{vef} of $\cell \in \T_h$. We assume that $\T_h$ is \emph{non-conforming}. In
particular, we allow that
\begin{assumption}\label{ass:nonconf}
	For any two cells $\cell,\cell' \in \T_h$, satisfying $\closure{\cell} \cap
	\closure{\cell'} \neq \emptyset$, there exists $f \in \mathcal{F}_\cell$ and
	$f' \in \mathcal{F}_{\cell'}$ such that: (i) $\closure{f} = \closure{f'} =
	\closure{\cell} \cap \closure{\cell'}$; or (ii) $\closure{f} = \closure{\cell}
	\cap \closure{\cell'}$ and $f \subsetneq f'$, or vice versa.
\end{assumption}

In other words, any pair of intersecting \acp{vef} in $\T_h$ are either identical
or one is a proper subset of the other. We notice that meshes satisfying (i)
everywhere are \emph{conforming}. On the other hand, a \emph{hanging} \ac{vef} is
any \ac{vef} $g \in \mathcal{F}_{\cell}$ satisfying $g \subset \closure{f}$ and
$\closure{f} = \closure{\cell} \cap \closure{\cell'}$ in (ii), while $f'$ is
referred to as the \emph{owner} \ac{vef} of $g$, \myadded{see
Figure~\ref{fig:immersed-setup-bb}}. Typical examples of hanging \acp{vef} in,
e.g.~2D, are cell vertices lying in the middle of an edge of a coarser cell.

\begin{figure}[ht!]
  \centering
  \begin{subfigure}{0.39\textwidth}
    \centering
    \includegraphics[width=0.95\textwidth]{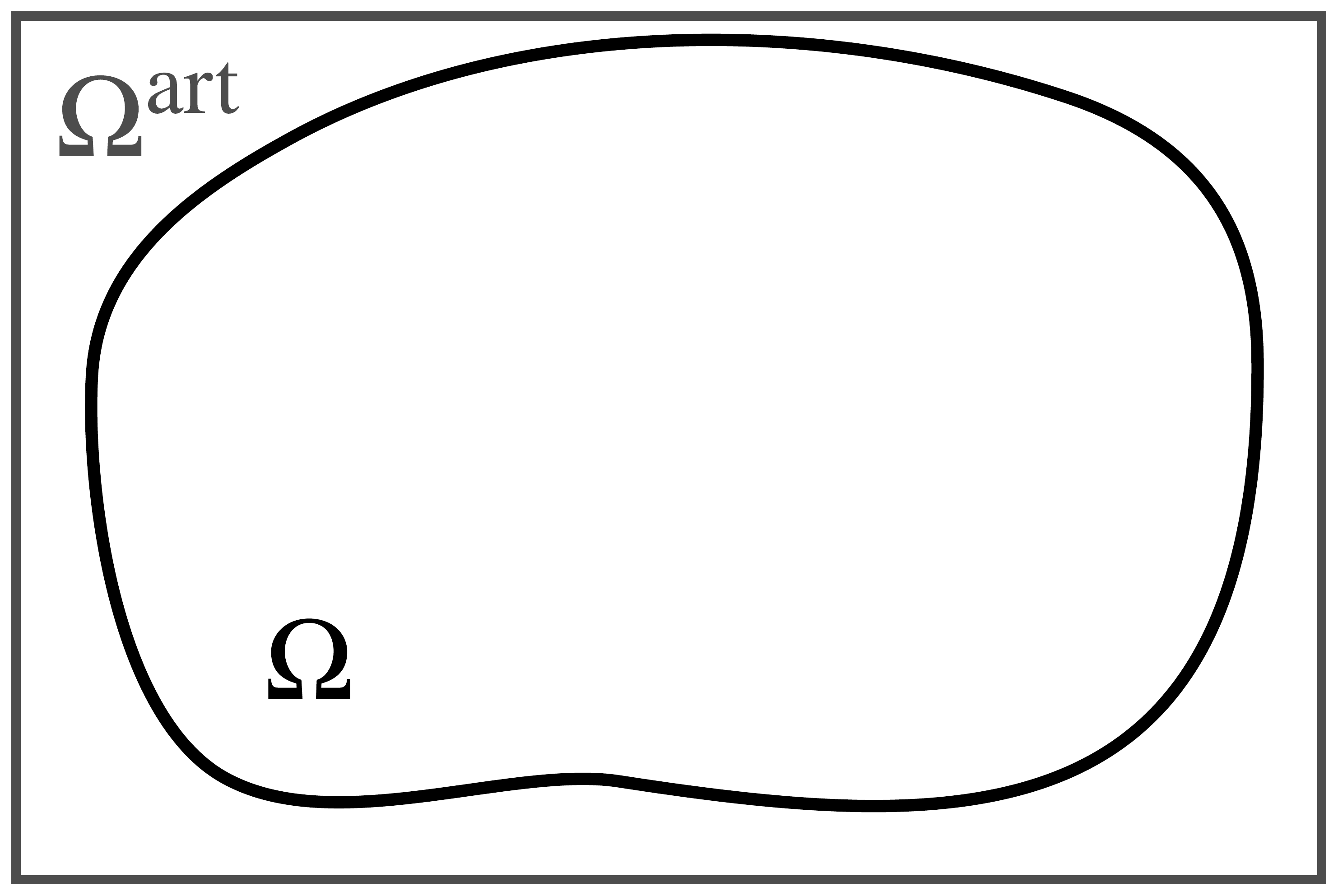}
    \caption{}
    \label{fig:immersed-setup-a}
  \end{subfigure}
  \begin{subfigure}{0.39\textwidth}
    \centering
    \includegraphics[width=0.95\textwidth]{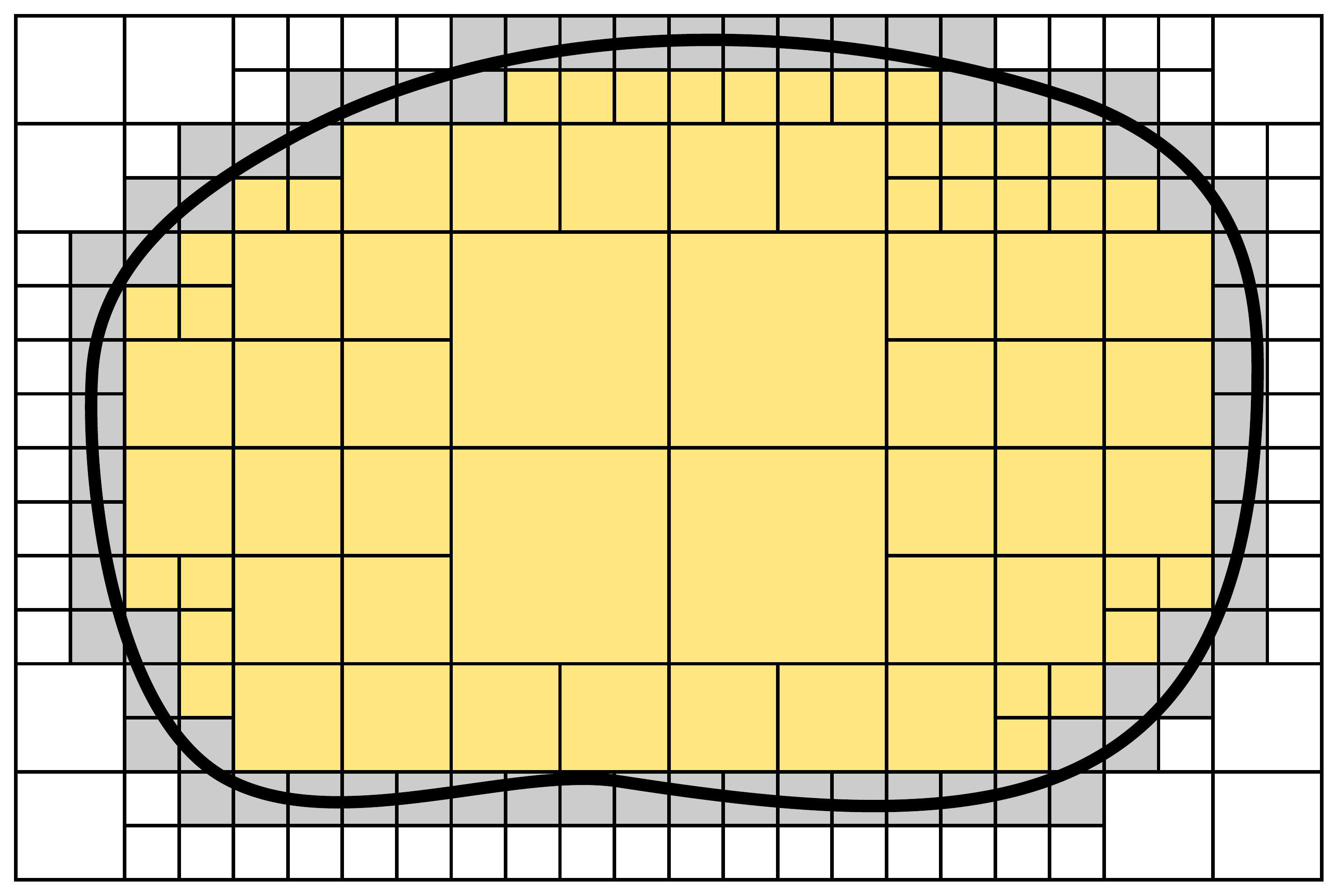}
    \caption{}
    \label{fig:immersed-setup-b}
  \end{subfigure}
  \begin{subfigure}{0.20\textwidth}
    \centering
    \includegraphics[width=\textwidth]{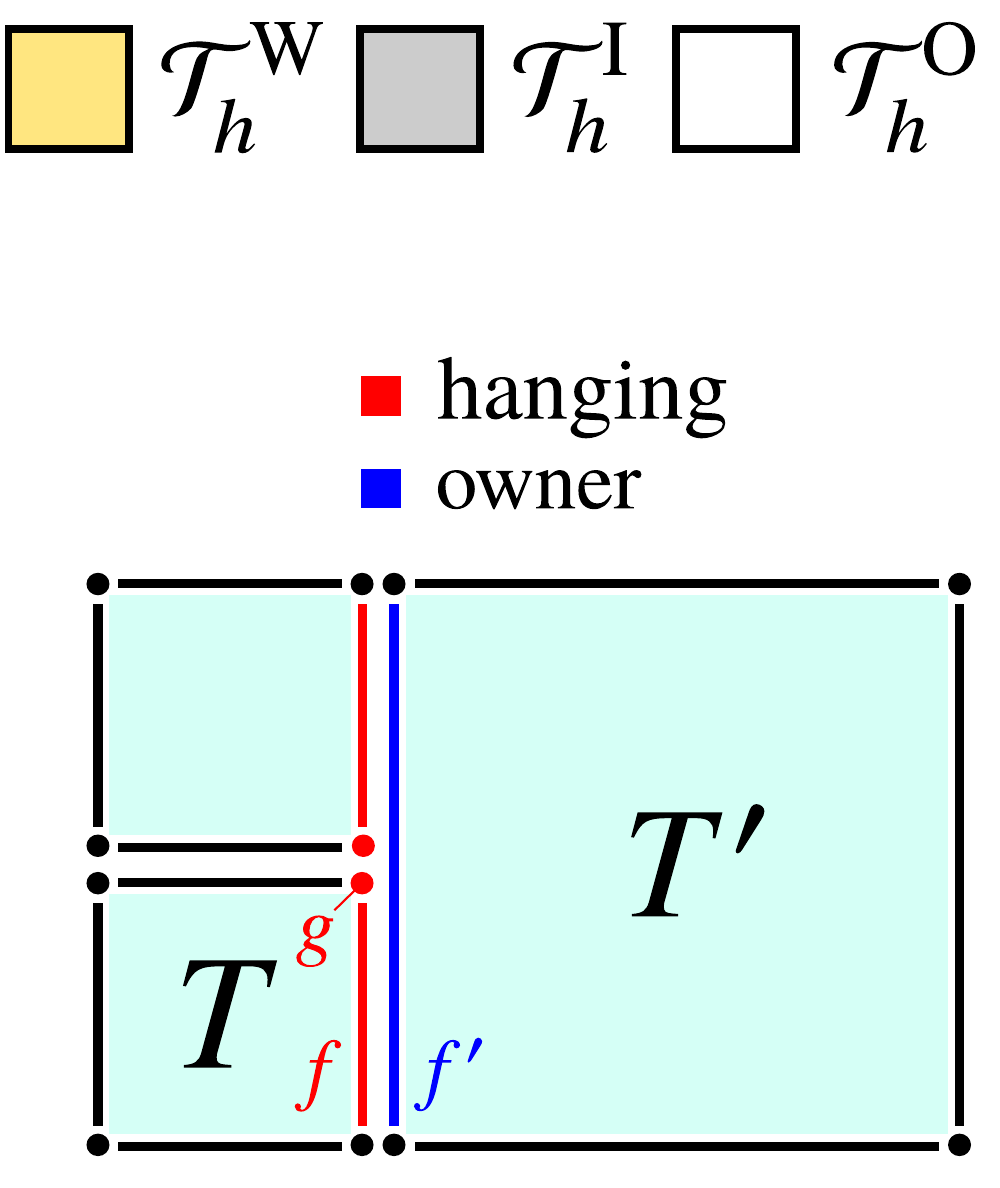}
    \caption{}
    \label{fig:immersed-setup-bb}
  \end{subfigure}
  \caption{\myadded{We define in~(\textsc{a}) a \emph{simple} artificial domain 
  $\Omega^\art$, which includes the physical one $\Omega$. In~(\textsc{b}), the 
  background mesh $\T_h$  meets the 2:1 balance condition. It is partitioned into 
  \emph{well-posed} $\T_h^\wp$, \emph{ill-posed} $\T_h^\ip$ and \emph{exterior} 
  $\T_h^\out$ cells (we assume $\eta_0 = 1$, i.e. well-posed iff interior and 
  ill-posed iff cut). In~(\textsc{c}), we illustrate Assumption~\ref{ass:nonconf} 
  (ii) with a hanging vertex $g$, a hanging edge $f$ and their owner edge $f'$.}}
  \label{fig:immersed-setup}
\end{figure}

As outlined in Section~\ref{sec:int}, we restrict ourselves to the family of
(non-conforming) \emph{forest-of-trees} meshes. This kind of meshes are derived
from recursive application of standard isotropic 1:$2^d$ refinement rules on a
(possibly unstructured) initial coarse mesh. By construction, they satisfy
Assumption~\ref{ass:nonconf}. We choose forest-of-trees, because they are a
well-established approach for parallel scalable adaptive mesh generation and
partitioning~\cite{burstedde_p4est_2011}; in particular, we aim to exploit a
recent highly-scalable parallel \ac{fe} framework that supports $h$-adaptivity
on forest-of-trees~\cite{Badia2019b}.

For \ac{fe} applications, mesh non-conformity hardens the construction of
conforming \ac{fe} spaces and the subsequent steps in the simulation. For the sake
of alleviating this extra complexity, we follow common
practice~\cite{burstedde_p4est_2011,bangerth_algorithms_2012} of enforcing the
\emph{2:1-balance} or \emph{1-irregularity} condition, that prescribes, at most,
2:1 size relations between neighbouring cells, \myadded{see
Figure~\ref{fig:immersed-setup-b}}. 2:1 balance ensures that hanging \ac{dof}
% \revtext{Remove details on 0-balance}{A general definition of this condition depends 
% on the dimension of the geometrical entity across two neighbouring 
% cells~\cite[Definition 2.8]{Badia2019b}. In our context, for simplicity, we adopt 
% the criterion that 2:1 size relations must hold for any two geometrically neighbouring 
% cells (i.e.~across a vertex, edge or face), the so-called 2:1 0-balance.\footnote{We 
% assume 0-balance, because we are interested in Lagrangian \acp{fe}. This choice leads 
% to a correct (parallel) \ac{fe} solver, although weaker 1-balance would also be enough, 
% see~\cite[Proposition 3.7]{Badia2019b}.}}
constraints are \emph{single-level} or \emph{direct}, i.e.~hanging \acp{dof} are 
not constrained by other hanging \acp{dof}~\cite[Proposition 3.6]{Badia2019b}. 
Furthermore, in a distributed-memory environment, any hanging \ac{dof} constraint 
can be locally applied, as each subdomain holds a single layer of ghost 
cells~\cite[Proposition 4.1]{Badia2019b}. Although the exposition from 
Sections~\ref{sec:std-fe-spaces} to~\ref{sec:dm-impl} assumes the mesh is a 2:1 
balanced forest-of-trees mesh (with isotropic refinements), all concepts introduced 
there can be generalised to other families of non-conforming meshes, such as 
anisotropic \emph{solvable} meshes~\cite{Cerveny2019}.

% Although the exposition from 
% Sections~\ref{sec:std-fe-spaces} to~\ref{sec:dm-impl} assumes the mesh is a 2:1 
% balanced forest-of-trees mesh (with isotropic refinements), all concepts introduced 
% there can be generalised to other families of non-conforming meshes. However, in 
% order to accommodate conforming \ac{fe} spaces, they must fulfil two necessary, 
% but not sufficient, conditions: (1) all hanging-to-owner \ac{vef} relations meet
% Assumption~\ref{ass:nonconf} and (2) any (chain of) constraints defined on the
% mesh ends with unconstrained \acp{dof}, i.e.~the mesh does not produce cyclic
% hanging \ac{dof} constraint dependencies. Apart from forest-of-trees, other
% meshing approaches, e.g.~anisotropic \emph{solvable} meshes
% in~\cite{Cerveny2019}, fulfil (1) and (2). It is not in our scope to fully
% characterise all possible families of non-conforming meshes that can be
% considered in this work.

We introduce now the immersed boundary setting on top of the artificial domain
$\Omega^\art$. For the sake of simplicity and without loss of generality,
the boundary of the physical domain $\partial \Omega$ is represented by the
zero level-set of a known scalar function $\varphi^\ls$, namely $\partial
\Omega \doteq \{ \x \in \mathbb{R}^d : \varphi^\ls(\x) = 0 \}$. The problem
geometry could be described by other means, e.g.\ from 3D CAD data, by
providing techniques to compute the intersection between cell edges and
surfaces. %(see, e.g.~\cite{marco_exact_2015}). 
In any case, the following exposition does not depend on the way geometry is handled.

Let now the physical domain be defined as the set of points where the level-set
function is negative, namely $\Omega \doteq \{ \x \in \mathbb{R}^d :
\varphi^\ls(\x) < 0 \}$. For any cell $\cell \in \T_h$, let us also define the
quantity $\eta_\cell \doteq |\cell \cap \Omega|/|\cell|$, where $|\cdot|$
denotes the measure (area or volume), and a user-defined parameter $\eta_0 \in
(0,1]$. In order to isolate badly cut cells, we classify cells of $\T_h$ in terms
of $\eta_\cell$ and $\eta_0$. A cell $\cell \in \T_h$ is:
(1) \emph{well-posed}, if $\eta_\cell \geq \eta_0$;
(2) \emph{ill-posed}, if $\eta_0 > \eta_\cell > 0$; or
(3) \emph{exterior}, if $\eta_\cell = 0$, i.e.~$\cell \cap \Omega = \emptyset$, \myadded{see 
Figure~\ref{fig:immersed-setup-b}}. We remark that, for $\eta_0 = 1$, well-posed cells 
coincide with interior cells $\cell \subset \Omega$, whereas ill-posed ones are cut. 
In the general case, $\eta_0 \neq 1$, well-posed cells can also be cut cells with a 
large enough portion inside the physical domain; the distinction between interior and 
cut cells is no longer relevant. The set of well-posed (resp.~ill-posed and exterior) 
cells is represented with $\T_h^\wp$ and its union $\Omega_\wp = \bigcup_{\cell \in 
\T_h^\wp} \overline{\cell} \subset \Omega$ (resp.~$(\T_h^\ip, \Omega_\ip)$ and 
$(\T_h^\out, \Omega_\out )$). We also have that $\{ \T_h^\wp, \T_h^\ip, \T_h^\out \}$ 
is a partition of $\T_h$. We let $\T_h^\act \doteq \T_h^\wp \cup \T_h^\ip$ and 
$\Omega^\act \doteq \Omega_\wp \cup \Omega_\ip$ denote the so-called \emph{active}
triangulation and domain.

\subsection{Cell aggregation}\label{sec:cell-aggr}

Ag\ac{fe} spaces are grounded on a cell aggregation map that assigns a well-posed 
cell to every ill-posed cell. We refer to this map as the \emph{root cell map} $R : 
\T_h \to \T_h^\wp$; it takes any cell $\cell \in \T_h$ and returns a cell $R(\cell) 
\in \T_h^\wp$, referred to as the \emph{root} cell. In order to define this map, we 
consider a partition of $\T_h$, denoted by $\T_h^\ag$, into non-overlapping cell 
aggregates $A_\cell$. Each aggregate $A_\cell$ is a connected set, composed of 
several ill-posed cells and \emph{only} one well-posed root cell $\cell$. Aggregates 
forming $\T_h^\ag$ are built with \myadded{a cell aggregation scheme~\cite{Badia2018} 
described in Figure~\ref{fig:aggr-steps}. 

The scheme builds the aggregates incrementally from the (well-posed) root cells,
by attaching facet-connected ill-posed cells to them, until all ill-posed cells
are aggregated. We recall that facets refer to edges in 2D or faces in 3D. For
non-conforming meshes, facet connections comprise those among cells of same or
different size. Frequently, an ill-posed cell is facet-connected to several
aggregates. Therefore, a criterion is needed to choose among the aggregating
candidates. Previous work on uniform meshes~\cite{Badia2018} adopt a rule that
minimises the distance between ill-posed and root cell barycentres. In this way,
we keep the characteristic length of the aggregates as small as possible to
improve Ag\ac{fem}'s accuracy . Here, in order to consider the effect of the
different cell sizes, it is more adequate to minimise the relative distance
between ill-posed and root cell nodes:}

\begin{definition}[Closest root cell criterion]\label{def:closest_root}\
  Given an ill-posed cell $\cell \in \T_h$ and the set of aggregating candidates
  \[ \mathcal{L}(\cell) = \{ \cell' \in \T_h \, : \, \cell' \text{ is already
  aggregated  and } \exists \text{ a facet } F \in \mathcal{F}_\cell \text{ or } F
  \in \mathcal{F}_{\cell'} \text{ with } \closure{F} = \closure{\cell} \cap
  \closure{\cell'}, \ F \cap \Omega \neq \emptyset \}, \] that is, $\mathcal{L}$
  is the set of aggregated cells connected to $\cell$ through a conforming or
  hanging facet $F$. The \emph{closest} aggregating candidate $\cell^*$
  satisfies\[
		\tilde{d}(\cell,\cell^*) = \min_{\cell' \in
		\mathcal{L}(\cell)} \tilde{d}(\cell,\cell')
	\]with\[
		\tilde{d}(\cell,\cell') \doteq \dfrac{\max_{\gamma \in
		\mathcal{F}_{R(\cell')}^0, \delta \in \mathcal{F}_{\cell}^0}
		\| \x^\gamma - \x^\delta \|_\infty}{\max_{\gamma,\gamma'
		 \in
		\mathcal{F}_{R(\cell')}^0}
		\| \x^\gamma - \x^{\gamma'} \|_\infty},
	\] for any $\cell' \in\mathcal{L}(\cell)$, where $\mathcal{F}_{\cell}^{0}$
	denotes the set of vertices of $\cell$, $R(\cell)$ the root of $\cell$,
	$\x^{\square}$ the coordinates of vertex $\square$ and $\| \cdot \|_{\infty}$
	denotes the infinity norm.
\end{definition}

\myadded{When there is more than one closest aggregating candidate $\cell^*$, we
simply choose the one whose root cell has higher global cell index.} The output of
the cell aggregation scheme is the root cell map $R$ and it can be readily applied
to arbitrary spatial dimensions. We observe that, by construction of the scheme,
maximum aggregate size is bounded above by a constant times the maximum cell size
in the mesh~\cite{Badia2018}. \myadded{Moreover, in order to assure that
aggregates are always connected sets, we assume $\T_{h}$ is defined, such that
$\cell \cap \Omega$ is connected, for any $\cell \in \T_{h}$. Connected aggregates
are convenient for the numerical analysis in Appendix~\ref{appendix:proofs}, as
they allow one to use the Deny-Lions lemma to prove approximability properties.}

% For general expressions of the immersed boundary $\Gamma$, it could happen that
% a cell $\cell \in \T_{h}$ is such that $\cell \cap \Omega$ is a set of connected
% domains that are disconnected with each other. In this case, we consider each of
% these disconnected components as a separate cut cell. Thus, we redefine the mesh
% $\T_{h}$, replicating these cut cells for each disconnected part, and use the
% cell aggregation procedure above verbatim. The reason for this is to assure that
% aggregates are always connected domains and, e.g.~the Deny-Lions lemma can be
% used to prove approximability properties in Appendix~\ref{appendix:proofs}.
% Alternatively, one could assume that $\Gamma$ has bounded curvature and the mesh
% is \emph{fine enough} \myadded{such that, for any $\cell \in \T_{h}$, $\cell
% \cap \Omega$ is connected} (see, e.g.~\cite{hansbo2002unfitted}).

\begin{figure}[ht!]
  \centering
  \begin{subfigure}{0.3\textwidth}
    \centering
    \includegraphics[width=0.95\textwidth]{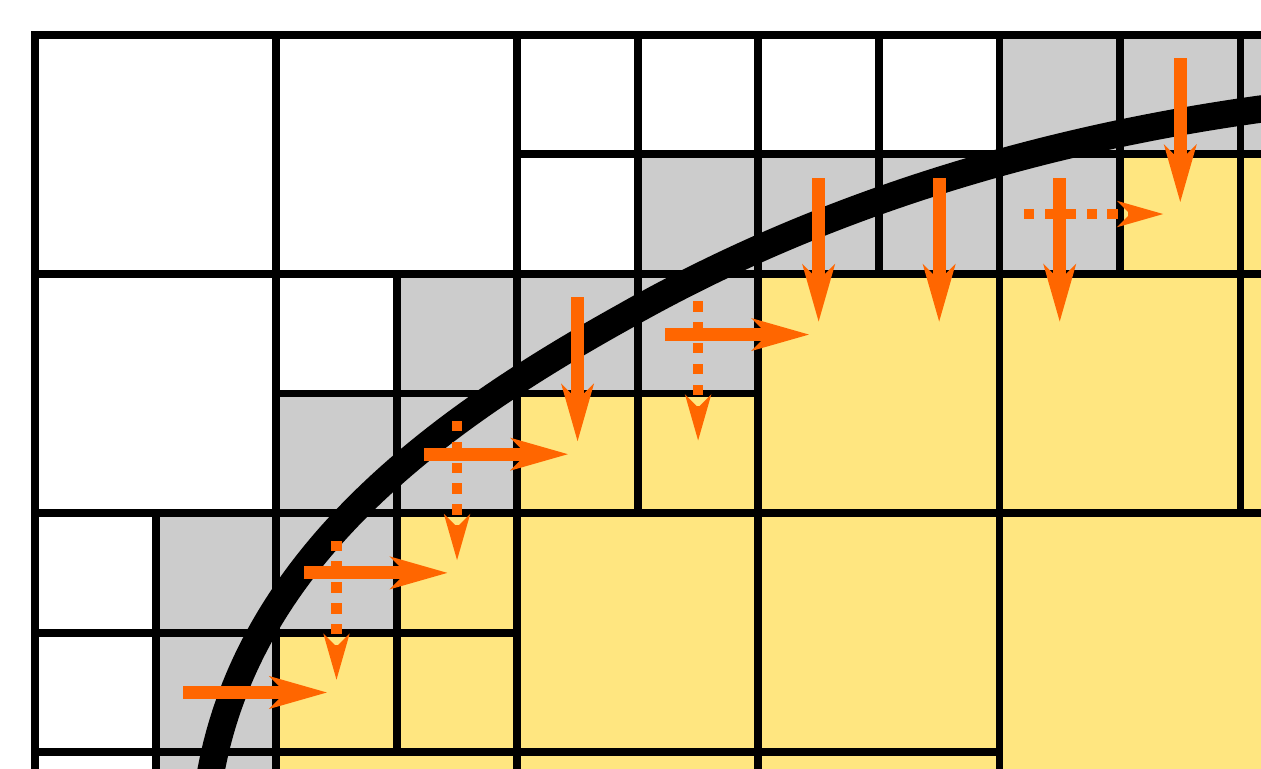}
    \caption{}
    \label{fig:aggr-steps-a}
  \end{subfigure}
  \begin{subfigure}{0.3\textwidth}
    \centering
    \includegraphics[width=0.95\textwidth]{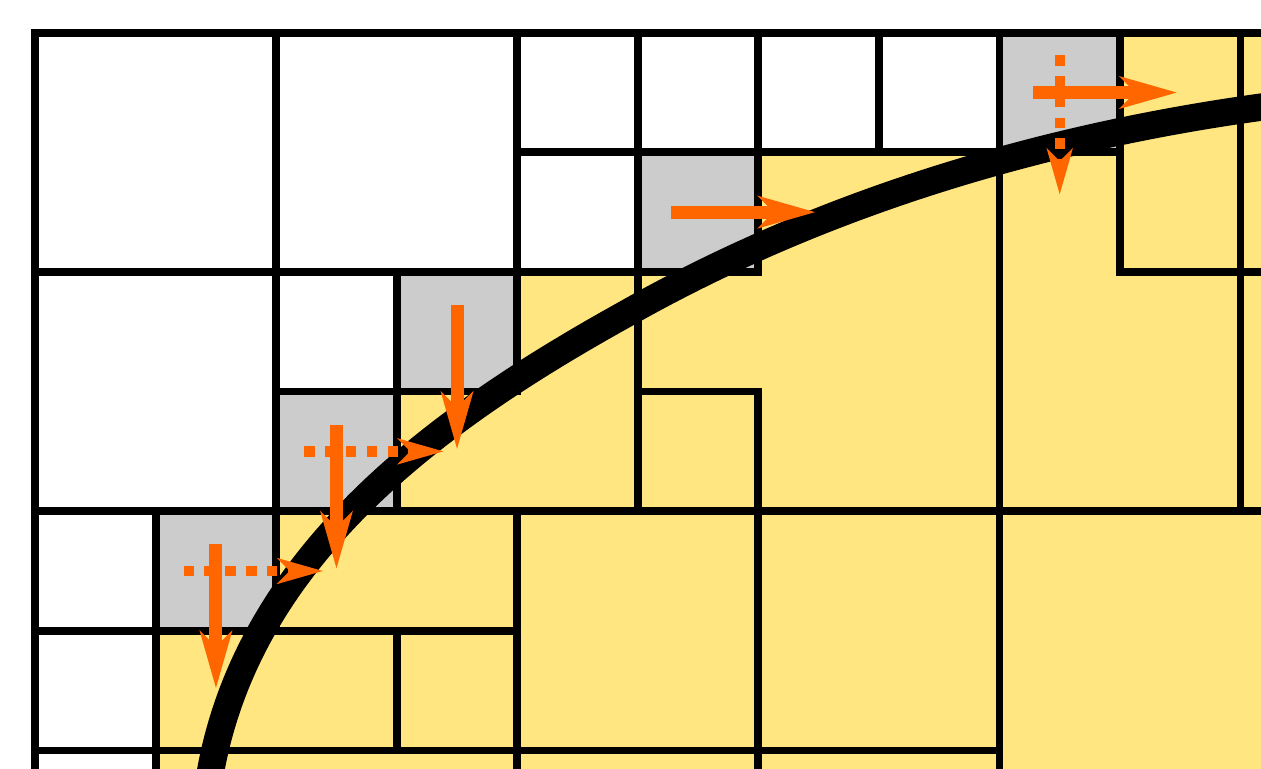}
    \caption{}
    \label{fig:aggr-steps-b}
  \end{subfigure}
  \begin{subfigure}{0.3\textwidth}
    \centering
    \includegraphics[width=0.95\textwidth]{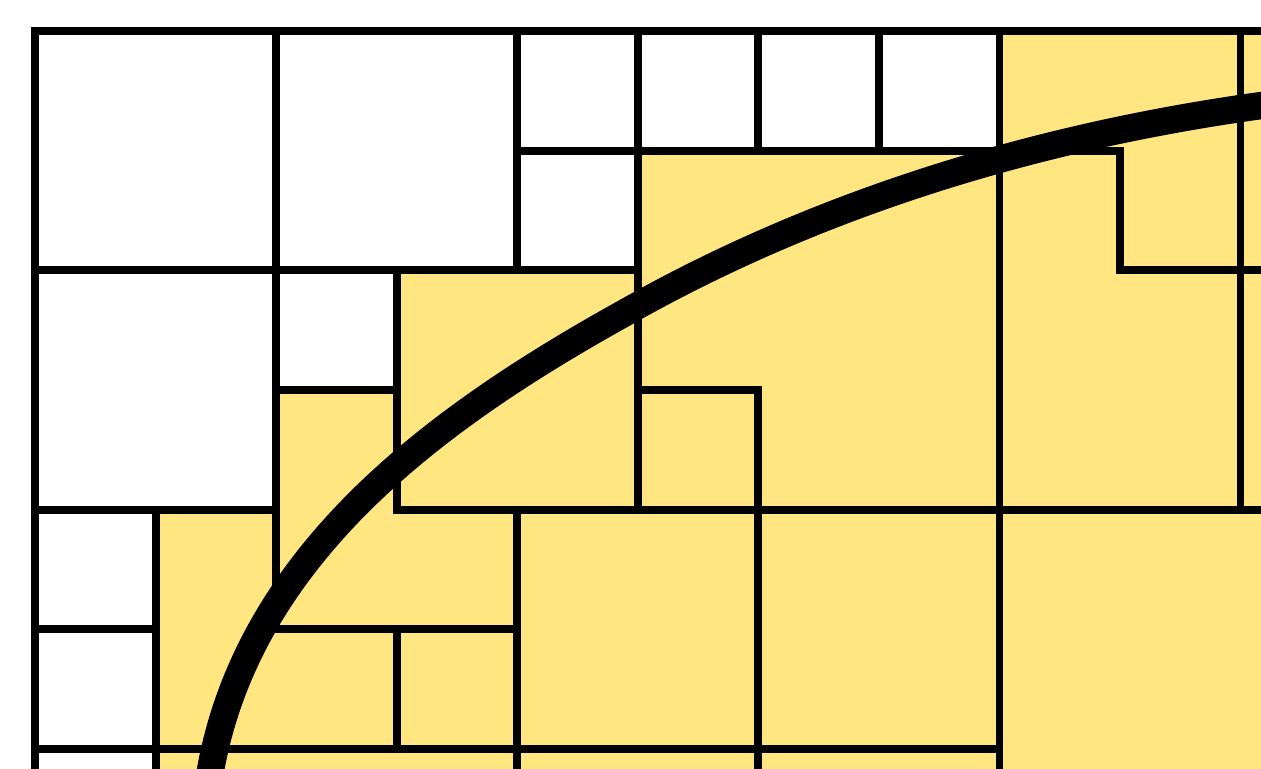}
    \caption{}
    \label{fig:aggr-steps-c}
  \end{subfigure}
  \caption{\myadded{Close-up to the top left corner of Figure~\ref{fig:immersed-setup-b} 
  describing the cell aggregation scheme~\cite{Badia2018} in three steps. The
  initial aggregates are well-posed cells; their root cells are assigned to be
  themselves. Next, we incrementally attach ill-posed cells. An ill-posed cell,
  when facet-connected to an aggregate, is attached to the closest root cell, in
  the sense of Definition~\ref{def:closest_root}. Arrows in (\textsc{a}) and
  (\textsc{b}) point ill-posed cells to all possible candidates; selected
  candidates are pointed by continuous arrows, non-selected with discontinuous
  ones. The black thin lines represent the boundaries of the aggregates. From one
  step to the next one, some of the lines between adjacent cells are removed. This
  means that the adjacent cells have been merged into the same aggregate. The
  procedure leads to $\T_h^\ag$, represented in (\textsc{c}).}}
  \label{fig:aggr-steps}
\end{figure}

\subsection{Standard Lagrangian conforming finite element
spaces}\label{sec:std-fe-spaces}

\myadded{Our aim now is to present our notation to describe conventional
\emph{conforming} \ac{fe} spaces on top of tree-based meshes; they are referred to
as \emph{standard} or \emph{std}, in contrast to the \emph{aggregated} spaces
presented later in Section~\ref{sec:ag-fe-spaces}. We aim at solving a \ac{pde}
problem in the physical domain $\Omega$, subject to boundary conditions on
$\partial \Omega$. We assume Dirichlet conditions on $\Gamma_\mathrm{D} \subset
\partial \Omega$. For unfitted meshes, it is not obvious to impose Dirichlet
conditions in the approximation space in a strong manner. In consequence, we will
assume \emph{weak} imposition of Dirichlet boundary conditions on
$\Gamma_\mathrm{D}$.

Our starting point is the typical \ac{cg} \ac{fe} space, denoted by $\V_h^\ncf$,
in which we enforce continuity across conforming \acp{vef}, i.e.~those meeting
Assumption~\ref{ass:nonconf} (i). As usual in \ac{fem}, $\V_h^\ncf$ is grounded on
defining cell-wise functional spaces $\mathcal{V}(\cell)$, a canonical basis for a
set of local \acp{dof}, a geometrical ownership of the local \acp{dof} by the cell
\acp{vef} and a local-to-global map to glue together local \acp{dof} that lie in
the same geometrical position. For the sake of simplicity and without loss of
generality, we assume the local spaces $\mathcal{V}(\cell)$ are scalar-valued
Lagrangian \acp{fe}, of the same order $q$ everywhere. Extension to vector-valued
or tensor-valued Lagrangian \acp{fe} is straightforward; it suffices to apply the
same approach component by component. We denote by $\Sigma$ the set of global
\acp{dof} in $\T_h^\act$ associated to $\V_h^\ncf$.

% Our starting point is the \ac{fe} space that takes the form
% \[
%   \V_h^\ncf \doteq \{ v|_T \in \V(\cell) \text{ for any } \cell \in \T_h^\act 
%   \, : \, \restrict{\jump{v}}{f} = 0 \text{ for any } f \in \tilde{\mathcal{F}} \},
% \]
% where $\tilde{\mathcal{F}}$ represents the set of \acp{vef} meeting
% Assumption~\ref{ass:nonconf} (i). Thus, $\V_h^\ncf$ has continuous traces across
% pairs of cells, whose touching interface \acp{vef} are identical. On the other
% hand, $\mathcal{V}(\cell)$ is a suitable vector space of functions defined on
% $\cell \in \T_h^\act$. For the sake of simplicity and without loss of generality,
% we assume the local spaces $\mathcal{V}(\cell)$ are scalar-valued Lagrangian
% \acp{fe}, of the same order $q$ everywhere. Extension to vector-valued or
% tensor-valued Lagrangian \acp{fe} is straightforward; it suffices to apply the
% same approach component by component. We denote by $\Sigma$ the set of global
% \acp{dof} in $\T_h^\act$ associated to $\V_h^\ncf$.

When $\T_h^\act$ is non-conforming, it is clear that $\V_h^\ncf$ yields
discontinuous approximations across hanging \acp{vef}. Therefore, the resulting
\ac{fe} space is non-conforming (i.e.~it is not a subspace of its
infinite-dimensional counterpart) and, thus, not suitable for \ac{cg} methods. To
recover global (trace) continuous \ac{fe} approximations, values of \acp{dof}
lying \emph{only} on hanging \acp{vef} cannot be arbitrary, they must be linearly
constrained. In practice, this means to restrict $\V_h^\ncf$ into a conforming
\ac{fe} subspace $\V_h^\std$.

In order to introduce $\V_h^\std$, let $\Sigma \doteq \{\Sigma^\F, \Sigma^\H \}$
denote a partition into \emph{free} and \emph{hanging} \acp{dof}; the latter
refers to the subset of global \acp{dof} lying \emph{only} on hanging \acp{vef}.
We let now $\M^{\H}_\sigma$ denote the subset of \acp{dof} constraining $\sigma 
\in \Sigma^\H$, referred to as the set of \emph{master} \acp{dof} of $\sigma$. 
Recalling Assumption~\ref{ass:nonconf} (ii), we observe that, given $\sigma \in 
\Sigma^\H$, lying on a hanging \ac{vef} $f$ of a cell $\cell$, its constraining 
\acp{dof} are located in the closure of their owner \ac{vef} $\overline{f'}$ of 
a coarser cell $\cell'$~\cite[Proposition 3.6]{Badia2019b}. Setup and resolution
of hanging \ac{dof} constraints for Lagrangian \ac{fe} spaces is
well-established knowledge~\cite{Rheinboldt1980} and, for conciseness, not
reproduced here.

Finally, we introduce the \emph{standard} conforming \ac{fe} space. Given $v_h =
\sum_{\sigma \in \Sigma} v^\sigma_h \phi^\sigma \in \V_h^\ncf$, we let
\begin{equation}\label{eq:std-fe-space}
	\V_h^\std \doteq \{ v_h \in \V_h^\ncf : v^\sigma_h = \sum_{\sigma' \in
	\M^{\H}_\sigma} C^\H_{\sigma\sigma'} v_h^{\sigma'} \ \text{for any} \ \sigma
	\in \Sigma^\H \},
\end{equation}
where $C^\H_{\sigma\sigma'} = \phi^{\sigma'} (\x^{\sigma})$ and $\phi^{\sigma'}$
is the global shape function of $\V_h^\ncf$ associated with $\sigma'$. Note that
$C^\H_{\sigma\sigma'} \neq 0$, by definition of $\sigma \in \Sigma^\H$ and
$\M^{\H}_\sigma$. We observe that $\V_h^\std \subset \V_h^\ncf$ is conforming, in
particular, $v_h \in {\mathcal{C}^0}(\Omega^\act)$.}

\subsection{Aggregated Lagrangian finite element spaces}\label{sec:ag-fe-spaces}

\myadded{The space $\V_h^\std$, introduced in Section~\ref{sec:std-fe-spaces}, is
conforming, but leads to arbitrarily ill-conditioned systems of linear algebraic
equations, unless an extra technique is used to remedy it. This is the main
motivation to introduce Ag\ac{fe} spaces (see, e.g.~\cite{Badia2018,Badia2018a}).
The main idea is to remove from $\V_h^\std$ ill-posed \acp{dof}, associated with
small cut cells, by constraining them as a linear combination of \acp{dof} with
local support in a well-posed cell. For this purpose, we assign each ill-posed
\ac{dof} to a well-posed cell, via the root cell map $R$. Following this, we
extrapolate the value at the ill-posed \ac{dof}, in terms of the \ac{dof} values
at the root cell. Thus, the problem is posed in terms of well-posed \acp{dof} only, 
recovering the ill-posed \acp{dof} with a discrete extension operator.

In this section, we derive an Ag\ac{fe} space $\V_h^\ag$ as a subspace of
$\V_h^\std$. We focus on laying out the key aspect to combine the new linear
constraints, arising from ill-posed \ac{dof} removal, with those already
restricting $\V_h^\std$, to enforce conformity. For the sake of completeness, we
refer to Appendix~\ref{appendix:agfespace} for a rigorous description of the extension operator with
combined constraints and the proof that $\V_h^\ag$ is well-defined, e.g.\ it
does not have cycling constraint dependencies. Besides, we demonstrate, in
Appendix~\ref{appendix:proofs}, (cut-independent) well-posedness and condition
number estimates of the linear system arising from this Ag\ac{fe} method on the
Poisson problem defined in \eqref{eq:PoissonEq}.

In order to construct $\V_h^\ag$, we start by recalling the partition of \acp{dof}
associated to $\V_h^\std$ into \emph{free} \acp{dof} and \emph{hanging}
constrained \acp{dof}. The first and most crucial step is to further distinguish,
on top of this partition, among \emph{well-posed} and \emph{ill-posed} \acp{dof}.
To this end, we must define sets of the form $\Sigma^{\X,\Y}$, where $\X \in
\{\wp,\ip\}$ refers to well-posed or ill-posed and $\Y \in \{\F,\H\}$ refers to
free or hanging. We refer to Figure~\ref{fig:spaces} for an illustration of this
classification. The key to combine the constraints is to define well-posed free
\acp{dof} $\Sigma^{\wp,\F}$ as those with local support in (at least one) a
well-posed cell. Specifically, we let $\Sigma^{\wp,\H} \subset \Sigma^\H$ denote
the set of hanging \acp{dof} that are located in $\T_h^\wp$, i.e.~they are a local
\ac{dof} of (at least one) well-posed cell. Then, $\Sigma^{\wp,\F} \subset
\Sigma^\F$ is defined as follows.
\begin{definition}  \label{def:free-well-posed}
  Given $\sigma \in \Sigma^\F$, then $\sigma \in \Sigma^{\wp,\F}$ is a
  \emph{well-posed free \ac{dof}}, if and only if it meets one of the following:
  (i) $\sigma$ is located in $\T_h^\wp$ or (ii) $\sigma$ does not meet
  condition~(i), but $\sigma \in \M_{\sigma'}^\H$ for some $\sigma' \in
  \Sigma^{\wp,\H}$, i.e.~$\sigma$ is outside $\T_h^\wp$, but constrains a
  well-posed hanging \ac{dof} $\sigma'$.
\end{definition}

\begin{figure}[ht!]
  \centering
  \begin{subfigure}{\textwidth}
    \vspace{0.1cm}
    \centering
    {\Large \color[RGB]{255,230,128} $\blacksquare$} $\Omega$ \enskip
    {\Large \color[RGB]{204,204,204} $\blacksquare$} $\Omega^\art \setminus \Omega$ \qquad
    {\Large \color{blue} $\bullet$} $\Sigma^{\wp,\F}$ \enskip
    {\Large \color[RGB]{255,0,255} $\bullet$} $\Sigma^{\wp,\H}$ \enskip
    {\color{blue} $\boldsymbol{\times}$} $\Sigma^{\ip,\F}$ \enskip
    {\color[RGB]{255,0,255} $\boldsymbol{\times}$} $\Sigma^{\ip,\H}$
	\end{subfigure} \vspace{0.01cm}\\
	\begin{subfigure}{0.48\textwidth}
    \includegraphics[width=0.95\textwidth]{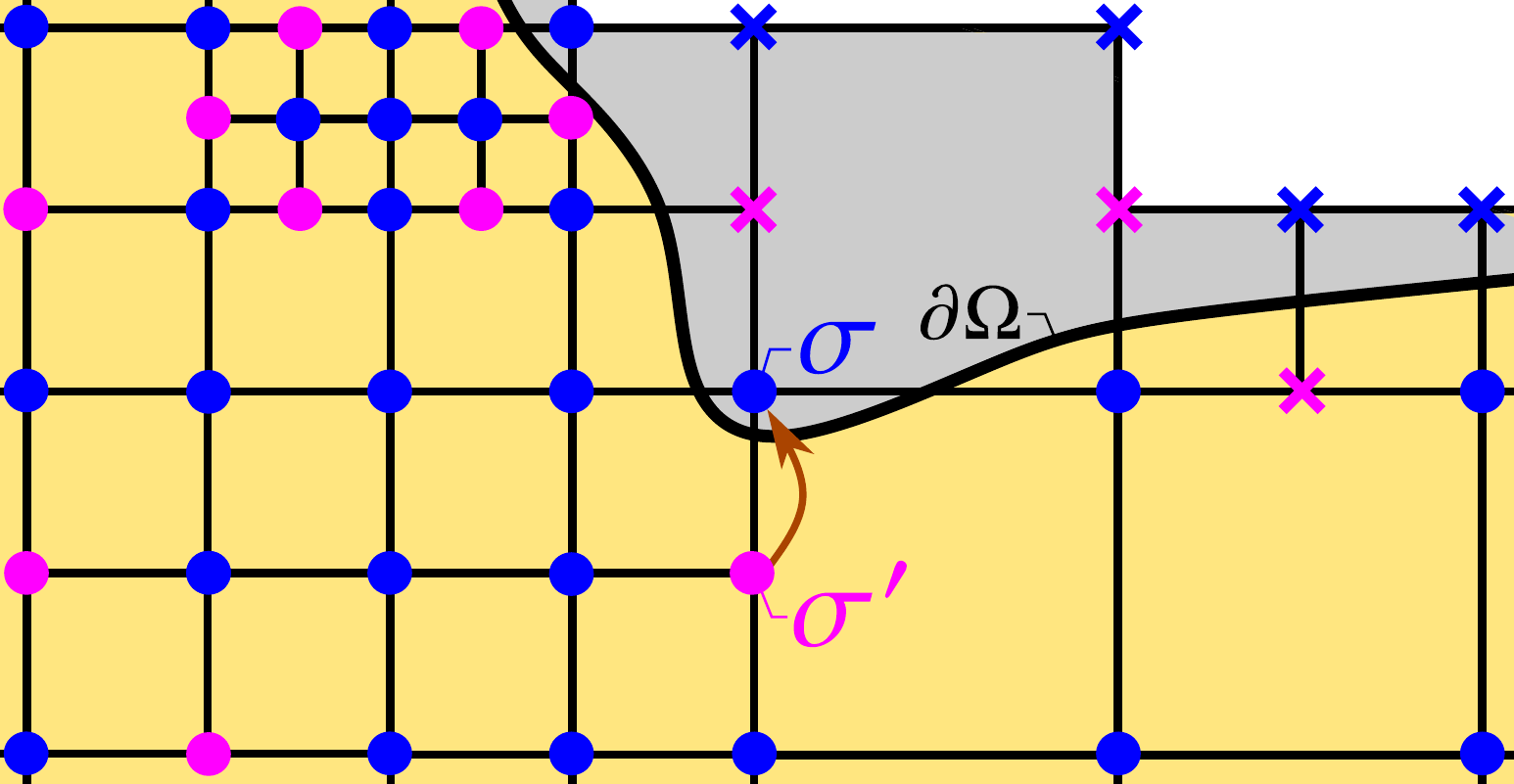}
    \caption{\myadded{Correct partition of $\Sigma$ for $\V_h^\ag$. Free \ac{dof}
    $\sigma$ is marked as well-posed, even though it is surrounded by ill-posed
    cells, because it constrains the well-posed hanging \ac{dof} $\sigma'$. Hence,
    it is well-posed due to Definition~\ref{def:free-well-posed} (ii).}}
    \label{fig:spaces_a}
  \end{subfigure} \quad
	\begin{subfigure}{0.48\textwidth}
    \includegraphics[width=0.95\textwidth]{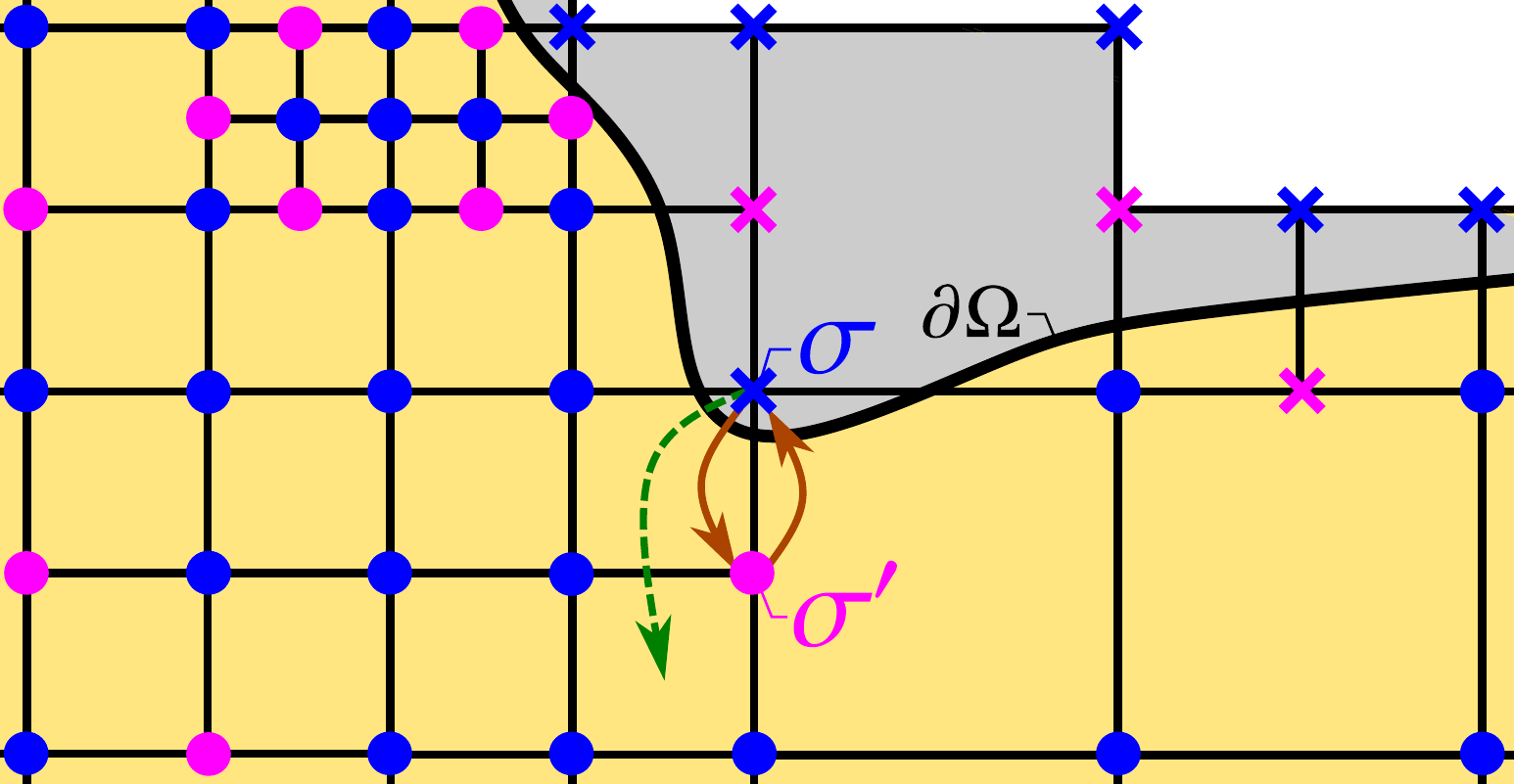}
    \caption{\myadded{Incorrect partition of $\Sigma$ for $\V_h^\ag$. Free \ac{dof}
    $\sigma$ is marked as ill-posed. If its root cell is pointed by the dashed
    arrow, then $\sigma'$ constrains $\sigma$. In parallel, $\sigma'$ is a hanging
    \ac{dof} constrained by $\sigma$. Thus, we have an unsolvable circular
    constraint dependency.}}
    \label{fig:spaces_b}
  \end{subfigure}
	\caption{\myadded{Classification of $\Sigma$ into $\{ \Sigma^{\wp,\F},
  \Sigma^{\wp,\H}, \Sigma^{\ip,\F}, \Sigma^{\ip,\H} \}$ on a portion of a mesh
  where $\eta_0 = 1$, i.e.~well/ill-posed cell iff interior/cut cell. The key to
  combine hanging and aggregation \ac{dof} constraints is to mark \acp{dof}
  meeting Definition~\ref{def:free-well-posed} (ii) as well-posed, as in
  (\textsc{a}). In this way, we circumvent any possible circular constraint
  dependencies, such as the one described in (\textsc{b}).}}
	\label{fig:spaces}
\end{figure}

This definition eliminates any situation with circular constraint dependencies, as
shown in Figure~\ref{fig:spaces_b} and detailed in
Appendix~\ref{appendix:agfespace}. Furthermore, it is backed by the numerical
analysis in Appendix~\ref{appendix:proofs}. We observe that $\Sigma^{\wp,\F}$
includes free \acp{dof} surrounded by ill-posed cells that constrain well-posed
hanging \acp{dof}, see Figure~\ref{fig:spaces_a}. If we let $\Sigma^{\ip,\Y}
\doteq \Sigma^{\Y} \setminus \Sigma^{\wp,\Y}$, then it becomes clear that
$\{\Sigma^{\wp,\F}, \Sigma^{\wp,\H}, \Sigma^{\ip,\F}, \Sigma^{\ip,\H} \}$ is a
partition of $\Sigma$. In contrast to free \acp{dof} in $\Sigma^{\wp,\F}$, any
$\sigma \in \Sigma^{\ip,\F}$ is liable to have arbitrarily small local support
and, following the Ag\ac{fem} rationale, must be constrained by \acp{dof} in
$\Sigma^{\wp,\F}$. It follows that, in the Ag\ac{fe} space, free \acp{dof} are
reduced to free well-posed \acp{dof}, i.e.~$\Sigma^{\wp,\F}$, whereas constrained
\acp{dof} are $\Sigma^\C \doteq \{\Sigma^{\wp,\H}, \Sigma^{\ip,\F},\Sigma^{\ip,\H}
\}$. In Appendix~\ref{appendix:agfespace} we show that any $\sigma \in \Sigma^\C$
can be resolved with \emph{direct} constraints, i.e.~linear constraints of the
same form as those in \eqref{eq:std-fe-space}, in terms of well-posed free
\acp{dof}, only. As a result, the (sequential) \emph{aggregated} or
\emph{ag.}~\ac{fe} space can be readily defined as
\begin{equation}\label{eq:ag-fe-space}
	\V_h^\ag \doteq \{ v_h \in \V_h^\ncf : v^\sigma_h = \sum_{\sigma' \in
	\M_\sigma} C_{\sigma\sigma'} v_h^{\sigma'} \ \text{for any} \ \sigma
	\in \Sigma^\C \},
\end{equation}
where $\M_\sigma$ is the set of \acp{dof} constraining $\sigma \in \Sigma^\C$ and
$C_{\sigma\sigma'}$ is the constraining coefficient for $\sigma' \in \M_\sigma$;
we refer to~\eqref{eq:ag-fe-masters} and~\eqref{eq:ag-fe-constraints} for their
respective full expressions. It is clear that $\V_h^\ag \subset \V_h^\std \subset
\V_h^\ncf$. For the sake of brevity, further aspects, such as the definition of
the resulting shape basis functions or finite element assembly operations are not
covered. In the end, constraints supplementing $\V_h^\ag$ are of multipoint linear
type, in the same way as those of $\V_h^\std$; they have been extensively covered
in the literature, see, e.g.~\cite{Verdugo2019,Shephard1984}. With regards to the
implementation, we remark that the set up of $\V_h^\ag$ can also potentially reuse
data structures and methods devoted to the construction of $\V_h^\std$ or, more
generally, any other \ac{fe} space endowed with linear algebraic constraints.}

\subsection{Distributed-memory extension}\label{sec:dm-impl}

After defining Ag\ac{fem} in a serial context, we \myadded{briefly} discuss its
extension to a \ac{dd} setup for implementation in a distributed-memory \myadded{computer}. We start by setting up the partition of the mesh into subdomains: Let
$\S$ be a partition of $\Omega^\art$ into subdomains obtained by the union of
cells in the background mesh $\T_h$, i.e.~for each cell $\cell \in \T_h$, there is
a subdomain $S \in \S$ such that $T \subset S$. We denote by $\T_h^{\L(S)}$ the
set of \emph{local} cells in subdomain $S \in \S$; naturally, $\{\T_h^{\L(S)}\}_{S
\in \S}$ forms a partition of $\T_h$, see~Figure~\ref{fig:dist-spaces}. We assume
that $\S$ is easy to generate. This is a reasonable assumption in our embedded
boundary context, where $\Omega^\art$ can be easily meshed with e.g.~tree-based
Cartesian grids, which are amenable to load-balanced partitions grounded on
space-filling curves~\cite{burstedde_p4est_2011}.

\begin{figure}[ht!]
  \centering
  \begin{subfigure}{\textwidth}
    \vspace{0.1cm}
    \centering
    \includegraphics[height=0.25cm]{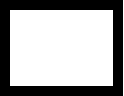} $\T_h^{\L(S_i)}$ \enskip
    \includegraphics[height=0.25cm]{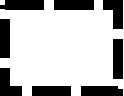} $\T_h^{\TG(S_i)}$ \enskip
    {\color[RGB]{255,230,128} \rule{10pt}{8pt}} $S_1$ \enskip
    {\color[RGB]{176,244,171} \rule{10pt}{8pt}} $S_2$ \enskip
    {\Large \color{blue} $\bullet$} $\Sigma_\L^{\wp,\F}$ \enskip
    {\Large \color[RGB]{255,0,255} $\bullet$} $\Sigma_\L^{\wp,\H}$ \enskip
    {\color{blue} $\boldsymbol{\times}$} $\Sigma_\L^{\ip,\F}$ \enskip
    {\color[RGB]{255,0,255} $\boldsymbol{\times}$} $\Sigma_\L^{\ip,\H}$
	\end{subfigure} \vspace{0.01cm}\\
	\begin{subfigure}{0.48\textwidth}
    \includegraphics[width=0.95\textwidth]{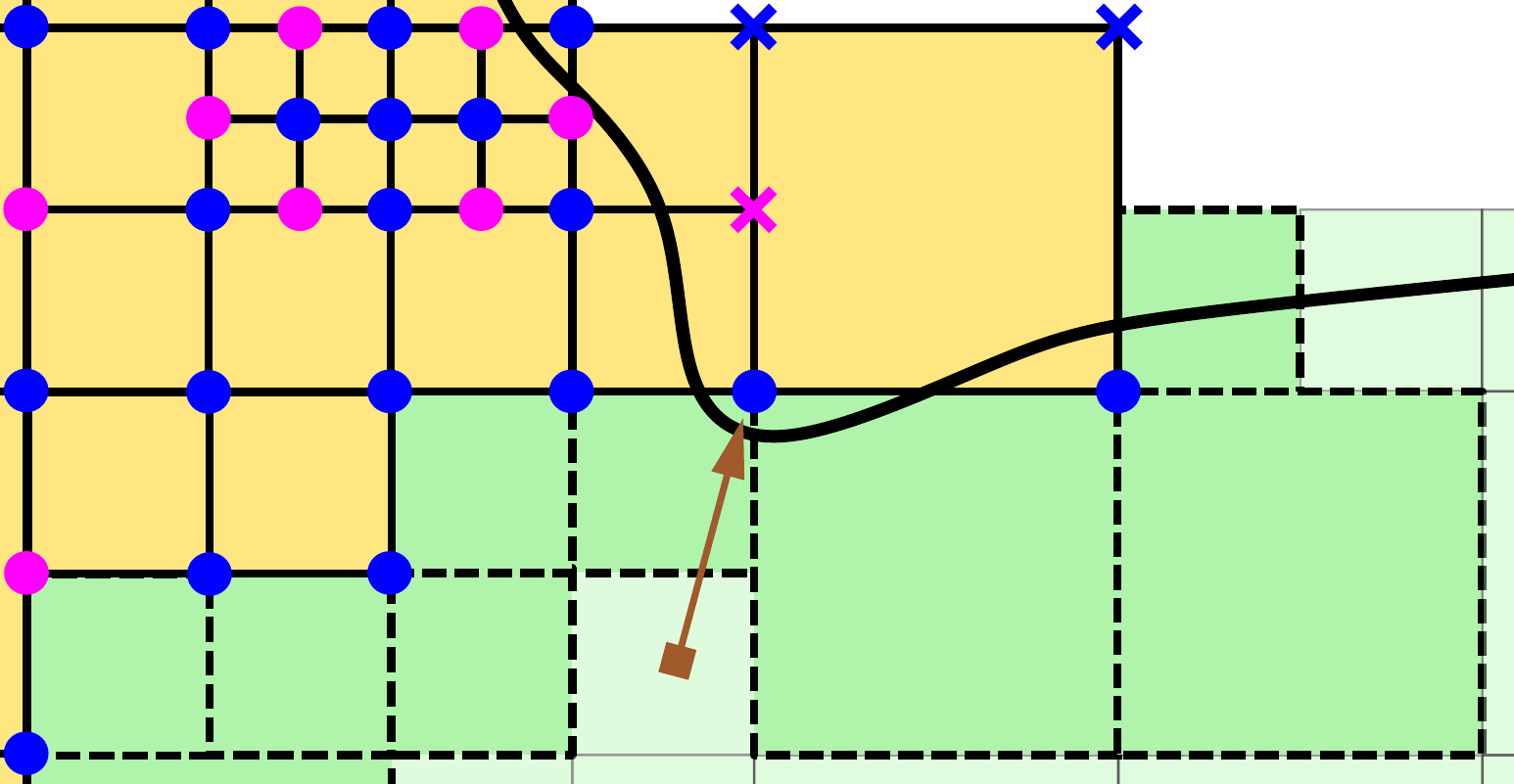}
    \caption{Partial view of $\T_h^{\L(S_1)} \cup \T_h^{\TG(S_1)}$.}
    \label{fig:dist-spaces-a}
  \end{subfigure} \quad
	\begin{subfigure}{0.48\textwidth}
    \includegraphics[width=0.95\textwidth]{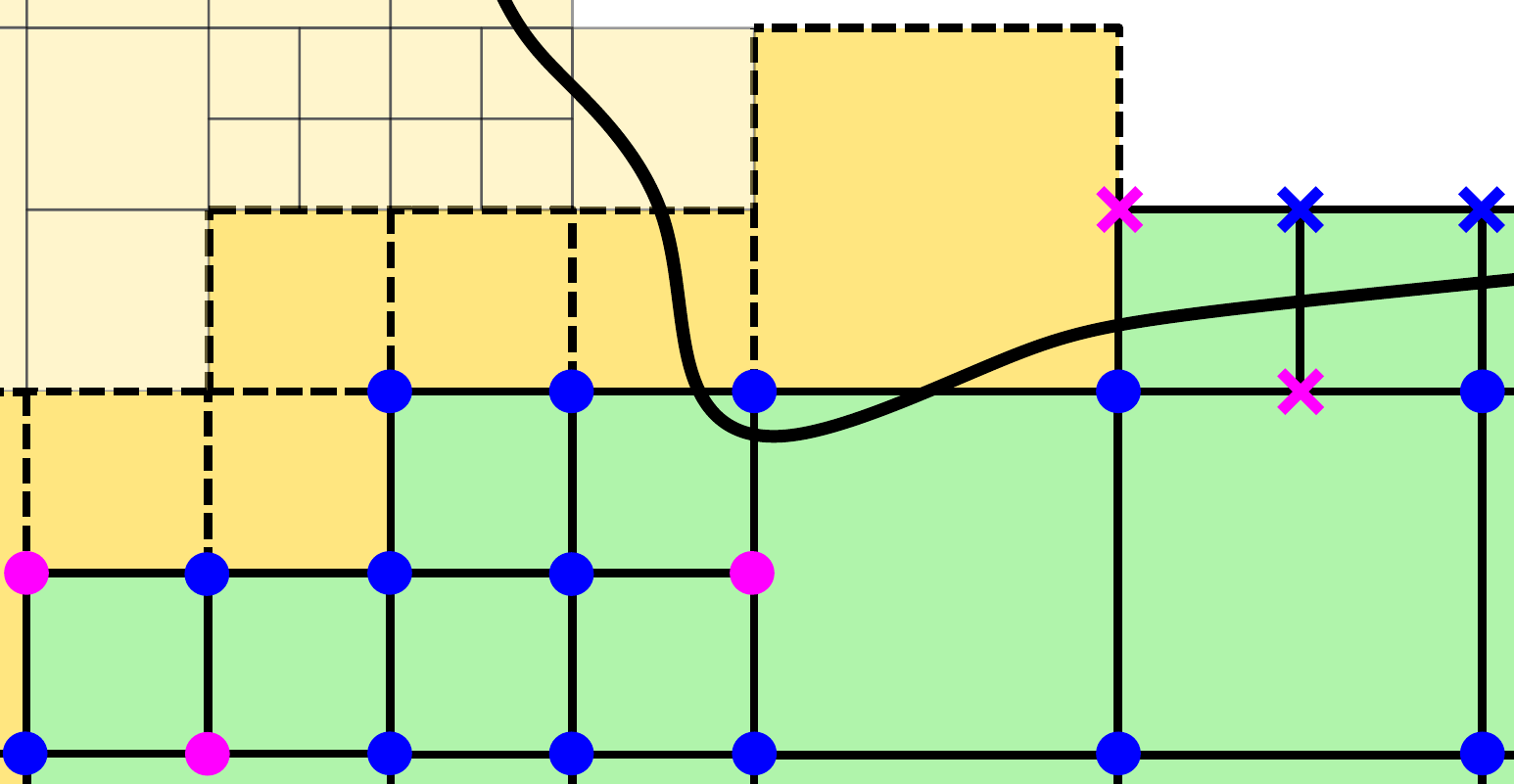}
    \caption{Partial view of $\T_h^{\L(S_2)} \cup \T_h^{\TG(S_2)}$.}
    \label{fig:dist-spaces-b}
  \end{subfigure}
  \caption{Classification of cells and \emph{local} \acp{dof} assuming a partition
  of the mesh portion in Figure~\ref{fig:spaces} into two subdomains. Light-shaded cells are not actually in the scope of $S_i$, $i
  = 1,2$. The arrow in
  \textsc{(a)} points at a free \ac{dof}, whose well-posed status can only be
  known by nearest-neighbour exchange, see Remark~\ref{rem:wp}; indeed, it has
  local support in a well-posed $S_2$-cell that is not in $\T_h^{\L(S_1)} \cup
  \T_h^{\TG(S_1)}$.}
	\label{fig:dist-spaces}
\end{figure}

In a parallel, distributed-memory environment, each subdomain $S$ is mapped to a
processor. Thus, each processor holds in memory a portion $\T_h^{\L(S)}$ of the
global mesh $\T_h$. Naturally, local \ac{fe} integration in $S$ is restricted to
cells in $\T_h^{\L(S)}$. However, to correctly perform the parallel \ac{fe}
analysis, the processor-local portion of $\T_h$ is usually extended with adjacent
off-processor cells, a.k.a.~\emph{ghost} cells. Ghost cells are essential to
generate the global \ac{dof} numbering, in particular, to glue together \acp{dof}
in processors that represent the same global \ac{dof}. For constrained spaces,
they are also needed to \emph{locally} solve \ac{dof} constraints that expand
beyond $\T_h^{\L(S)}$. Standard practice in large-scale \ac{fe} codes is to
constrain the ghost cell set to a single layer of ghost cells. Here, we refer to
them as the \emph{true} ghosts, given by $\T_h^{\TG(S)} \doteq \{ \cell \in \T_h
\setminus \T_h^{\L(S)} : \ \closure[1]{\cell} \cap \closure[1]{S} \neq \emptyset
\}$. This layer suffices to glue together global \acp{dof} among processors for
non-constrained spaces. However, it is not necessarily enough \myadded{to meet the
requirements} of constrained ones.

\myadded{Our goal in this section is to identify the minimum set of ghost cells that
we must attach to $\T_h^{\L(S)}$ in order to define the $S$-subdomain restriction
of $\V_h^\ag$ into $S$, which leads to the distributed version of $\V_h^\ag$.
Hereafter, all quantities refer to a given subdomain $S$, and we drop the subindex
$S$ unless needed for clarity. We assume our initial distributed-memory setting
considers processors holding $\T_h^\L \cup \T_h^\TG$ locally. Besides, each
processor holds a subdomain restriction of the root cell map $R$, such that it
leads to the same aggregates as the ones obtained with the sequential method;
see~\cite{Verdugo2019} for details on the distributed-memory cell aggregation
scheme. We let now $\Sigma_\L$ denote the set of ($S$-subdomain) local \acp{dof},
i.e.~those located in $\T_h^\L$. As in the sequential version of Ag\ac{fem}, we
have that $\{\Sigma_\L^{\wp,\F}, \Sigma_\L^{\wp,\H}, \Sigma_\L^{\ip,\F},
\Sigma_\L^{\ip,\H} \}$ forms a partition of $\Sigma_\L$, as shown
in~Figure~\ref{fig:dist-spaces}. Hence, $\Sigma_\L^\C \doteq \{\Sigma_\L^{\wp,\H},
\Sigma_\L^{\ip,\F}, \Sigma_\L^{\ip,\H} \}$ is the subset of ($S$-subdomain) local
constrained \acp{dof}.

For the sake of parallel performance and efficiency, we want to design our
parallel algorithms and data structures, concerning the setup of $\V_h^\ag$, in
such a way that they maximise local work, while minimising inter-processor
communication. In our context, this amounts to ensure that, given $\sigma \in
\Sigma_\L^\C$, we can resolve its full constraint dependency locally, in the scope
of the processor. In other words, any constraining \ac{dof} $\sigma' \in
\M_\sigma$ must be found in the processor-local portion of $\T_h$. In order to see
how we can fulfil this requirement with $\V_h^\ag$, we recover first two particular
cases, already addressed in previous literature:}
\begin{enumerate}
  \item \myadded{$\V_h^\ag$ does not have ill-posed \acp{dof}, i.e.~$\V_h^\ag \equiv
  \V_h^\std$: We recall, from Section~\ref{sec:std-fe-spaces}, that constraining
  \acp{dof} of hanging \acp{dof} are located on their coarser cells around. As all
  coarser cells around $\T_h^\L$ are in $\T_h^\L \cup \T_h^\TG$, all constraint
  dependencies of hanging \acp{dof} in $\T_h^\L$ do not expand beyond $\T_h^\L
  \cup \T_h^\TG$, i.e.~all hanging \ac{dof} constraints in $\T_h^\L$ can be
  resolved in $\T_h^\L \cup \T_h^\TG$, see~\cite[Proposition 4.1]{Badia2019b}.}
  \item \myadded{$\V_h^\ag$ does not have hanging \acp{dof}, i.e.~$\V_h^\ag$ is
  defined on a conforming mesh: In general, given $\sigma \in \Sigma_\L^{\ip,\F}$,
  the subdomain, where its root cell is located, is different from the current
  subdomain. In particular, the root cell can be outside $\T_h^\L \cup \T_h^\TG$,
  as in Figure~\ref{fig:remote-root-a}. This means that the constraint dependency
  of $\sigma$ propagates away from $\T_h^\L \cup \T_h^\TG$. In order to cancel the
  constraint associated to $\sigma$, we need to attach the missing root cell to
  $\T_h^\L \cup \T_h^\TG$~\cite{Verdugo2019}. We let $\T_h^\RG$ denote the set of
  all missing \emph{remote} root cells.}
\end{enumerate}

\begin{figure}[ht!]
  \centering
	\begin{subfigure}{0.5\textwidth}
	  \begin{small}
			{\color[RGB]{255,230,128} \rule{10pt}{8pt}} $S_3$\quad
			{\color[RGB]{204,204,204} \rule{10pt}{8pt}} $S_2$\quad
      {\color[RGB]{176,244,171} \rule{10pt}{8pt}} $S_1$\quad
      {\Large \color{blue} $\bullet$} Free\quad
      {\Large \color[RGB]{255,0,255} $\bullet$} Hanging\quad
      {\color{red} $\boldsymbol{\times}$} Ill-posed
	  \end{small}
	\end{subfigure} \\
	\begin{subfigure}{0.48\textwidth}
    \includegraphics[width=0.95\textwidth]{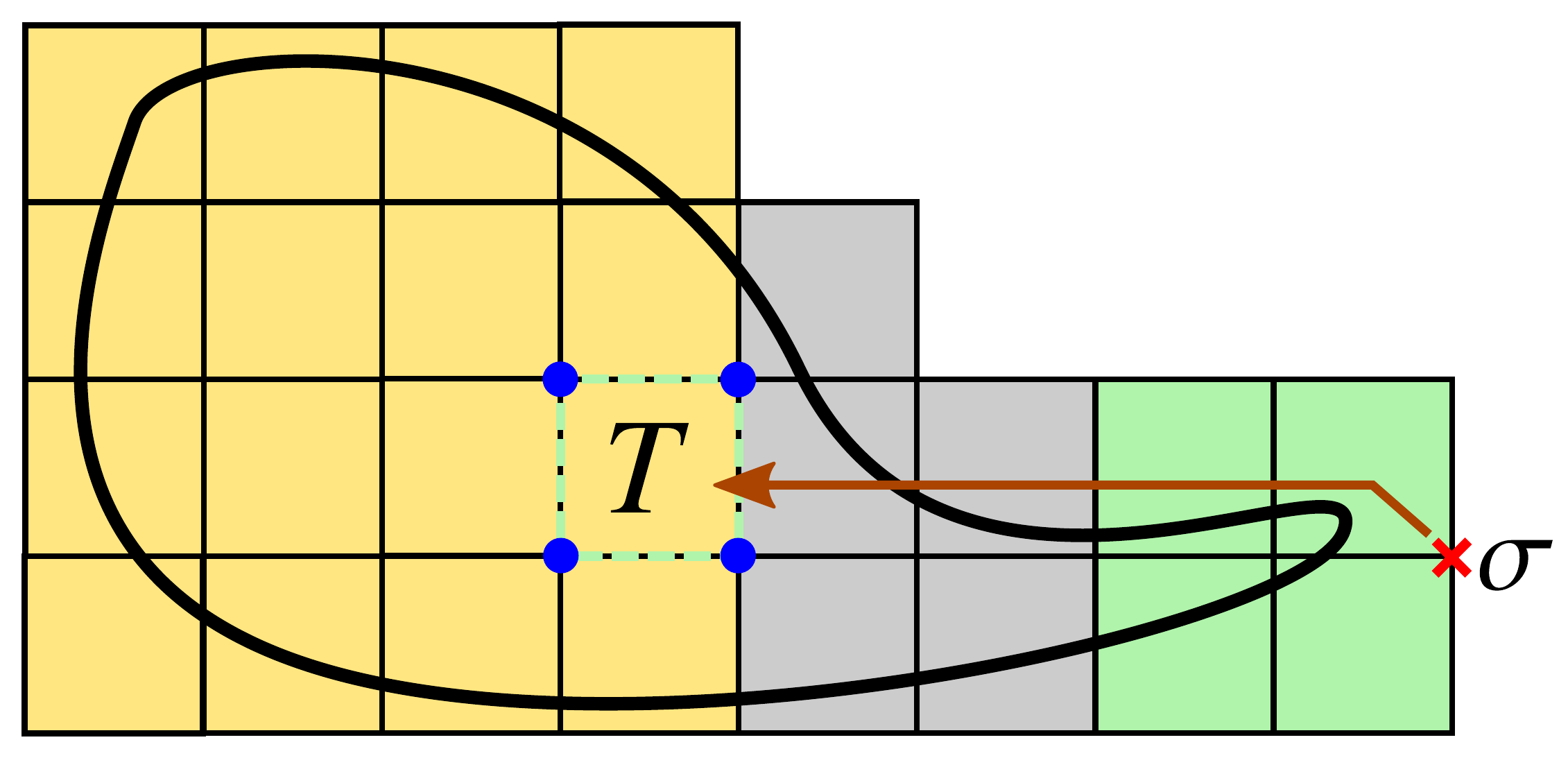}
    \caption{Conforming mesh. The value at $\sigma$ is constrained by the nodal
    values of root cell $T$, which belongs to $S_3$. As a result, to resolve the
    constraint, it is necessary to send to $S_1$ data from $T$.}
    \label{fig:remote-root-a}
  \end{subfigure} \quad
	\begin{subfigure}{0.48\textwidth}
    \includegraphics[width=0.95\textwidth]{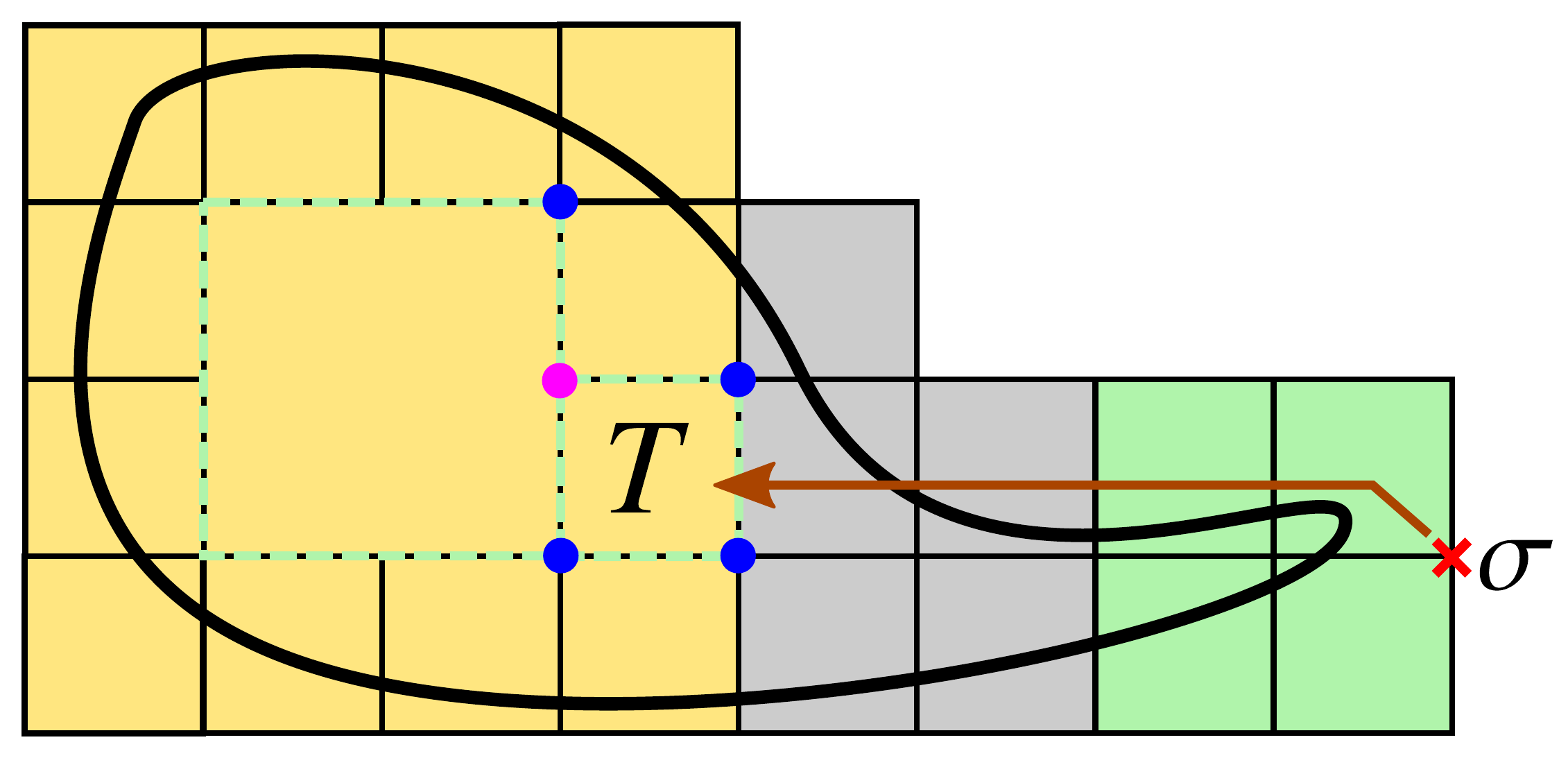}
    \caption{Non-conforming mesh. In contrast with \textsc{(a)}, $T$ touches a
    coarser cell. Therefore, to fully resolve the constraint it is necessary to
    send to $S_1$ data from $T$ \emph{and} its coarser neighbour.}
    \label{fig:remote-root-b}
	\end{subfigure}
  \caption{An ill-posed \ac{dof} $\sigma$ is constrained by a root cell $T$ in
  $S_3$, which is not a neighbour of $S_1$.}
	\label{fig:remote-root}
\end{figure}

\myadded{When both hanging and aggregation \ac{dof} constraints are present, the key
difference with respect to scenario~(2) is that root cells may be in contact with
coarser cells. This means that root cells may have hanging \acp{dof}, which are
cancelled by \acp{dof} located at their coarser cells around, as pointed out in
scenario (1). Therefore, apart from missing root cells, $\T_h^\RG$ must also
contain all missing coarser cells around the root cells relevant to $S$, see
Figure~\ref{fig:remote-root-b}. Specifically, if $\sigma$ is located in $\T_h^\L
\cup \T_h^\TG$, then both the root cell, to which is mapped to, and its coarser
cells around should be in the processor-local portion of $\T_h$. For uniform
meshes, algorithms in charge of importing data associated with these missing cells
are covered in~\cite{Verdugo2019}. They are grounded on the so-called parallel
direct and inverse path reconstruction schemes, which only need nearest-neighbour
communication patterns. For non-conforming meshes, it suffices to modify them such
that they account for non-conforming adjacency in the path reconstruction and also
import missing coarser cells around roots into the processor. We stress the fact
that this approach does not involve any mesh reconfiguration and repartition,
e.g.~it keeps the space-filling curve partition, which is essential for
performance purposes. It also has little impact on overall parallel performance
and scalability and it can be easily implemented in distributed-memory \ac{fe}
codes, as evidenced in Section \ref{sec:numericals}.

In conclusion, if we increment the ghost cell layer with $\T_h^\RG$ as explained
above, then we can correctly identify \emph{within the processor} all constraining
\acp{dof} that are located beyond $\T_h^\L$. It follows that we can resolve all
hanging and aggregation \ac{dof} constraints in $\T_h^\L$ \emph{without any extra
interprocessor communication}, thus achieving our parallel performance target of
maximizing local work, while minimising inter-processor communication. Another
relevant outcome is that we can accommodate to $\Sigma_\L$ the same rationale
detailed in Appendix~\ref{appendix:agfespace} to solve the mixed constraints; all
subsequent steps, to derive expressions for (unified aggregation and hanging) sets
of master \acp{dof} and constraining coefficients, follow exactly the sequential
Ag\ac{fem} counterpart, using the corresponding subdomain definitions. It leads to
the definition of the distributed version of $\V_h^\ag$ and will not be reproduced
here to keep the presentation short.

\begin{remark}
  We observe the following detail: In order to determine whether a \emph{free}
  \ac{dof} $\sigma \in \Sigma_\L$ is well- or ill-posed, i.e. whether $\sigma \in
  \Sigma_\L^{\wp,\F}$ or $\sigma \in \Sigma_\L^{\ip,\F}$, the processor needs to
  know the set of cells, where $\sigma$ has local support. However, a processor
  may not know the full set, based solely on local information, i.e.~$\T_h^\L \cup
  \T_h^\TG$. This scenario is illustrated in Figure~\ref{fig:dist-spaces-a}.
  Fortunately, to let the processor know the complete set of cells, it suffices to
  combine the local information with a single nearest-neighbour communication.
  This result is backed by~\cite[Proposition 4.3]{Badia2019b}, where an analogous
  issue is described, when trying to recover all the processors where $\sigma$ has
  local support, instead of its well- or ill-posed cell status.
  % We observe the following detail: In order to define the partition
  % $\{\Sigma_\L^{\wp,\F}, \Sigma_\L^{\wp,\H}, \Sigma_\L^{\ip,\F},
  % \Sigma_\L^{\ip,\H} \}$, we need to locally determine whether a \ac{dof} is well-
  % or ill-posed. To this end, the processor needs to have in its scope the list of
  % cells where the \ac{dof} has local support. The complete list is generally not
  % included in $\T_h^\L \cup \T_h^\TG$, see Figure~\ref{fig:dist-spaces-a}.
  % However, as mentioned above in the case $\V_h^\ag \equiv \V_h^\std$, all
  % \acp{dof} of $V_h^\std$ with local support in $\T_h^\L$ are located in $\T_h^\L
  % \cup \T_h^\TG$. As a result, a \ac{dof} that has local support in a well-posed
  % cell of a given subdomain $S$, either via (1) or (2) of
  % Definition~\ref{def:free-well-posed}, will be necessarily detected by the
  % corresponding processor. It follows that, there is always at least one neighbour
  % processor, that is able to correctly identify the \ac{dof} as well-posed with
  % processor-local information. Hence, any \ac{dof}, lacking local access to the
  % complete list of cells, can determine, with standard nearest-neighbour
  % communication, whether it has local support on a well-posed cell or not.
  \label{rem:wp}
\end{remark}

}

\section{Numerical experiments}\label{sec:numericals}

Our purpose in this section is to assess numerically the behaviour of
$h$-Ag\ac{fem}. We start with a description of the model problem in
Section~\ref{sec:model-problem}. We consider a Poisson equation with
non-homogeneous Dirichlet boundary conditions and a Nitsche-type variational
form. We introduce next the experimental benchmarks in
Section~\ref{sec:experimental-setup}, composed of several manufactured problems
defined in a set of complex geometries. After this, we jump into the numerical
experiments themselves. We describe and discuss the results of two sets of
experiments, namely convergence tests in Section~\ref{sec:convergence-tests} and
weak-scaling tests in Section~\ref{sec:weak-scaling}.

\subsection{Model problem}\label{sec:model-problem}

Numerical examples consider the Poisson equation with non-homogeneous Dirichlet
boundary conditions. After scaling with the diffusion term, the equation reads:
\emph{find} $u \in H^1(\Omega)$ \emph{such that}
\begin{equation}\label{eq:PoissonEq}
	-\Delta u = f,  \quad \text{in} \ \Omega, \qquad u = g,  \quad \text{on} \
	\Gamma_\mathrm{D} \doteq \partial \Omega,
\end{equation}
where  $f\in H^{-1}(\Omega)$ is the source term and $g\in H^{1/2}(\partial
\Omega)$ is the prescribed value on the Dirichlet boundary. In the numerical
tests, we study both $\V_h^\std$ and $\V_h^\ag$, see
Sections~\ref{sec:std-fe-spaces} and~\ref{sec:ag-fe-spaces}, as possible choices
of $\V_h^{\rm x}$. As stated in Section~\ref{sec:std-fe-spaces}, we consider
weak imposition of boundary conditions, since unfitted methods do not easily
accommodate prescribed values in a strong sense. As usual in the embedded
boundary community, we resort to Nitsche's method to circumvent this
problem~\cite{Badia2018,Schillinger2015,burman_cutfem_2015}. We observe that
this approach provides a consistent numerical scheme with optimal convergence
rates (even for high-order \acp{fe}). According to this, we approximate
\eqref{eq:PoissonEq} with the variational formulation: \emph{find}
$u_h \in \V_h^{\rm x}$ \emph{such that} $\mathrm{a}(u_h,v_h)=\mathrm{b}(v_h)$
\emph{for all} $v_h \in \V_h^{\rm x}$, with
\begin{equation}\label{eq:weak-PoissonEq}
	\begin{array}{l}
		\displaystyle \mathrm{a}(u_h,v_h) \doteq \int_{\Omega} \nabla u_h \cdot
		\nabla v_h \mathrm{\ d}\Omega + \int_{\partial \Omega} \left( \tau u_h v_h  - u_h
		\left( \boldsymbol{n} \cdot \nabla v_h \right ) - v_h \left( \boldsymbol{n}
		\cdot \nabla u_h \right) \right) \mathrm{\ d}{\Gamma}, \quad \text{and} \\
		\displaystyle \mathrm{b}(v_h) \doteq \int_{\Omega} v_h f \mathrm{\ d}\Omega
		+ \int_{\partial \Omega} \left( \tau v_h g - \left( \boldsymbol{n} \cdot \nabla v_h
		\right ) g \right) \mathrm{\ d}\Gamma,
	\end{array}
\end{equation}
with $\boldsymbol{n}$ being the outward unit normal on $\partial \Omega$. We
note that forms $\mathrm{a}(\cdot,\cdot)$ and $\mathrm{b}(\cdot)$ include the
usual terms, resulting from the integration by parts of~\eqref{eq:PoissonEq},
plus additional terms associated with the weak imposition of Dirichlet boundary
conditions with Nitsche's method. %~\cite{Becker2002,nitsche_uber_2013}. 
For further details, we refer to Appendix~\ref{appendix:proofs-wp}, where we prove
well-posedness of Problem~\eqref{eq:weak-PoissonEq} considering $\V_h^\ag$ as the
discretisation space.

Coefficient $\tau > 0$ denotes a mesh-dependent parameter that has to be large
enough to ensure coercivity of $\mathrm{a}(\cdot,\cdot)$. It is prescribed with
the same rationale given in~\cite[Section 4.2]{Verdugo2019}. For $\V_h^\ag$, we
have that $\tau = {\beta^\ag}{h_{\cell}^{-1}}$ for all $\cell \in \T_h^\ip$, where
$h_{\cell}$ is the cell characteristic size and $\beta^\ag$ is a user-defined
constant parameter. Numerical experiments take $\beta^\ag = 25.0$; this value is
enough for having a well-posed problem in all cases considered in
Section~\ref{sec:convergence-tests}. When using $\V_h^\std$, the value in a
generic ill-posed cell takes the form
\begin{equation}\label{eq:tau_std}
	\tau = \beta^\std \lambda_\cell^{\max}, \ \text{for all} \ \cell \in \T_h^\ip,
\end{equation}
where $\beta^\std = 2.0$ and $\lambda_\cell^{\max}$ is the maximum eigenvalue of
the generalised eigenvalue problem: \textit{find} $\mu_\cell \in \V_h^\std|_\cell$
\textit{and} $\lambda_\cell \in \mathbb{R}$ \textit{such that}
\[
	\int_{\cell \cap \Omega} \nabla \mu_\cell \cdot \nabla \xi_\cell \mathrm{\ d}
	\Omega = \lambda_\cell \int_{\cell \cap \partial \Omega} (\nabla \mu_\cell \cdot
	\boldsymbol{n} ) (\nabla \xi_\cell \cdot \boldsymbol{n} ) \mathrm{\ d} \Gamma,
	\enskip \text{for all} \ \xi_\cell \in \V_h^\std|_\cell, \enskip \text{for all}
	\ \cell \in \T_h^\ip.
\]
We notice that $\tau$ computed as in \eqref{eq:tau_std} can be
arbitrarily large, as the measure of the cut $\Omega \cap \cell$, $\cell \in
\T_h^\ip$, tends to zero. This means that, in contrast with $\V_h^\ag$, strongly
ill-conditioned systems of linear equations may arise with $\V_h^\std$, depending
on the position of the cut.

\subsection{Experimental setup}\label{sec:experimental-setup}

The model problem is defined on five different 2D and 3D non-trivial domains
shown in Figure~\ref{fig:geometries}: (a) a planar ``pacman" shape, (b) a
popcorn flake with a wedge removed, (c) a hollow block, (d) a 3-by-3 array of
(c) and (e) a spiral. These geometries appear often in the literature to
study robustness and performance of unfitted \ac{fe} methods (see,
e.g.~\cite{burman_cutfem_2015,Badia2018}). The artificial domain $\Omega^\art$,
on top of which the mesh is generated, is the cuboid $\left[ -1,1 \right]^d$, $d
= 2,3$, for cases (a-c), $\left[ 0,1 \right]^3$ for case (d) and $\left[ -1,1
\right]^2 \times \left[ 0,2 \right]$ for case (e).

\begin{figure}[!h]
  \centering
  \begin{subfigure}[t]{0.30\textwidth}
    \includegraphics[width=\textwidth]{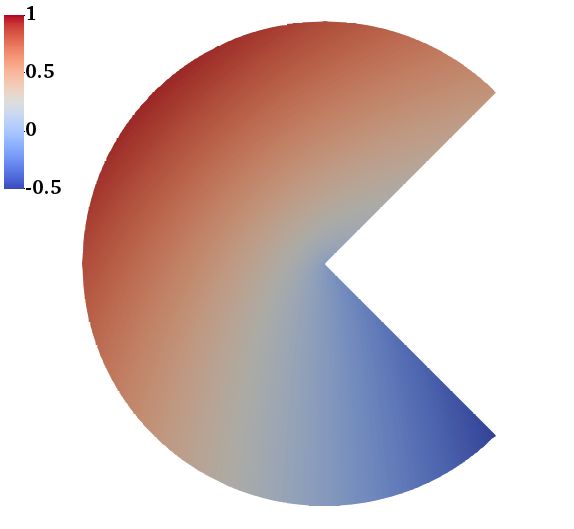}
    \caption{Pacman}
  \end{subfigure}
  \begin{subfigure}[t]{0.30\textwidth}
    \includegraphics[width=\textwidth]{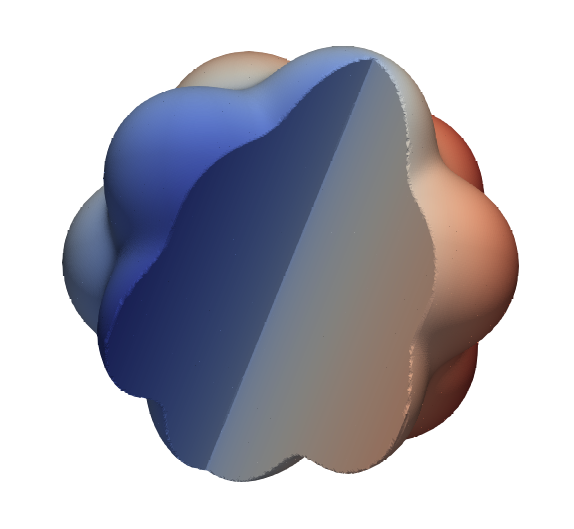}
    \caption{Popcorn}
  \end{subfigure} \\
  \begin{subfigure}[t]{0.30\textwidth}
    \includegraphics[width=\textwidth]{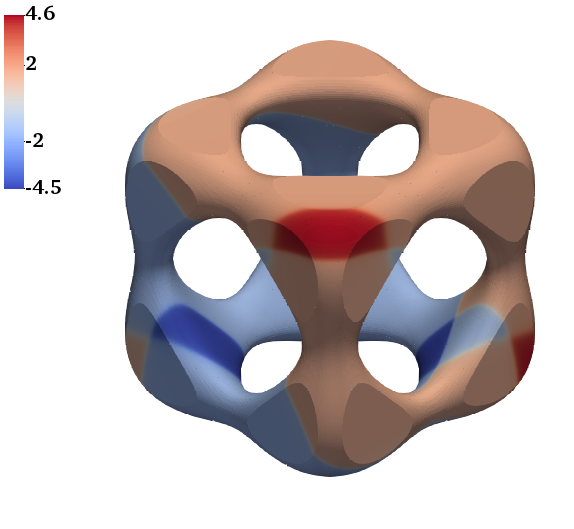}
    \caption{Hollow block}
  \end{subfigure}
  \begin{subfigure}[t]{0.30\textwidth}
    \includegraphics[width=\textwidth]{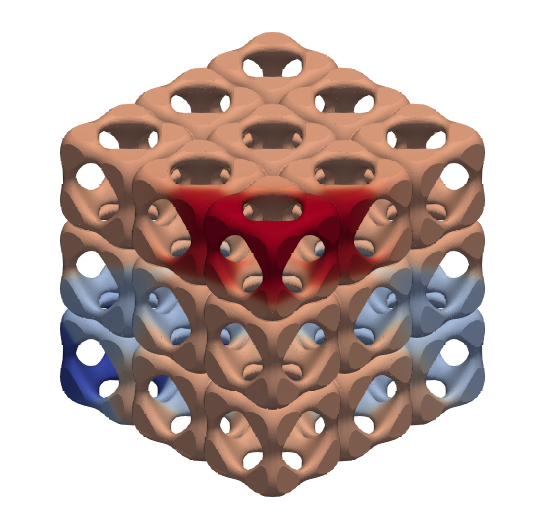}
    \caption{H. b. array}
  \end{subfigure}
  \begin{subfigure}[t]{0.30\textwidth}
    \includegraphics[width=\textwidth]{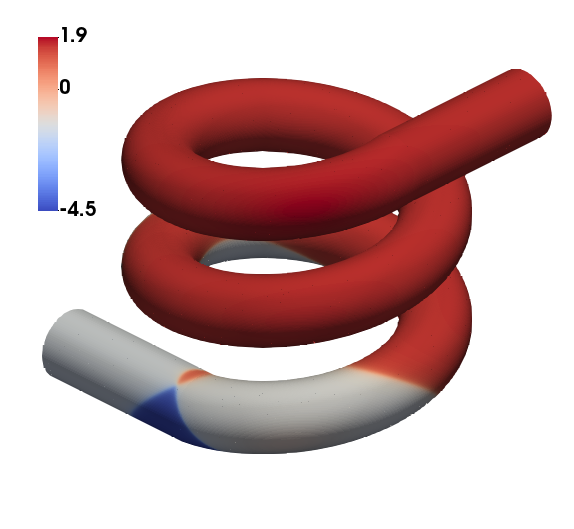}
    \caption{Spiral}
  \end{subfigure}
  \caption{Geometries and numerical solution to the problems studied in the
  examples. (a-b) consider the Fichera corner problem in
  \eqref{eq:fichera}, whereas (c-e) the multiple ``shock" in
  \eqref{eq:shock}.}
  \label{fig:geometries}
\end{figure}

As illustrated in Figure~\ref{fig:geometries}, for geometries (a-b), the source
term and boundary conditions of the Poisson equation are defined, such that the
\ac{pde} has the exact solution
\begin{equation}\label{eq:fichera}
	\begin{aligned}
		& u(r,\theta) = r ^ \alpha \sin \alpha \theta, \ r = \sqrt{ x^2 + y^2 }, \
		\theta = \arctan y/x, \ \alpha = 2 / 3, \\ &
		(x,y) \in \Omega \subset \mathbb{R}^2, \ z=0 \text{ in 2D}, \qquad
		(x,y,z) \in \Omega \subset \mathbb{R}^3 \text{ in 3D}.
	\end{aligned}
\end{equation}
The same applies to (c-e), but seeking a different exact solution given by
\begin{equation}\label{eq:shock}
	\begin{aligned}
		u(r) &= \sum_{i=1,3} \arctan \tau^i ( r - r_0^i ), \\ r = || \x
		-\x_0^i ||_2, \ \x &=
		(x,y,z)\in\Omega\subset\mathbb{R}^3,
	\end{aligned}
\end{equation}
where $\|\cdot\|_2$ denotes the Euclidean norm and
\[\def\arraystretch{1.2}
	\begin{array}{lll}
		\tau^1 = 60, & (x_0^1, y_0^1, z_0^1) = (-1,-1,1), & r_0^1 = 2.5, \\ \tau^2
		= 80, & (x_0^2, y_0^2, z_0^2) = (1,1,-1), & r_0^2 = 1.75, \text{ and } \\
		\tau^3 = 120, & (x_0^3, y_0^3, z_0^3) = (0.5,-3,-3), & r_0^3 = 4.5.
	\end{array}
\]
Problems \eqref{eq:fichera} and \eqref{eq:shock} correspond to adapted versions
of two classical hp-\ac{fem} benchmarks, namely, the Fichera corner and ``shock"
problems (see, e.g.~\cite{demkowicz2006computing}). Derivatives of solution $u$
in \eqref{eq:fichera} are singular at the $r = 0$ axis; in particular, $u \in
H^{1+\frac{2}{3}}(\Omega)$. Recalling \textit{a priori} error estimates, it is
well known that the rate of convergence of the standard \ac{fe} method with
uniform $h$-refinements, when applied to this case, is bounded by regularity
only. Specifically, the energy-norm error\footnote{Recall that, for the
unit-diffusion Poisson equation, the energy norm is given by $\| u \|_{a}^2 =
\int_{\Omega} \left| \nabla u \right|^2 \mathrm{d} \Omega$.} satisfies $\| u -
u_h \|_{a} \leq C h^{-2/3} \| u \|_{H^{1+\frac{2}{3}}(\Omega)}$. However, by
combining \emph{a posteriori} error estimation and $h$-adaptive refinements,
optimal rates of convergence can be restored~\cite{demkowicz2006computing}. On
the other hand, problem \eqref{eq:shock} is characterised by three intersecting
shocks. The solution to the problem is smooth, but it sharply varies in the
neighbourhood of the shocks. In this case, $h$-adaptive standard \ac{fem} does
not affect rates of convergence, but potentially yields meshes that minimise the
number of cells required to achieve a given discretisation error.

The variety of shapes and benchmarks considered here aims to show (i) the
capability of $h$-Ag\ac{fem} of retaining the same benefits $h$-adaptivity
brings, when combined with standard \ac{fem}, while being able (ii) to deal with
complex and diverse 2D and 3D domains in a robust manner and (iii) to yield
remarkable parallel efficiency with state-of-the-art \textit{out-of-the-box}
scalable iterative linear solvers for symmetric positive definite matrices. In
order to do this, we confront numerical results obtained with $\V_h^\ag$ against
those of $\V_h^\std$. In the plots, the two spaces are labelled as
\textit{aggregated} (or \textit{ag.}) and \textit{standard} (or \textit{std.}).
All examples run on \emph{background Cartesian grids}, with standard isotropic
1:4 (2D) and 1:8 (3D) refinement rules; they are commonly referred to as quad-
or octrees in 2D or 3D, resp. Apart from that, continuous FE spaces composed of
first order Lagrangian finite elements are employed. 

In the numerical experiments, we perform convergence tests using three different
remeshing strategies (uniform refinements, \ac{lb} and
\ac{ob}~\cite{li1995theoretical,onate1993study}) in a parallel,
distributed-memory environment. We also assess robustness to cut location and
assess sensitivity to the well-posedness threshold $\eta_0$. Finally, we perform
a weak scalability analysis for some selected ag.\ cases; Table~\ref{tab:params}
summarises the main parameters and computational strategies used in the
numerical examples.

We carry out the numerical experiments at the Marenostrum-IV (MN-IV)
supercomputer, hosted by the Barcelona Supercomputing Centre. Concerning the
software, an MPI-parallel implementation of the $h$-Ag\ac{fem} method is
available at \fempar{}~\cite{badia-fempar}. \fempar{} is linked
against \p4est{} v2.2~\cite{burstedde_p4est_2011}, as the octree Cartesian grid
manipulation engine, and \petsc{} v3.11.1~\cite{petsc-user-ref}
distributed-memory linear algebra data structures and solvers. To show that
$\V_h^\ag$ leads to systems, that are amenable to well established scalable
linear solvers for standard \ac{fe} analysis on body-fitted meshes, we resort to
the broad suite of linear solvers available in the \petsc{}
library~\cite{petsc-user-ref}. In particular, we use a \ac{cgm} method,
preconditioned by a smoothed-aggregation \ac{amg} scheme called
\gamg{}. %~\cite{GAMGweb}. 
The preconditioner is set up in favour of reducing, as
much as possible, the deviation from its default configuration, as
in~\cite{Verdugo2019}. We do this in order to show that Ag\ac{fem} blends well
with common \ac{amg} solvers, whereas std.\ unfitted \ac{fem} does not. Both
solver and preconditioner are readily available through the Krylov Methods
\texttt{KSP} module of \petsc{}. In order to advance convergence tests down to
low global energy-norm error values, without being polluted by the linear solver
accuracy, convergence of \gamg{} is declared when $\| \mathbf{r} \|_2/ \|
\mathbf{b} \|_2 < 10^{-9}$ within the first $500$ iterations, where
$\mathbf{r}\doteq \mathbf{b}-\mathbf{A}\mathbf{x}^{\rm cg}$ is the
unpreconditioned residual.

\begin{table}[ht!]
	\centering
	\begin{small}
		\begin{tabular}{ll}
			\toprule
			Description & Considered methods/values \\
			\midrule
			Model problem & Poisson equation (Nitsche's formulation)
			\vspace{0.12cm} \\
			Problem geometry & 2D: Pacman shape; 3D: Popcorn flake, \\
			& Hollow block, Hollow block array and spiral \vspace{0.12cm} \\
			\ac{amr} benchmark & Fichera corner and multiple-shock
			problem~\cite{demkowicz2006computing} \vspace{0.12cm} \\
			Remeshing strategy & Uniform, Li and Bettess~\cite{li1995theoretical},
			and O\~{n}ate and Bugeda~\cite{onate1993study} \vspace{0.12cm} \\
      Experimental computer environment & Parallel (distributed-memory) 
      \vspace{0.12cm} \\
      Mesh topology & Single quad- or octree \vspace{0.12cm} \\
			Parallel mesh generation and partitioning tool & \p4est{}
			library~\cite{burstedde_p4est_2011} \vspace{0.12cm} \\
			Well-posed cut cell criterion & $\eta_0 = 0.25$ \vspace{0.12cm} \\
			\ac{fe} spaces & Aggregated $\V_h^\ag$ and standard $\V_h^\std$
			\vspace{0.12cm} \\
			Cell type & Hexahedral cells \vspace{0.12cm} \\
			Interpolation & Piece-wise bi/trilinear shape functions \vspace{0.12cm} \\
      Linear solver & Preconditioned conjugate gradients \vspace{0.12cm} \\
			Parallel preconditioner & Smoothed-aggregation \gamg{}%~\cite{GAMGweb}
			\vspace{0.12cm} \\
			\gamg{} stopping criterion & $\| \mathbf{r} \|_2/ \| \mathbf{b} \|_2 <
			10^{-9}$ \vspace{0.12cm} \\
			Coef. in Nitsche's penalty term for $\V_h^\ag$ & $\beta = 25.0$ \\
			\bottomrule
		\end{tabular}
	\end{small}
	\caption{Summary of the main parameters and computational strategies used in
	the numerical examples.}
	\label{tab:params}
\end{table}

\subsection{Convergence tests}\label{sec:convergence-tests}

Convergence tests in relative energy-norm error are carried out with three
different mesh refinement strategies. The first one is uniform $h$-refinements,
in pursuance of both exposing the behaviour of Ag\ac{fem}, in absence of hanging
node constraints, and the limited regularity of the Fichera corner problem. The
remaining two are error-driven; they are distinguished by different optimality
criteria on the elemental error indicator $\gamma_\cell$, for any $\cell \in
\T_h$. It is not in the scope of this work to design \emph{a posteriori} error
estimation techniques for Ag\ac{fem}, although there are some works with other
unfitted \ac{fe} methods that explore this question~\cite{burman2019posteriori}.
Hence, since the target problems have known analytical solution, $\gamma_\cell$
is taken as the energy norm of the local true error $e = u - u_h$, that is,
\[
	\gamma_\cell = \| e \|_{\restrict{a}{\cell\cap\Omega}} = \| u - u_h
	\|_{\restrict{a}{\cell\cap\Omega}}, \quad \cell \in \T_h.
\]
Error-driven mesh adaptation seeks an optimal mesh with an iterative procedure.
In the examples below, optimality is declared when the global \textit{absolute}
discretisation error, measured in energy norm, $\| e \|_{a}$ is below a
prescribed quantity $\gamma$, i.e.\
\begin{equation}
	\| e \|_{a} \le \gamma, \quad \gamma > 0.
	\label{eq:acceptability_criterion}
\end{equation}
\eqref{eq:acceptability_criterion} is referred to as the
\textit{acceptability criterion}. % Usually done with the relative error,
% but the results and the discussion is the same, regardless of error type.

The process starts with an initial guess of the optimal mesh. After finding the
approximate solution and the exact cell-wise error distribution, a new mesh is
defined with a remeshing strategy. This step consists in comparing each
$\gamma_\cell$ to a given threshold, commonly known as the \textit{optimality
criterion}, which is later defined. Depending on the result of the comparison, a
different remeshing flag is assigned to the cell. If $\gamma_\cell$ is above the
threshold, $\cell$ is marked for refinement. Otherwise, $\cell$ is left unmarked
or, optionally, marked for coarsening, when $\gamma_\cell$ falls well below the
threshold. Following this, the mesh is transformed, according to the cell-wise
remeshing flags, and partitioned. Next, a new \ac{fe} space is created, by
distributing \acp{dof} on top of the new mesh and computing the nonconforming
\ac{dof} constraints and, if using $\V_h^\ag$, also the ill-posed \ac{dof}
constraints. After \ac{fe} integration and assembly, the resulting linear system
is solved and the cell-wise error distribution is updated. If the current mesh
complies with the acceptability criterion of \eqref{eq:acceptability_criterion},
the process is stopped, otherwise it goes back to the application of the
optimality criterion. % This high-level \ac{fe} simulation loop is described 
% in Alg.~\add{}.

As mentioned before, two different optimality criteria (thresholds for
refinement) are studied%, see e.g.~\cite{diez1999unified} for a detailed review. 
The first one, the \ac{lb}~\cite{li1995notes,li1995theoretical}
criterion, establishes that the error distribution in an optimal mesh (denoted
with *) is uniform, that is
\begin{equation}
	\| e^* \|_{ \cell^*\cap\Omega } = \frac{\gamma}{\sqrt{M^*}}, \quad \cell^* =
	1,\ldots,M^*,
	\label{eq:li_and_bettess}
\end{equation}
where $M^*$ is the number of cells in the optimal mesh. At each mesh adaptation
step, this quantity is estimated as
\begin{equation}
	M^* = \gamma^{-d/m} \left( \sum_{\cell=1}^{M} \| e \|_{\cell}^{d/(m+d/2)}
	\right)^{(m+d/2)/m},
	\label{eq:estimate_mesh_size}
\end{equation}
with $d$ the space dimension, $m$ the degree of the interpolation (in
second-order elliptic problems) and $M$ the number of cells of the current
iteration. On the other hand, the \ac{ob}~\cite{onate1993study} criterion
considers that the distribution of error \textit{density} in an optimal mesh is
uniform, that is
\begin{equation}
	\| e^* \|_{ \cell^*\cap\Omega } = \frac{\gamma \Omega_{\cell^*\cap\Omega}^{1/2}
	}{
	\Omega^{1/2} }, \quad \cell^* = 1,\ldots,M^*,
	\label{eq:onate_and_bugeda}
\end{equation}
where $\Omega$ is the measure of the domain and $\Omega_{\cell^*\cap\Omega}$ is
the measure of $\cell^*\cap\Omega$. While the former criterion has
been proved~\cite{li1995notes} to provide standard body-fitted \ac{fe} meshes
satisfying \eqref{eq:acceptability_criterion} with the least number of
elements, the latter scales the threshold in terms of the size of the
(ill-posed) cell. One of the goals of the following experiments is to see how
both strategies perform in the context of unfitted \acp{fe}. Note that, with
respect to standard body-fitted \ac{fem}, both remeshing strategies are almost
applied verbatim to an unfitted \ac{fe} setting; the only difference being that
the local quantities in cut cells are computed in the interior part only, in the
same way as for the local integration of the weak form stated in
\eqref{eq:weak-PoissonEq}.

Convergence tests with uniform $h$-refinements follow the usual procedure,
whereas error-driven tests are controlled with a finite sequence of decreasing
error objectives $\gamma_i$, $i > 1$. For each $i > 1$, the iterative procedure
described above is carried out to find the mesh that complies with the
acceptability criterion of \eqref{eq:acceptability_criterion} with
$\gamma = \gamma_i$. If subscript $\gamma_i$ refers to the quantities obtained at the
last mesh iteration, at the end of the procedure we can extract the pair
\[ \left( \frac{\| e \|_{a,\gamma_i}}{\| u
\|_{a,\gamma_i}},N_{\mathrm{dofs}}^{\gamma_i} \right),
\]
that is, a point of the convergence test curve. Figure~\ref{fig:meshes} depicts
some meshes found with this iterative procedure, using the \ac{lb}
acceptability criterion.

\begin{figure}[!h]
  \centering
%  \begin{subfigure}[t]{0.24\textwidth}
%    \includegraphics[width=\textwidth]{pacman_mesh_1.png}
%    \caption{$4.7 \cdot 10^{-2}$}
%  \end{subfigure}
%  \begin{subfigure}[t]{0.19\textwidth}
%    \includegraphics[width=\textwidth]{pacman_mesh_2.png}
%    \caption{$3.0 \cdot 10^{-2}$}
%  \end{subfigure}
%  \begin{subfigure}[t]{0.24\textwidth}
%    \includegraphics[width=\textwidth]{pacman_mesh_3.png}
%    \caption{$1.8 \cdot 10^{-2}$}
%  \end{subfigure}
%  \begin{subfigure}[t]{0.24\textwidth}
%    \includegraphics[width=\textwidth]{pacman_mesh_4.png}
%    \caption{$1.5 \cdot 10^{-2}$}
%  \end{subfigure}
  \begin{subfigure}[t]{0.19\textwidth}
    \includegraphics[width=\textwidth]{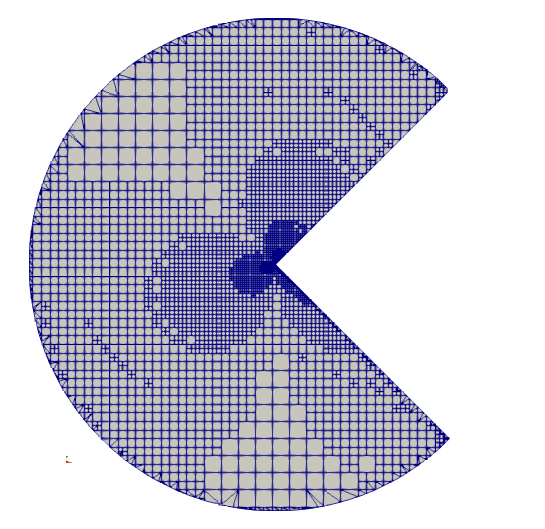}
%    \caption{$7.6 \cdot 10^{-3}$}
  \end{subfigure}
%  \\
%  \begin{subfigure}[t]{0.24\textwidth}
%    \includegraphics[width=\textwidth]{cheese_mesh_1.png}
%    \caption{$8.6 \cdot 10^{0}$}
%  \end{subfigure}
%  \begin{subfigure}[t]{0.24\textwidth}
%    \includegraphics[width=\textwidth]{cheese_mesh_2.png}
%    \caption{$6.8 \cdot 10^{0}$}
%  \end{subfigure}
%  \begin{subfigure}[t]{0.24\textwidth}
%    \includegraphics[width=\textwidth]{cheese_mesh_3.png}
%    \caption{$4.1 \cdot 10^{0}$}
%  \end{subfigure}
%  \begin{subfigure}[t]{0.19\textwidth}
%    \includegraphics[width=\textwidth]{cheese_mesh_4.png}
%    \caption{$3.5 \cdot 10^{0}$}
%  \end{subfigure}
  \begin{subfigure}[t]{0.24\textwidth}
    \includegraphics[width=\textwidth, trim = 0.0cm 1.5cm 0.0cm 0.0cm, clip]{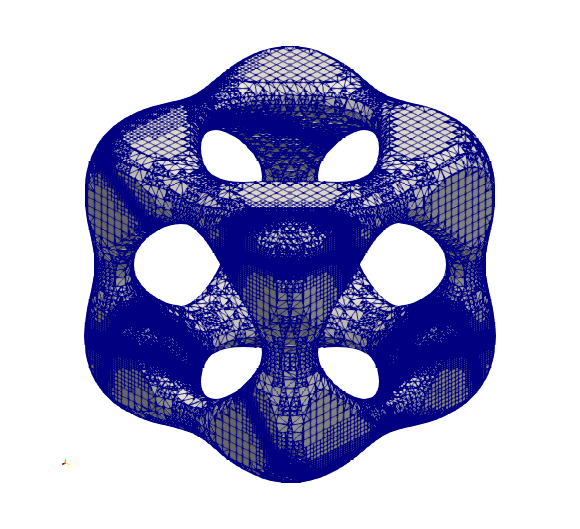}
%    \caption{$1.8 \cdot 10^{0}$}
  \end{subfigure}
%  \\
%  \begin{subfigure}[t]{0.24\textwidth}
%    \includegraphics[width=\textwidth]{spiral_mesh_1.png}
%    \caption{$4.3 \cdot 10^{0}$}
%  \end{subfigure}
%  \begin{subfigure}[t]{0.24\textwidth}
%    \includegraphics[width=\textwidth]{spiral_mesh_2.png}
%    \caption{$3.8 \cdot 10^{0}$}
%  \end{subfigure}
%  \begin{subfigure}[t]{0.24\textwidth}
%    \includegraphics[width=\textwidth]{spiral_mesh_3.png}
%    \caption{$2.1 \cdot 10^{0}$}
%  \end{subfigure}
%  \begin{subfigure}[t]{0.19\textwidth}
%    \includegraphics[width=\textwidth]{spiral_mesh_5.png}
%    \caption{$1.1 \cdot 10^{0}$}
%  \end{subfigure}
  \begin{subfigure}[t]{0.24\textwidth}
    \includegraphics[width=\textwidth]{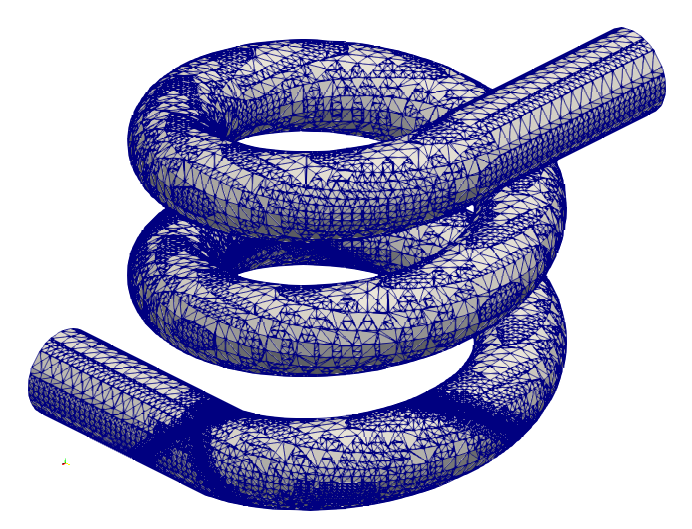}
%    \caption{$9.5 \cdot 10^{-1}$}
  \end{subfigure}
  \caption{Pacman-Fichera, hollow-shock and spiral-shock examples: optimal 
  meshes obtained with the \ac{lb} criterion for the convergence test.}
  \label{fig:meshes}
\end{figure}

We carry out all convergence tests for a fixed number of threads (MPI tasks). We
employ six MN-IV high-memory nodes and map each core to a different MPI task.
Therefore, the experiments are launched in $6 \cdot 48 = 288$ processors. The
partition of the mesh considers $288$ subdomains and is defined to seek an equal
distribution of the number of cells among processors (\p4est{} default setting).
The well-posedness threshold for aggregation $\eta_0$ (see
Section~\ref{sec:cell-aggr}) is prescribed to $0.25$ in what follows.

Let us now start the discussion of the numerical results obtained with
convergence tests. As shown in Figure~\ref{fig:tests_parallel} $h$-Ag\ac{fem}
behaviour consistently mirrors the one of std.\ $h$-\ac{fem}. This includes that
(1) $h$-Ag\ac{fem} always produces more optimal meshes, in terms of the error,
than its non-adaptive version; and (2) optimal convergence rates are retained,
even for the $h$-Ag\ac{fem} Fichera problems, where convergence in the
non-adaptive version is limited by regularity.

Although the std.\ method is slightly more accurate than its ag.\ counterpart
for the Fichera problems with uniform refinements, the usual behaviour is that
they are very similar in terms of accuracy. Another outcome observed is that the
\ac{lb} criterion is clearly more cost-efficient, in terms of mesh size, than
the \ac{ob} one for both std.\ and ag.\ variants. This is also reported
in~\cite{diez1999unified} with std.\ $h$-\ac{fem}.

However, as expected, the linear solver does not manage to generate a solution
in most of std.\ \ac{fe} cases. Either the preconditioner cannot be generated or
it is not positive definite (thus, incompatible with the conjugate gradient
method). Both issues are directly related to the severe ill-conditioning of
matrices obtained with the std.\ method, as extensively reported in previous
works~\cite{Badia2018,Verdugo2019}. On the other hand, when using the ag.\
method, \gamg{} is fully robust and converges towards the solution at the
$10^{-9}$ tolerance.

\begin{figure}[!h]
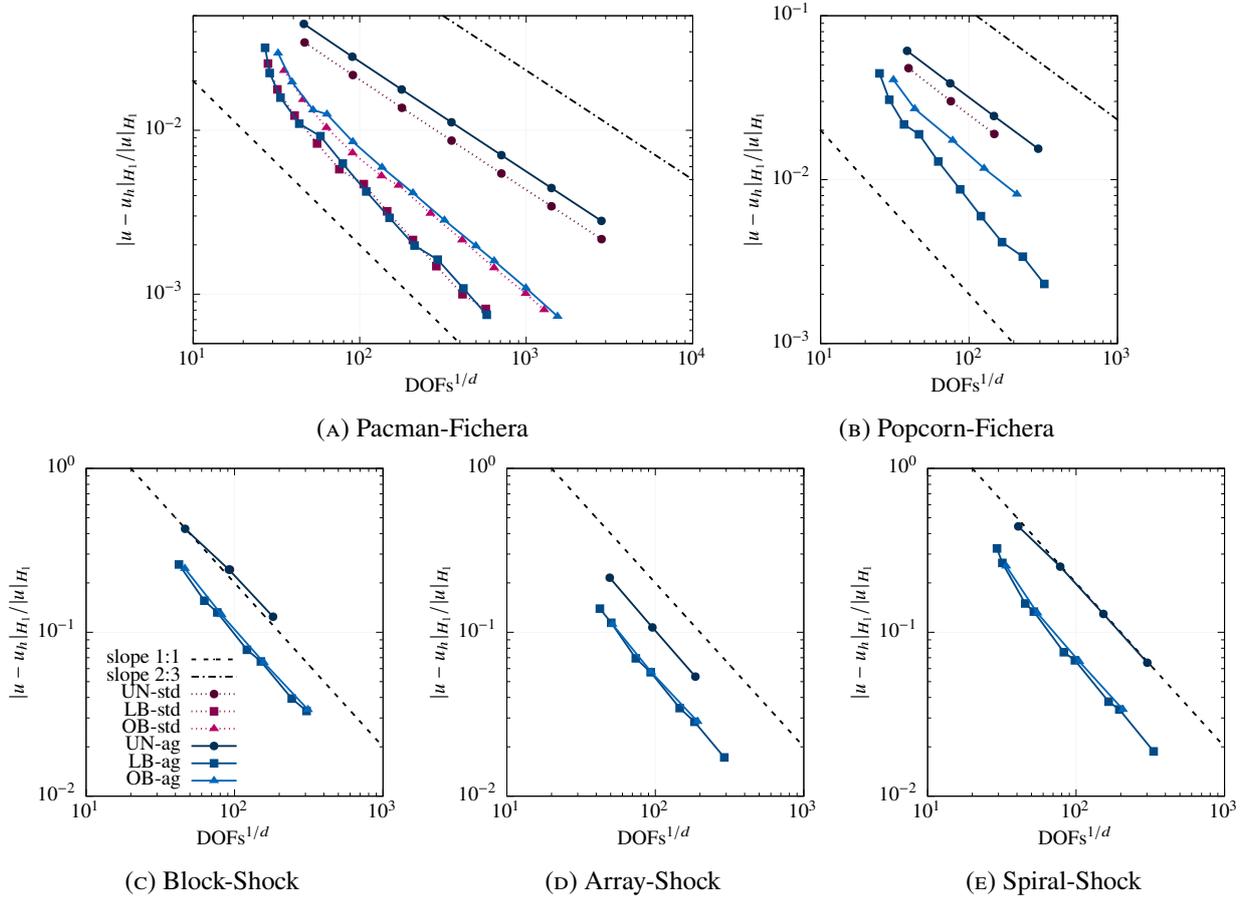

  \centering
  \begin{subfigure}[t]{0.48\textwidth}
    \scalebox{0.7}{\input{figures/error_vs_dofs_parallel/pb0c0d0e0f1g0h0i0j0k0.tex}}
    \caption{Pacman-Fichera}
    \label{fig:tests_parallel-pacman}
  \end{subfigure}
  \begin{subfigure}[t]{0.32\textwidth}
    \scalebox{0.7}{\input{figures/error_vs_dofs_parallel/pb0c1d0e2f1g0h0i0j0k0.tex}}
    \caption{Popcorn-Fichera}
  \end{subfigure} \\
  \begin{subfigure}[t]{0.32\textwidth}
    \scalebox{0.7}{\input{figures/error_vs_dofs_parallel/pb0c1d0e3f1g0h0i0j0k1.tex}}
    \caption{Block-Shock}
  \end{subfigure}
  \begin{subfigure}[t]{0.32\textwidth}
    \scalebox{0.7}{\input{figures/error_vs_dofs_parallel/pb0c1d0e4f1g0h0i0j0k1.tex}}
    \caption{Array-Shock}
  \end{subfigure}
  \begin{subfigure}[t]{0.32\textwidth}
    \scalebox{0.7}{\input{figures/error_vs_dofs_parallel/pb0c1d0e1f1g0h0i0j0k1.tex}}
    \caption{Spiral-Shock}
  \end{subfigure}
  \caption{Convergence tests in parallel environment for 288 tasks.}
  \label{fig:tests_parallel}
\end{figure}

Despite poor robustness of the solver with the std.\ method, available results
in Figure~\ref{fig:iterations_parallel} are enough to clearly identify higher
growth rates in number of iterations for std.\ matrices, than for ag.\ ones.
This exposes that, among the two methods, only $h$-Ag\ac{fem} is potentially
scalable, as the number of iterations mildly grows with the size of the problem;
even for $h$-Ag\ac{fem} points in Figure~\ref{fig:iterations_parallel} with the
largest number of \acp{dof}, convergence is declared in almost twenty
iterations. We have verified that, in this context, the solver achieves
single-digit reduction of the residual norm in 2-3 iterations, at most. Textbook
multigrid efficiency is attained when the solver uses a modest number of point
smoothing steps and convergence nearly advances at one digit in reduction of the
residual norm per iteration~\cite{petsc-user-ref}. We have checked the former is
satisfied, by inspecting \petsc{} log data, whereas the latter is broadly
fulfilled in Ag\ac{fem} experiments. Therefore, \gamg{} on $h$-Ag\ac{fem}
matrices is not only robust, but also efficient.

\begin{figure}[!h]
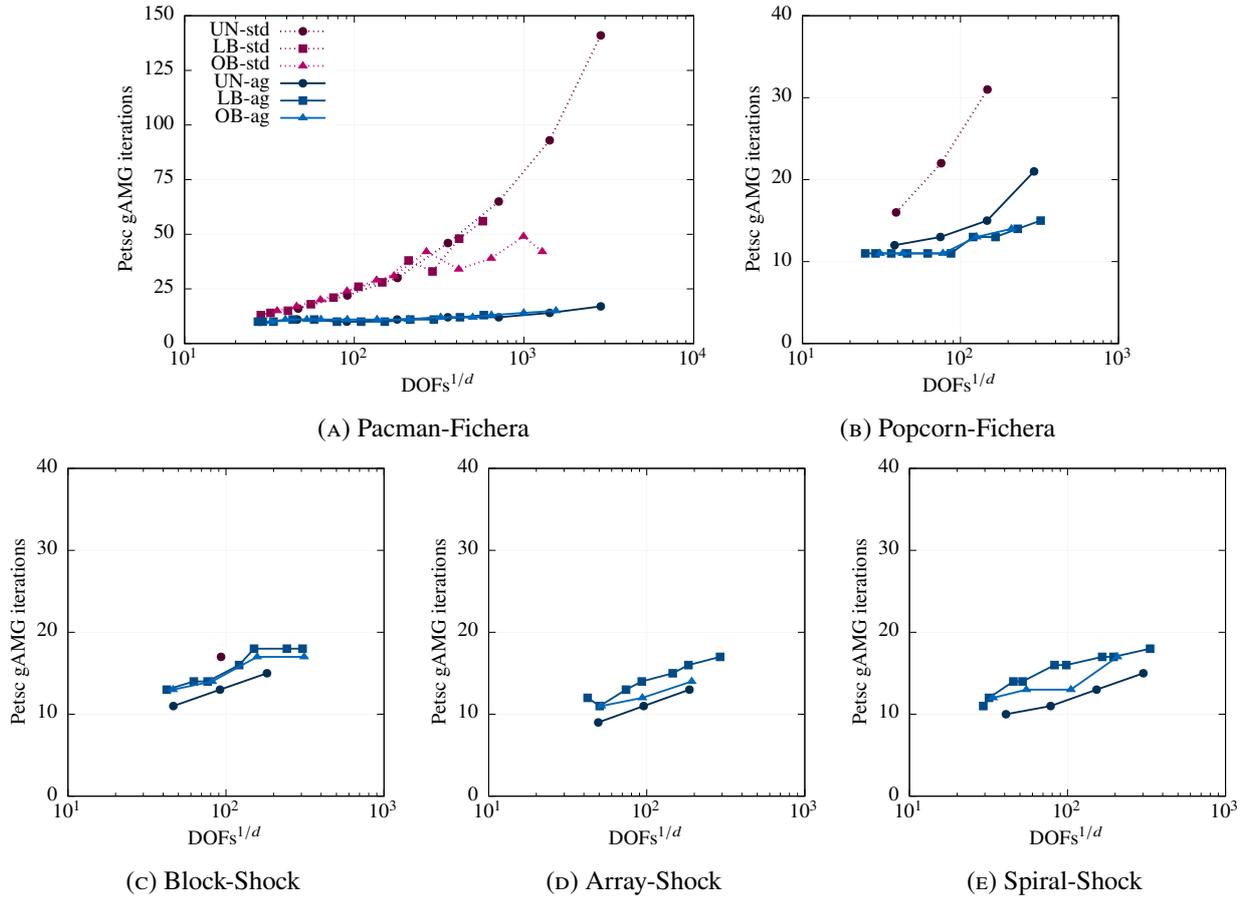

  \centering
  \begin{subfigure}[t]{0.48\textwidth}
    \scalebox{0.7}{\input{figures/iterations/ib0c0d0e0f1g0h0i0j0k0.tex}}
    \caption{Pacman-Fichera}
  \end{subfigure}
  \begin{subfigure}[t]{0.32\textwidth}
    \scalebox{0.7}{\input{figures/iterations/ib0c1d0e2f1g0h0i0j0k0.tex}}
    \caption{Popcorn-Fichera}
  \end{subfigure} \\
  \begin{subfigure}[t]{0.32\textwidth}
    \scalebox{0.7}{\input{figures/iterations/ib0c1d0e3f1g0h0i0j0k1.tex}}
    \caption{Block-Shock}
  \end{subfigure}
  \begin{subfigure}[t]{0.32\textwidth}
    \scalebox{0.7}{\input{figures/iterations/ib0c1d0e4f1g0h0i0j0k1.tex}}
    \caption{Array-Shock}
  \end{subfigure}
  \begin{subfigure}[t]{0.32\textwidth}
    \scalebox{0.7}{\input{figures/iterations/ib0c1d0e1f1g0h0i0j0k1.tex}}
    \caption{Spiral-Shock}
  \end{subfigure}
  \caption{\gamg{} solver iterations in parallel environment for 288 number of
  tasks.}
  \label{fig:iterations_parallel}
\end{figure}

A final experiment with convergence tests looks at the sensitivity of Ag\ac{fem}
to the well-posedness threshold $\eta_0$. As it is shown in
Figure~\ref{fig:gamma_0_sensitivity}, low values of $\eta_0$ may not bypass the
small cut-cell problem and hinder \gamg{} solvability, as shown in the 3D
parallel examples. Conversely, high values of $\eta_0$ do not affect robustness,
but increase solver iterations and reduce (local) accuracy. This is most likely
an effect of excessive well-posed-to-ill-posed \ac{dof} extrapolation. As a
result, optimal $\eta_0$ values may be found in the middle of the $[0,1]$ range.
This means that, while enforcing a minimum amount of aggregation is required to
guarantee robustness, superfluous aggregation deteriorates solver efficiency.
This effect is particularly prominent in $h$-adaptivity; setting $\eta_0 = 1$ on
uniform meshes leads to decent results, as demonstrated in a previous
work~\cite{Verdugo2019}.

\begin{figure}[!h]
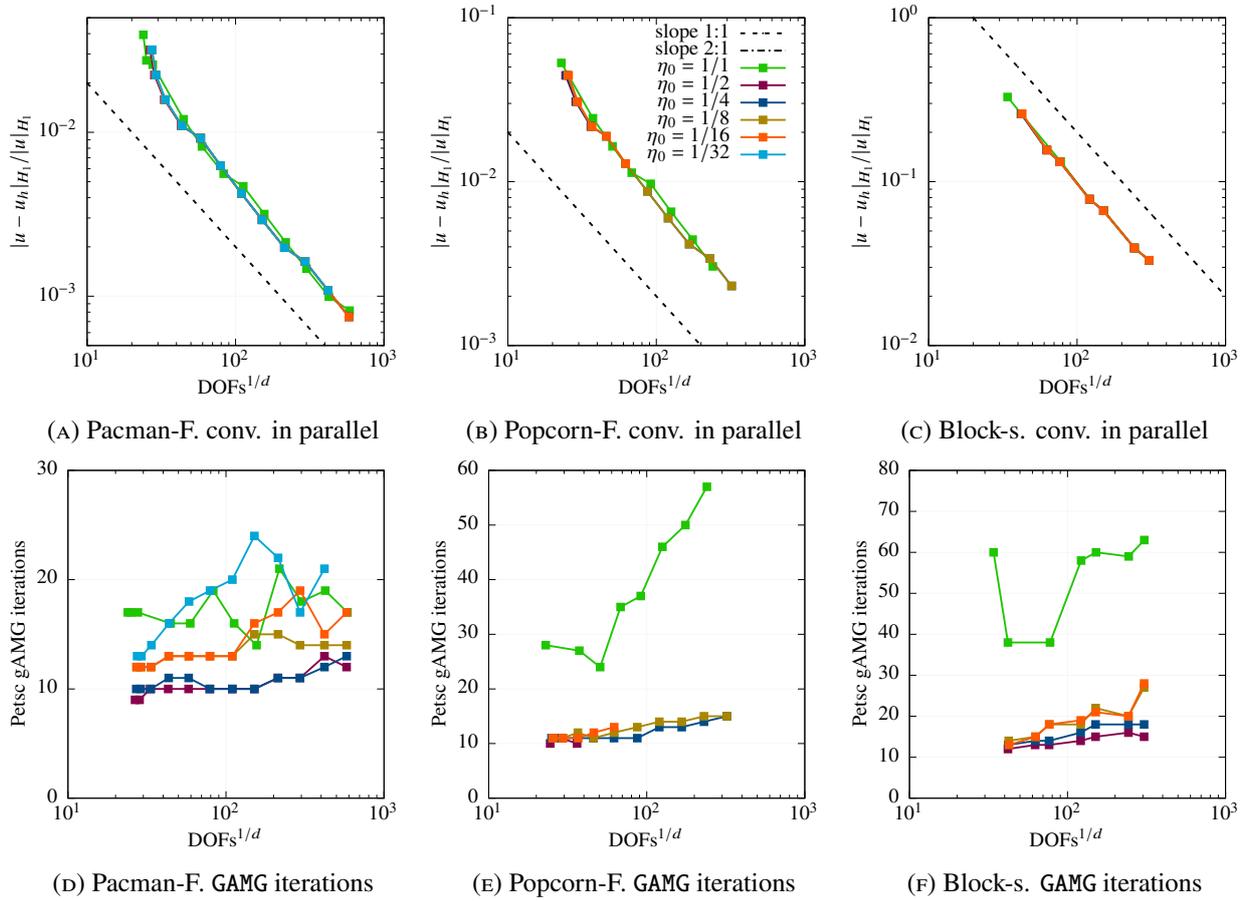

  \centering
  \begin{subfigure}[t]{0.32\textwidth}
    \scalebox{0.7}{\input{figures/eta_0_sensitivity/pb0c0d0e0f1g0h0i0j0k0lxmx.tex}}
    \caption{Pacman-F. conv. in parallel}
  \end{subfigure}
  \begin{subfigure}[t]{0.32\textwidth}
    \scalebox{0.7}{\input{figures/eta_0_sensitivity/pb0c1d0e2f1g0h0i0j0k0lxmx.tex}}
    \caption{Popcorn-F. conv. in parallel}
  \end{subfigure}
  \begin{subfigure}[t]{0.32\textwidth}
    \scalebox{0.7}{\input{figures/eta_0_sensitivity/pb0c1d0e3f1g0h0i0j0k1lxmx.tex}}
    \caption{Block-s. conv. in parallel}
  \end{subfigure} \\
  \begin{subfigure}[t]{0.32\textwidth}
    \scalebox{0.7}{\input{figures/eta_0_sensitivity/ib0c0d0e0f1g0h0i0j0k0lxmx.tex}}
    \caption{Pacman-F. \gamg{} iterations}
  \end{subfigure}
  \begin{subfigure}[t]{0.32\textwidth}
    \scalebox{0.7}{\input{figures/eta_0_sensitivity/ib0c1d0e2f1g0h0i0j0k0lxmx.tex}}
    \caption{Popcorn-F. \gamg{} iterations}
  \end{subfigure}
  \begin{subfigure}[t]{0.32\textwidth}
    \scalebox{0.7}{\input{figures/eta_0_sensitivity/ib0c1d0e3f1g0h0i0j0k1lxmx.tex}}
    \caption{Block-s. \gamg{} iterations}
  \end{subfigure} \\
  \caption{$h$-Ag\ac{fem} sensitivity to $\eta_0$ with the \ac{lb} criterion.
  Recall that $\eta_0 = 0.25$ is the reference value in previous experiments
  (see Figures~\ref{fig:tests_parallel}-\ref{fig:iterations_parallel}).}
  \label{fig:gamma_0_sensitivity}
\end{figure}

\subsection{Weak scaling}\label{sec:weak-scaling}

The starting point of weak scaling tests is the parallel convergence test setup
of the previous section. As explained, a single convergence test case results in
a set of pairs
\[ \left\{ \left( \frac{\| e \|_{a,\gamma_i}}{\| u
\|_{a,\gamma_i}},N_{\mathrm{dofs}}^{\gamma_i} \right) \right\}_{\gamma_i > 1},
\] associated with a finite sequence of decreasing target error values
$\gamma_i$, $i > 1$. Each test corresponds to an individual curve, e.g.\ the
\ac{lb}-ag curve for the Pacman-Fichera test case in
Figure~\ref{fig:tests_parallel-pacman}. Other quantities can be extracted from
the test, e.g.\ the size of the global triangulation
$N_{\mathrm{cells}}^{\gamma_i}$. In Section~\ref{sec:convergence-tests}, each
pair was obtained for a fixed number of processors $P = 288$. Naturally, as
$N_{\mathrm{cells}}^{\gamma_i}$ increases with $i$, so does the size of the
local portion of the triangulation $n_{\mathrm{cells}}^{\gamma_i}$, owned by
each processor.

Given $\left\{N_{\mathrm{cells}}^{\gamma_i}\right\}_{i > 1}$ associated with a
convergence test, a weak scaling one can be derived by adjusting the number of
processors $P^i$ for each $\gamma_i$, such that $n_{\mathrm{cells}}^{\gamma_i}$
remains approximately constant for all $i > 1$. This can be achieved, by e.g.\
prescribing
\[
	P^i = P^1 \left\lfloor
	\frac{N_{\mathrm{cells}}^{\gamma_i}}{N_{\mathrm{cells}}^{\gamma_1}}
	\right\rfloor, \ i > 1,
\]
where $P^1$ is a fixed initial number of processors and $\lfloor \cdot \rfloor$
is the \textit{floor} function; given a real number $x$, $\lfloor x \rfloor$ is
the greatest integer less than or equal to $x$. From here, the weak scaling test
consists merely in repeating the convergence test, taking $P^i$ processors for
each $\gamma_i$. In this way, by keeping the local size of the mesh
$n_{\mathrm{cells}}^{\gamma_i}$ constant, we can straightforwardly study how
$h$-Ag\ac{fem} scales with global size of the problem.\footnote{We have checked
that, using this approach, the local size of the problem (local number of
\acp{dof}) increases monotonically, though mildly, for $i > 1$. Thus, this
conservative approach allows us to examine how the problem scales, avoiding
cumbersome strategies to balance \acp{dof}.} Table~\ref{tab:partitions} gathers
the sequences $\left\{P^i\right\}_{i > 1}$ obtained following this procedure for
the two test cases that will be studied in this section, namely, the
Popcorn-Fichera and Hollow-Shock problems for the Ag\ac{fem} method with the
$\ac{lb}$ remeshing criterion and $\eta_0 = 0.25$.

\begin{table}[!h]
	\centering
	\begin{small}
		\begin{tabular}{lrrrrrrr}
		\toprule
		\multicolumn{8}{c}{Popcorn-Fichera \ac{lb}-ag with $\eta_0 = 0.25$ and
		$n_{\mathrm{cells}} \approx 15.5k$} \\
		\midrule
		$P$ & 2 & 8 & 19 & 52 & 132 & 349 & 883 \\
		$N_{\mathrm{cells}}$ & 31k & 130k & 301k & 800k & 2,025k & 5,354k & 13,553k
		\\
		\midrule
		\multicolumn{8}{c}{Hollow-shock \ac{lb}-ag with $\eta_0 = 0.25$ and
		$n_{\mathrm{cells}} \approx 21.0k$} \\
		\midrule
		$P$ & 6 & 17 & 29 & 107 & 194 & 790 & 1,484 \\
		$N_{\mathrm{cells}}$ & 126k & 369k & 612k & 2,261k & 4,083k & 16,662k &
		31,221k \\
		\bottomrule
		\end{tabular}
	\end{small}\vspace{0.2cm}
	\caption{Number of subdomains and total cells in the background mesh for the
	cases considered in the weak scaling tests of Figure~\ref{fig:ws_results}. For
	each case, local mesh size, given by $n_{\mathrm{cells}}$, remains
	quasi-constant with the number of subdomains $P$.} 
	\label{tab:partitions}
\end{table}

In weak scaling tests, we monitor wall clock times spent in the main phases of
(i) the Ag\ac{fem} method and (ii) the linear solver. We additionally get (iii)
the number of \gamg{} solver iterations. As finding the optimal mesh for each
$\gamma_i$, $i > 1$ is an iterative \ac{amr} process, we only report these
quantities for the optimal mesh (last iteration). In the \ac{fe} simulation
loop, the starting control point is right after generating and partitioning the
optimal mesh. From here, and following the order of the simulation pipeline, we
report the time consumed in relevant Ag\ac{fem}-related phases
\begin{enumerate}
	\item parallel cell aggregation, i.e. generation of the distributed-memory root
	cell map $R$ (Section~\ref{sec:cell-aggr}),
  \item import data from missing remote root (and their coarse neighbour) cells,
  i.e. import $\T_h^{\RG}$ (Section~\ref{sec:dm-impl}),
	\item setup of the distributed $\V_h^\std$ space
	(Section~\ref{sec:std-fe-spaces}), accounting for hanging \ac{dof}
	constraints,
	\item setup of the distributed $\V_h^\ag$ space on top of $\V_h^\std$
	(Sections~\ref{sec:ag-fe-spaces} and~\ref{sec:dm-impl}), with mixed
	hanging and aggregation \ac{dof} constraints.
\end{enumerate}
This is followed (and completed) by gathering the time spent in the linear
solver setup and run stages, as well as the number of solver iterations needed
to find the approximate solution to the problem on the optimal mesh. The
convergence criterion is the same as the one of the previous section, i.e.\ $\|
\mathbf{r} \|_2/ \| \mathbf{b} \|_2 < 10^{-9}$.

To allocate the MPI tasks in the MN-IV supercomputer, we resort to the default
task placement policy of Intel MPI (v2018.4.057) with partially filled nodes.
For each point of the test, the number of nodes $N^i$ is selected as $N^i =
\left\lceil P^i/48 \right\rceil$, where $\lceil \cdot \rceil$ is the
\textit{ceiling} function; given a real number $x$, $\lceil x \rceil$ is
the smallest integer more than or equal to $x$. If $P^i$ is not multiple of 48,
the placement policy fully populates the first $N-1$ nodes with 48 MPI tasks
per node; the remaining $P^i - 48(N-1)$ MPI tasks are mapped to the last node.

Figure~\ref{fig:ws_results} gathers all the quantities surveyed in weak scaling
tests. All main phases of the $h$-Ag\ac{fem} method exhibit remarkable
scalability (Figures~\ref{fig:agfem_phases_popcorn}
and~\ref{fig:agfem_phases_hollow}). The results are also qualitatively similar
for both geometries. Concerning solver performance in
Figures~\ref{fig:solver_phases_popcorn}-\ref{fig:cg_iterations_hollow}, although
times and iterations do not scale as well as $h$-Ag\ac{fem}-specific phases,
results are still sound. Different system matrix conditioning could explain the
slight differences between the two problems in solver performance. In any case,
growth rate is mild, compared to growth of problem size. For instance, in the
Hollow-shock example, total solver wall clock time (setup plus run) scales from
0.55 to 2.34 s, while the problem size scales from 126,232 to 16,619,828 cells.
This means the total solver time increases by a factor of $4.3x$, whereas the
problem size by a factor of $131.7x$. On the other hand, solver degradation is
likely not fully attributed to $h$-Ag\ac{fem}; see, e.g.~the results
in~\cite{Verdugo2019}, showing that \gamg{} loses parallel efficiency even when
dealing with body-fitted meshes.

\begin{figure}[!h]
  \centering
  \begin{subfigure}[t]{0.95\textwidth}
    \includegraphics[width=\textwidth]{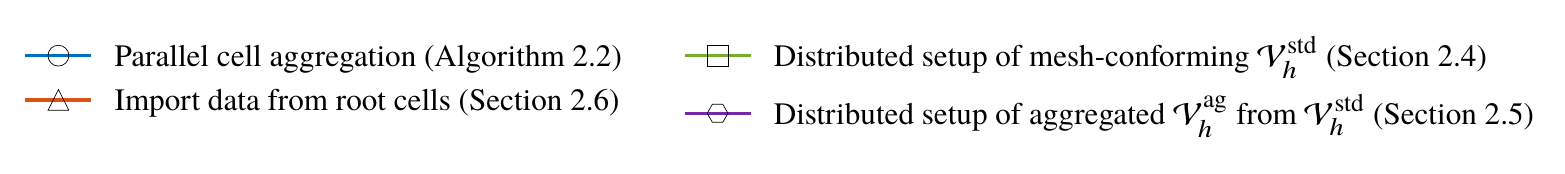}
  \end{subfigure}
  \\ \vspace{-0.25cm}
  \begin{subfigure}[t]{0.49\textwidth}
    \includegraphics[width=\textwidth]{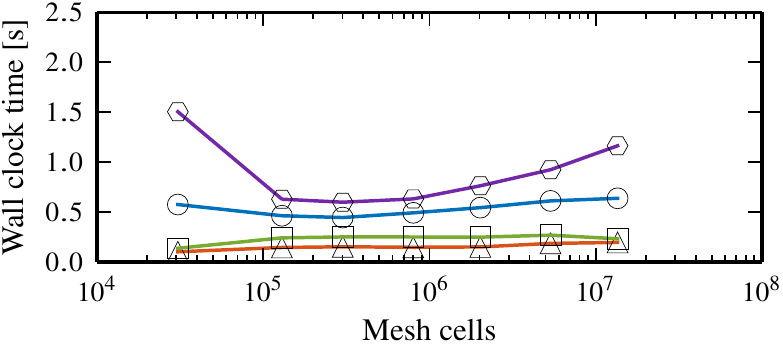}
    \caption{Popcorn-Fichera LB $\eta_0 = 0.25$}
    \label{fig:agfem_phases_popcorn}
  \end{subfigure}
  \begin{subfigure}[t]{0.49\textwidth}
    \includegraphics[width=\textwidth]{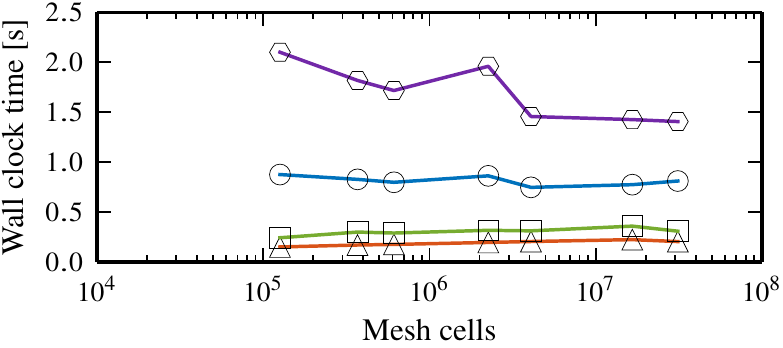}
    \caption{Hollow-shock LB $\eta_0 = 0.25$}
    \label{fig:agfem_phases_hollow}
  \end{subfigure}
  \\
  \begin{subfigure}[t]{0.95\textwidth}
    \includegraphics[width=\textwidth, trim=0.0cm 0.1cm 0.0cm 0.3cm, clip]{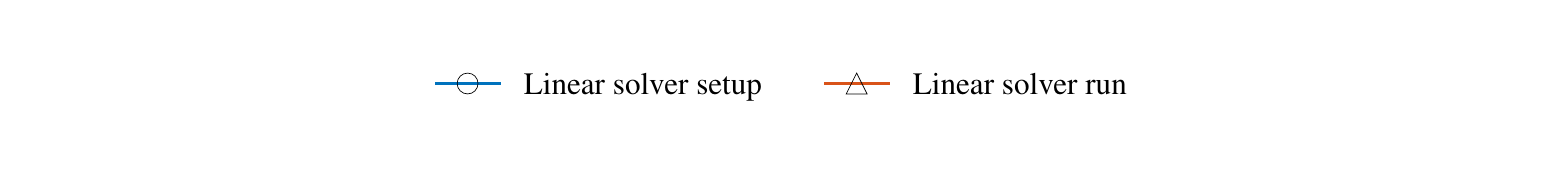}
  \end{subfigure}
  \\ \vspace{-0.5cm}
  \begin{subfigure}[t]{0.49\textwidth}
    \includegraphics[width=\textwidth]{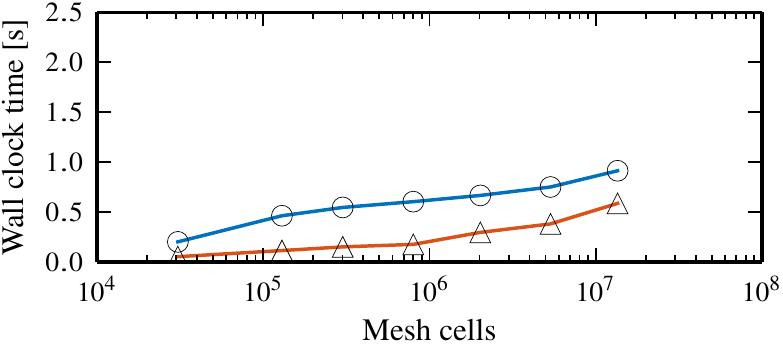}
    \caption{Popcorn-Fichera LB $\eta_0 = 0.25$}
    \label{fig:solver_phases_popcorn}
  \end{subfigure}
  \begin{subfigure}[t]{0.49\textwidth}
    \includegraphics[width=\textwidth]{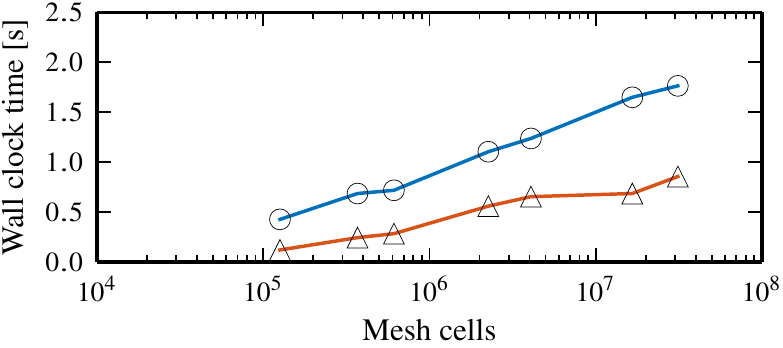}
    \caption{Hollow-shock LB $\eta_0 = 0.25$}
    \label{fig:solver_phases_hollow}
  \end{subfigure}
  \\ \vspace{0.5cm}
  \begin{subfigure}[t]{0.49\textwidth}
    \includegraphics[width=\textwidth]{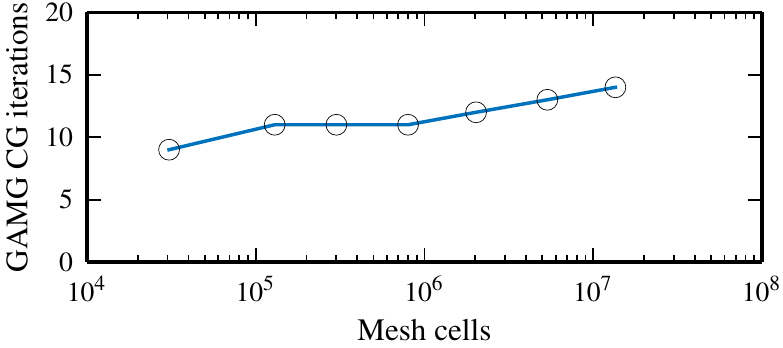}
    \caption{Popcorn-Fichera LB $\eta_0 = 0.25$}
    \label{fig:cg_iterations_popcorn}
  \end{subfigure}
  \begin{subfigure}[t]{0.49\textwidth}
    \includegraphics[width=\textwidth]{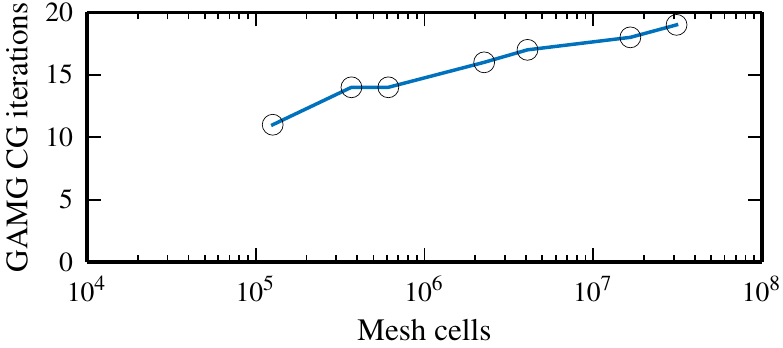}
    \caption{Hollow-shock LB $\eta_0 = 0.25$}
    \label{fig:cg_iterations_hollow}
  \end{subfigure}
  \caption{AgFEM weak scaling tests up to 1,484 MPI tasks, as specified in Table~\ref{tab:partitions}.}
  \label{fig:ws_results}
\end{figure}

\section{Conclusions}\label{sec:conclusions}

In this work, we have introduced the aggregated finite element method on
parallel adaptive tree-based meshes, referred to as $h$-Ag\ac{fem}. The main
difficulty is to establish how to combine hanging \ac{dof} constraints, arising
from mesh non-conformity, with aggregation ones, which are needed to get rid of
the small cut cell problem,
in the definition of the discrete extension operator from well-posed to ill-posed \acp{dof}.
We have followed a two-level strategy, grounded on
building the aggregated \ac{fe} space on top of an existing conforming \ac{fe}
space.

As main contributions of the paper, we have shown that (a) our approach allows
one to define a unified Ag\ac{fe} space accounting for both type of constraints,
without circular constraint dependencies; the key point is to mark as ill-posed
\acp{dof} those without local support in a well-posed cell. We have also
described how, (b) by carefully extending the layer of ghost cells, a
distributed-memory version of $h$-Ag\ac{fem} can be easily incorporated into
existing large-scale \ac{fe} codes. With numerical experimentation on the
Poisson problem, we have studied the behaviour of $h$-Ag\ac{fem}. It (c) enjoys
the same benefits of standard $h$-\ac{fem} on body-fitted meshes. In particular,
it restores optimal rates of convergence, implied by order of approximation
alone, and it is amenable to standard mesh optimality criteria. Likewise, it
also (d) inherits good properties from Ag\ac{fem} on uniform meshes, above all
robustness with respect to cut location. We have also demonstrated (e) good
parallel performance of a distributed-memory implementation of $h$-Ag\ac{fem};
the main outcome is that it can efficiently exploit well-known \ac{amg}
preconditioners available in, e.g.\ \petsc. Finally, we have (f) carried out a
complete numerical analysis that supports the design of the method and 
the  numerical results. 

We have successfully managed to bridge unfitted methods and parallel
non-conforming tree-based meshes for the first time. $h$-Ag\ac{fem} has
the potential to grow and tackle large-scale multi-phase and multi-physics
\ac{fe} applications on arbitrarily complex geometries, aided by functional and
geometrical error-driven mesh adaptation. As future work, it also remains to
extend $h$-Ag\ac{fem} to high-order \ac{fe} approximations and, more generally,
$hp$-adaptivity.

\appendix
\section{Derivation of the Ag\ac{fe} space $\V_h^\ag$}\label{appendix:agfespace}

In this appendix, our goal is to show that any constrained \ac{dof} $\sigma \in
\Sigma^\C$ of the Ag\ac{fe} space $\V_h^\ag$ given in ~\eqref{eq:ag-fe-space}, can
be resolved with \emph{direct} constraints. This means that it is composed by
linear constraints of the same form as those in \eqref{eq:std-fe-space}, i.e.~in
terms of well-posed free \acp{dof}, only. For this purpose, we go over each subset
of $\Sigma^\C$ and characterise the subsets of $\Sigma^{\wp,\F}$ constraining
them, as well as the coefficients of the linear constraints. We also argue that
the resulting constraint dependency graph, drawn in
Figure~\ref{fig:agfem_constr_graph}, has no cyclic constraint dependencies. The
discussion leads to the definition of an aggregated \ac{fe} space $\V_h^\ag$ that
is a subspace of $\V_h^\std$ with the same structure, i.e.~restricted with linear
constraints.

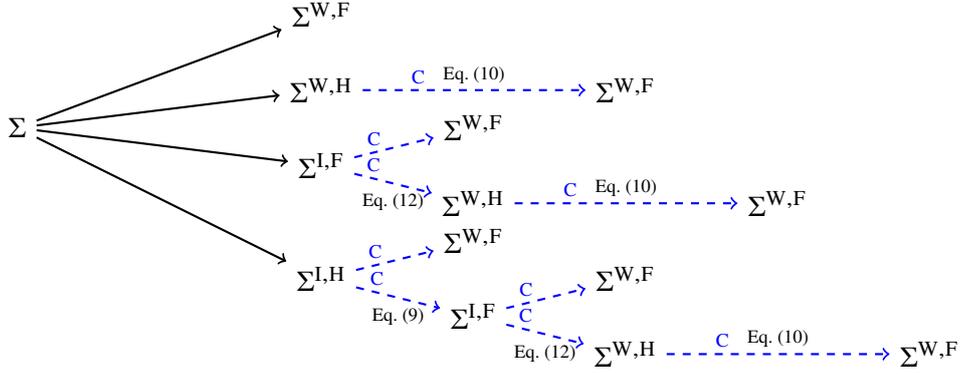
\begin{figure}[ht!]
  \centering
  \begin{tikzpicture}
   
    \node (S)               at ( 0,   0) {$\Sigma$}        ;

    \node (SWF)             at ( 4, 1.5) {$\Sigma^{\wp,\F}$};
    \node (SWH)             at ( 4, 0.5) {$\Sigma^{\wp,\H}$};
    \node (SIF)             at ( 4,-0.5) {$\Sigma^{\ip,\F}$};
    \node (SIH)             at ( 4,  -2) {$\Sigma^{\ip,\H}$};

    \node (SIF-SWF)         at ( 6,   0) {$\Sigma^{\wp,\F}$};
    \node (SIF-SWH)         at ( 6,  -1) {$\Sigma^{\wp,\H}$};
    \node (SIH-SWF)         at ( 6,-1.5) {$\Sigma^{\wp,\F}$};
    \node (SIH-SIF)         at ( 6,-2.5) {$\Sigma^{\ip,\F}$};

    \node (SWH-SWF)         at ( 8, 0.5) {$\Sigma^{\wp,\F}$};
    \node (SIF-SWH-SWF)     at (10,  -1) {$\Sigma^{\wp,\F}$};
    \node (SIH-SIF-SWF)     at ( 8,  -2) {$\Sigma^{\wp,\F}$};
    \node (SIH-SIF-SWH)     at ( 8,  -3) {$\Sigma^{\wp,\H}$};

    \node (SIH-SIF-SWH-SWF) at (12,-3) {$\Sigma^{\wp,\F}$};

    \draw[thick,->] (S) -- (SWF); \draw[thick,->] (S) -- (SWH);
    \draw[thick,->] (S) -- (SIF); \draw[thick,->] (S) -- (SIH);
    
    \draw[dashed,blue,thick,->] (SWH) -- node[above, near start, yshift=-0.05cm]
    {\tiny $\C$} node[black, above, yshift=-0.05cm] {\tiny
    Eq.~\eqref{eq:prop-well-posed-hanging}} (SWH-SWF);
    \draw[dashed,blue,thick,->] (SIF) -- node[above, near start, yshift=-0.05cm]
    {\tiny $\C$} (SIF-SWF); \draw[dashed,blue,thick,->] (SIF) -- node[above,
    near start, yshift=-0.05cm] {\tiny $\C$} node[black, above, yshift=-0.5cm]
    {\tiny Eq.~\eqref{eq:agfem_constraining}} (SIF-SWH);
    \draw[dashed,blue,thick,->] (SIH) -- node[above, near start, yshift=-0.05cm]
    {\tiny $\C$} (SIH-SWF); \draw[dashed,blue,thick,->] (SIH) -- node[above,
    near start, yshift=-0.05cm] {\tiny $\C$} node[black, above, yshift=-0.5cm]
    {\tiny Eq.~\eqref{eq:prop-hanging}} (SIH-SIF);

    \draw[dashed,blue,thick,->] (SIF-SWH) -- node[above, near start,
    yshift=-0.05cm] {\tiny $\C$} node[black, above, yshift=-0.05cm] {\tiny
    Eq.~\eqref{eq:prop-well-posed-hanging}} (SIF-SWH-SWF);
    \draw[dashed,blue,thick,->] (SIH-SIF) -- node[above, near start,
    yshift=-0.05cm] {\tiny $\C$} (SIH-SIF-SWF); \draw[dashed,blue,thick,->]
    (SIH-SIF) -- node[above, near start, yshift=-0.05cm] {\tiny $\C$}
    node[black, above, yshift=-0.5cm] {\tiny Eq.~\eqref{eq:agfem_constraining}}
    (SIH-SIF-SWH);

    \draw[dashed,blue,thick,->] (SIH-SIF-SWH) -- node[above, near start,
    yshift=-0.05cm] {\tiny $\C$}  node[black, above, yshift=-0.05cm] {\tiny
    Eq.~\eqref{eq:prop-well-posed-hanging}} (SIH-SIF-SWH-SWF);

\end{tikzpicture}
  \caption{Constraint dependency graph of the Ag\ac{fe} space $\V_h^\ag$. The
  set of global \acp{dof} $\Sigma$ is partitioned into $\{ \Sigma^{\wp,\F},
  \Sigma^{\wp,\H}, \Sigma^{\ip,\F}, \Sigma^{\ip,\H} \}$. Subsets
  $\Sigma^{\wp,\H}$, $\Sigma^{\ip,\F}$ and $\Sigma^{\ip,\H}$ are all constrained
  by $\Sigma^{\wp,\F}$ with a dependency graph represented by dashed blue edges
  marked with a $\C$. Dashed blue edges link a constrained subset with the
  subsets where its masters belong to. We observe that the graph has no cycles.}
	\label{fig:agfem_constr_graph}
\end{figure}

According to this, given $\sigma \in \Sigma^\C$,
\begin{enumerate}
  \item if \underline{$\sigma \in \Sigma^{\wp,\H}$}, then $\M^\H_\sigma$ is formed
  by \acp{dof} located in \acp{vef} of coarser neighbour cells around $\sigma$,
  see Section~\ref{sec:std-fe-spaces}. As $\T_h$ meets the 2:1 balance condition,
  constraining \acp{dof} of hanging \acp{dof} are free \acp{dof}~\cite[Proposition
  4.1]{Badia2019b}, i.e.~
  \begin{equation}\label{eq:prop-hanging}
    \sigma \in \Sigma^\H \Rightarrow \M^\H_\sigma \subset \Sigma^\F.
  \end{equation}
  Recalling Definition~\ref{def:free-well-posed} (ii), it follows that
  master \acp{dof} of $\sigma$ are necessarily contained in the set of
  well-posed free \acp{dof}, i.e.~
  \begin{equation}\label{eq:prop-well-posed-hanging}
    \sigma \in \Sigma^{\wp,\H} \Rightarrow \M^\H_\sigma \subset \Sigma^{\wp,\F}.
  \end{equation}
  Therefore, linear constraints of $\sigma \in \Sigma^{\wp,\H}$ remain unchanged
  in the new Ag\ac{fe} space.
  \item If \underline{$\sigma \in \Sigma^{\ip,\F}$}, then we assume that we have
  composed the root cell map $R : \T_h^\act \to \T_h^\wp$, introduced in
  Section~\ref{sec:cell-aggr}, with a map between ill-posed free \acp{dof}
  $\Sigma^{\ip,\F}$ and ill-posed cells $\T_h^\ip$. In other words, we assign
  first each ill-posed free \ac{dof} to one of its surrounding ill-posed cells.
  The chosen cell is then mapped onto a well-posed cell via $R$. Thus, the outcome
  of this composition is a map $K : \Sigma^{\ip,\F} \to \T_h^\wp$, that assigns an
  ill-posed free \ac{dof} to a well-posed cell; see formal definitions in,
  e.g.~\cite{Verdugo2019,Badia2018}. Given $\sigma \in \Sigma^{\ip,\F}$, let us
  denote by $\M^{\A\A}_\sigma$ the subset of \acp{dof} $\tilde{\sigma}$ located in
  $K(\sigma)$, such that $\phi^{\tilde{\sigma}} (\x^{\sigma}) \neq 0$. We refer to
  $\M^{\A\A}_\sigma$ as the set of ``direct'' Ag\ac{fem} master \acp{dof} of
  $\sigma \in \Sigma^{\ip,\F}$. As usual in Ag\ac{fe} methods, given $v_h \in
  \V_h^\std$ and $\sigma \in \Sigma^{\ip,\F}$, we enforce the constraint
  \begin{equation}
    v^\sigma_h = \sum_{\tilde{\sigma} \in \M^{\A\A}_\sigma}
    C^{\A\A}_{\sigma\tilde{\sigma}} v_h^{\tilde{\sigma}}, 
    \quad \text{with} \ C^{\A\A}_{\sigma\tilde{\sigma}} \doteq 
    \phi^{\tilde{\sigma}} (\x^{\sigma}),
    \label{eq:agfem_constraints}
  \end{equation}
  that is, we linearly extrapolate the nodal value of an ill-posed \ac{dof} with
  the values at the local \acp{dof} of its root cell.  In general,
  $\M^{\A\A}_\sigma$ is composed of both free and hanging \acp{dof}, i.e.~some
  \acp{dof} in the root cell can be hanging; the latter are not master
  \acp{dof}, in the strict sense, and we need to remove them, i.e.~rewrite
  \eqref{eq:agfem_constraints} in terms of well-posed free \acp{dof},
  only. For that purpose, we introduce the partition $\M^{\A\A}_\sigma =
  \{\M^{\A\F}_\sigma, \M^{\A\H}_\sigma \}$, with $\M^{\A\F}_\sigma \doteq
  \M^{\A\A}_\sigma \cap \Sigma^{\F} $, $\M^{\A\H}_\sigma \doteq \M^{\A\A}_\sigma
  \cap \Sigma^{\H}$. Since the image of $K$ is in $\T_h^\wp$, it is clear that
  $\M^{\A\F}_\sigma \subset \Sigma^{\wp,\F}$ and $\M^{\A\H}_\sigma \subset
  \Sigma^{\wp,\H}$. We also have that
  \begin{equation}
    \sigma \in \Sigma^{\ip,\F} \Rightarrow \M^{\A\A}_\sigma \subset 
    \Sigma^{\wp,\F} \cup \Sigma^{\wp,\H}.
    \label{eq:agfem_constraining}
  \end{equation}
  Recalling the first case, i.e.~\underline{$\sigma \in \Sigma^{\wp,\H}$}, the
  set of \acp{dof} that are masters of $\M^{\A\H}_\sigma$ is given by
  $\bigcup_{\sigma' \in \M^{\A\H}_\sigma} \M^\H_{\sigma'}$ and, by
  \eqref{eq:prop-well-posed-hanging}, it is included in
  $\Sigma^{\wp,\F}$. If the previous property didn't hold, then $\Sigma^{\ip,\F}
  \cap \left( \bigcup_{\sigma' \in \M^{\A\H}_\sigma} \M^\H_{\sigma'} \right)
  \neq \emptyset$ and it could be possible that $\sigma \in \bigcup_{\sigma' \in
  \M^{\A\H}_\sigma} \M^\H_{\sigma'}$, i.e.~$\sigma$ could (circularly)
  constrain itself, as in the situation depicted in Figure~\ref{fig:spaces_b}.
  
  Hence, the ``true'' set of master \acp{dof} of $\sigma \in \Sigma^{\ip,\F}$ is
  $\M^\A_\sigma \doteq \M^{\A\F}_\sigma \cup \left( \bigcup_{\sigma' \in
  \M^{\A\H}_\sigma} \M^\H_{\sigma'} \right)$; note that the two set members of
  $\M^\A_\sigma$ are not necessarily disjoint, but $\M^\A_\sigma \subset
  \Sigma^{\wp,\F}$. Besides, recalling that hanging \acp{dof} are constrained by
  free \acp{dof} on top of \acp{vef} of coarser neighbour cells, $\M^\A_\sigma$
  are composed of \acp{dof} located in root cells and (neighbouring) coarser
  cells around them.
  
  After cancelling hanging \acp{dof}, we can derive an analogous expression to
  \eqref{eq:agfem_constraints}, in terms of well-posed free \acp{dof}
  only. The value of the Ag\ac{fem} constraint, for $\sigma \in \Sigma^{\ip,\F}$
  and $\tilde{\sigma} \in \M^\A_\sigma$, is
  \begin{equation}
  C^\A_{\sigma\tilde{\sigma}} \doteq
    \left\lbrace
    \begin{array}{ll}
      C^{\A\A}_{\sigma\tilde{\sigma}} & \text{if} \ \tilde{\sigma} \in
      \M^{\A\F}_\sigma, \ \text{only} \\
      C^{\A\A}_{\sigma\tilde{\sigma}} + \sum_{ \left( \sigma' \in
      \M^{\A\H}_\sigma \ \text{s.t.} \ \tilde \sigma \in \M^\H_{\sigma'} \right)}
      C^{\A\A}_{\sigma\sigma'} C^\H_{\sigma'\tilde{\sigma}} & \text{if} \
      \tilde{\sigma} \in \M^{\A\F}_\sigma \cap \left(
      \bigcup_{\sigma' \in \M^{\A\H}_\sigma} \M^\H_{\sigma'} \right) \\
      \sum_{ \left( \sigma' \in \M^{\A\H}_\sigma \ \text{s.t.} \ \tilde \sigma
      \in \M^\H_{\sigma'} \right)} C^{\A\A}_{\sigma\sigma'}
      C^\H_{\sigma'\tilde{\sigma}} & \text{if} \ \tilde{\sigma} \in 
      \bigcup_{\sigma' \in \M^{\A\H}_\sigma} \M^\H_{\sigma'}, \ \text{only} .
    \end{array}
    \right.\label{eq:agfem_coefficients}
  \end{equation}
  We refer to Figure~\ref{fig:spaces_c} for an illustration of the three
  types of $\tilde{\sigma} \in \M_\sigma^\A$ in~\eqref{eq:agfem_coefficients}.
  \begin{figure}[ht!]
    \centering
    \includegraphics[width=0.24\textwidth]{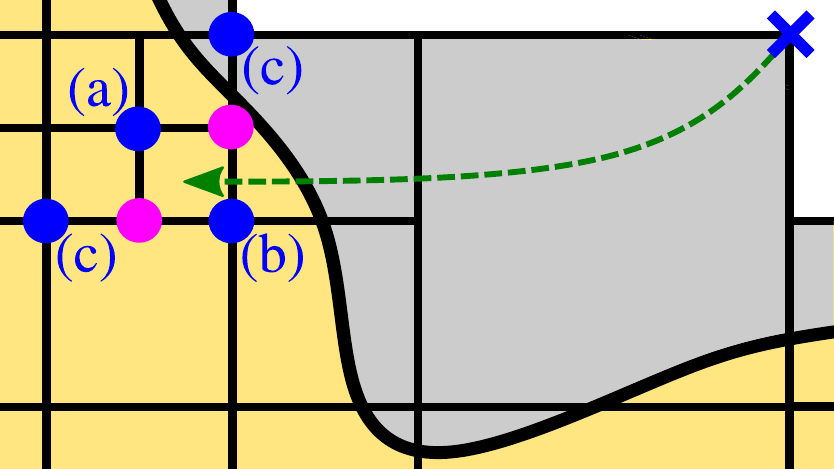}
    \caption{Close-up of Figure~\ref{fig:spaces_a}. Assuming that the top right
    ill-posed \ac{dof} is mapped to the well-posed cell pointed by the dashed
    arrow, we mark with letters and classify all \acp{dof} $\tilde{\sigma} \in
    \M_\sigma^\A$, as they are distinguished in~\eqref{eq:agfem_coefficients}.
    In this sense, (a) shows $\tilde{\sigma} \in \M_\sigma^{\A\F}$, only; (b)
    shows $\tilde{\sigma} \in \M_\sigma^{\A\F} \cap \left(\bigcup_{\sigma' \in
    \M^{\A\H}_\sigma} \M^\H_{\sigma'} \right)$; (c) shows $\tilde{\sigma} \in
    \bigcup_{\sigma' \in \M^{\A\H}_\sigma} \M^\H_{\sigma'}$, only. We observe
    that \acp{dof} (c) are only in neighbouring coarser
    cells.}\label{fig:spaces_c}
  \end{figure}

  \item If \underline{$\sigma \in \Sigma^{\ip,\H}$}, then $\sigma$ cannot be
  constrained as in the previous case, i.e.~hanging \ac{dof} constraints have to
  be imposed first, to preserve conformity. According to this, $\sigma$ can be
  constrained by either well-posed or ill-posed free \acp{dof},
  i.e.~$\M^\H_\sigma \subset \Sigma^{\wp,\F} \cup \Sigma^{\ip,\F}$; this is an
  immediate consequence of \eqref{eq:prop-hanging}. If we consider now
  a partition of $\M^\H_\sigma$ into well-posed and ill-posed master \acp{dof}
  and use case \underline{$\sigma \in \Sigma^{\ip,\F}$} to remove ill-posed
  master \acp{dof}, we deduce that
  \[
    \M^\H_\sigma = \left( \M^\H_\sigma \cap \Sigma^{\wp,\F} \right) \cup
    \left( \M^\H_\sigma \cap \Sigma^{\ip,\F} \right) = \left( \M^\H_\sigma
    \cap \Sigma^{\wp,\F} \right) \cup \left( \bigcup_{\sigma' \in
    \M^\H_\sigma \cap \Sigma^{\ip,\F} } \M^\A_{\sigma'} \right) \subset
    \Sigma^{\wp,\F},
  \]
  i.e.~we can compute the constraints in terms of well-posed free \acp{dof}
  only; again the two sets in the right-hand side are not necessarily disjoint.
  After cancelling the Ag\ac{fem} constraints of $\sigma' \in \M^\H_\sigma \cap
  \Sigma^{\ip,\F}$, the constraint coefficient for $\sigma \in \Sigma^{\ip,\H}$
  and $\sigma' \in \M^\H_\sigma$ becomes
  \[
    C^{\H\A}_{\sigma\sigma'} \doteq
      \left\lbrace
      \begin{array}{ll}
        \C^\H_{\sigma\sigma'} & \text{if} \ \sigma' \in \left( \M^\H_\sigma
        \cap \Sigma^{\wp,\F} \right), \ \text{only} \\
        \C^\H_{\sigma\sigma'} + \sum_{ \left( \tilde{\sigma} \in
        \M^\A_\sigma \ \text{s.t.} \ \sigma' \in \M^\H_{\tilde{\sigma}} \right)}
        C^\A_{\sigma\tilde{\sigma}} C^\H_{\tilde{\sigma}\sigma'} & \text{if} \
        \left( \M^\H_\sigma \cap \Sigma^{\wp,\F} \right) \cap \left( \bigcup_{
        \sigma' \in \M^\H_\sigma \cap \Sigma^{\ip,\F} } \M^\A_{\sigma'} \right) \\
        \sum_{ \left( \tilde{\sigma} \in \M^\A_\sigma \ \text{s.t.} \ \sigma' \in
        \M^\H_{\tilde{\sigma}} \right)} C^\A_{\sigma\tilde{\sigma}}
        C^\H_{\tilde{\sigma}\sigma'} & \text{otherwise}.
      \end{array}
      \right.
  \]
\end{enumerate}

The last step to derive the Ag\ac{fe} space is to gather the previous cases,
combining hanging and aggregation \ac{dof} constraints, into a unified form
equivalent to \eqref{eq:agfem_constraints}. Given $\sigma \in
\Sigma^\C$, the set of master \acp{dof} is
\begin{equation}
\M_\sigma\doteq
	\left\lbrace
	\begin{array}{ll}
		\M^\H_\sigma & \text{if} \ \sigma \in \Sigma^{\wp,\H} \\
		\M^\A_\sigma & \text{if} \ \sigma \in \Sigma^{\ip,\F} \\
		\left( \M^\H_\sigma \cap \Sigma^{\wp,\F} \right) \cup \left(
    \bigcup_{\sigma' \in \M^\H_\sigma \cap \Sigma^{\ip,\F} } \M^\A_{\sigma'}
    \right) & \text{if} \ \sigma \in \Sigma^{\ip,\H}.
	\end{array}
  \right.
\label{eq:ag-fe-masters}
\end{equation}
By definition, $\M_\sigma \subset \Sigma^{\wp,\F}$, for all $\sigma \in
\Sigma^\C$, i.e.~all constraints can be solved by free well-posed \acp{dof} and,
thus, there are no cyclic constraint dependencies; see also the constraint
dependency graph represented in Figure~\ref{fig:agfem_constr_graph}. On the
other hand, the constraint coefficient for $\sigma \in \Sigma^\C$ and $\sigma'
\in \M_\sigma$ is
\begin{equation}
C_{\sigma\sigma'} \doteq
  \def\arraystretch{1.2}
	\left\lbrace
	\begin{array}{ll}
		\C^\H_{\sigma\sigma'} & \text{if} \ \sigma \in \Sigma^{\wp,\H} \\
		\C^\A_{\sigma\sigma'} & \text{if} \ \sigma \in \Sigma^{\ip,\F} \\
		\C^{\H\A}_{\sigma\sigma'} & \text{if} \ \sigma \in \Sigma^{\ip,\H}.
	\end{array}
	\right.
  \label{eq:ag-fe-constraints}
\end{equation}
With these notations, the (sequential) \emph{aggregated} or \emph{ag.}~\ac{fe}
space $\V_h^\ag$ obeys to the form stated in ~\eqref{eq:ag-fe-space}.

\section{Numerical analysis}\label{appendix:proofs}

In this appendix, we prove that both the condition number of (a) the mass matrix
associated to the Ag\ac{fe} space defined in \eqref{eq:ag-fe-space} and
(b) the linear system arising from \eqref{eq:weak-PoissonEq} are
bounded. The bounds do not depend on the cut location (but they do depend on the
well-posedness threshold $\eta_0$). We use the notation $A \lesssim B$ (resp. 
$A \gtrsim B$) to represent $A \leq C B$ (resp. $A \geq CB$) for a positive 
constant $C > 0$ independent of the interface-mesh intersection or the mesh 
cells sizes.

\subsection{Mass matrix condition number}
In order to bound the condition number of the mass matrix, we seek to show the
equivalence, for functions in $\V_h^\ag$, between the $L^2(\Omega)$-norm and the
Euclidean norm of well-posed free \acp{dof}. We devote the next paragraphs to
introduce necessary definitions and preliminary results. Given $u_h \in \V_h^\ag$,
let us denote the nodal vector of well-posed free \acp{dof} by $\u$. For a given
$\cell \in \T_h$ and \ac{vef} $f$, the cell- or \ac{vef}-wise coordinate vector is
represented with $\u_\cell$ or $\u_f$ and its characteristic sizes by $h_\cell$
or $h_f$. First, we rely on the maximum and minimum eigenvalues of the local
mass matrix in the physical cell $\cell$ or any of its \acp{vef} $f \in
\mathcal{F}_{\cell}$:
\begin{equation}\label{eq:mass-spectre}
  \lambda_{\min} h_X^{d_X} \| \u_{X} \|_2^2 \leq \| u_h \|_{L^2(X)}^2 \leq 
  \lambda_{\max} h_X^{d_X} \| \u_{X} \|_2^2, \quad \text{for} \ u_h \in 
  \mathcal{V}(\cell),
\end{equation}
with $X = \cell$ or $X = f \in \mathcal{F}_{\cell}$ and $\| \cdot \|_2$ denoting
the Euclidean norm. The values of $\lambda_{\min}$, $\lambda_{\max} > 0$ only
depend on the order of the \ac{fe} space and can be computed for different
orders on $n$-cubes or $n$-simplices~\cite{elman2014finite}. By combining
\eqref{eq:mass-spectre} for $\cell$ and one of its \acp{vef}, we deduce
the bound
\begin{equation}\label{eq:l2-f-bound}
  \| u_h \|_{L^2(\cell)}^2 \gtrsim h_f^{d-d_f} \| u_h \|_{L^2(f)}^2 > 0, 
  \quad \text{for} \ u_h \in \mathcal{V}(\cell), \ f \in \mathcal{F}_{\cell}.
\end{equation}
We observe that \eqref{eq:l2-f-bound} can be applied to any $\cell \in
\T_h^\wp$ and corresponding \acp{vef}, because we are integrating on the whole
objects. If we consider integration on the cut portion of the cell $\Omega \cap
\cell$, \eqref{eq:mass-spectre} also holds, up to a positive constant
that depends on the well-posedness threshold $\eta_0$. This is a consequence of
the following result.
\begin{lemma}\label{lem:l2-norm-cut-bound}
  Given a well-posed cell $\cell \in \T_h^\wp$ and $u_h \in \mathcal{V}(\cell)$,
  there exists $C(\eta_0) > 0$, dependent on the well-posedness threshold
  $\eta_0$, such that $\| u_h \|_{L^2(\Omega \cap \cell)}^2 \geq C(\eta_0) \|
  u_h \|_{L^2(\cell)}^2$.
\end{lemma}
\begin{proof}
  Since we consider a well-posedness threshold $0 < \eta_0 \leq 1$, any cell
  $\cell \in \T_h^\wp$ can be either (i) (full) interior or (ii) cut. For case
  (i), the bound trivially holds. For case (ii), given any polynomial defined in
  the cell, in particular, any shape function, we must have that $\int_{\Omega
  \cap \cell} {p(x)}^2 \geq C(\eta_0) \int_{\cell} {p(x)}^2 > 0$, for a
  bounded, strictly positive, constant $C(\eta_0)$ that depends on $\eta_0$. If
  this were not the case, then we would have that $p(x)$ vanishes in $\Omega
  \cap \cell$, with $\left| \Omega \cap \cell \right| \neq 0$. As $p(x)$ is a
  polynomial, the only possibility is that $p \equiv 0$ in $\cell$. Hence, the
  bound also holds for case (ii).
\end{proof}

\begin{remark}
  We observe that we generally do not have an analogous bound to that of
  Lemma~\ref{lem:l2-norm-cut-bound} for $f \in \mathcal{F}_\cell$, with $\cell
  \in \T_h^\wp$, because $| f \cap \Omega |$ can be arbitrarily small. 
\end{remark}

Now, we consider the partition $\Sigma^{\wp,\F} \doteq \{\Sigma_{\rm
int}^{\wp,\F}, \Sigma_{\rm ext}^{\wp,\F} \}$, where $\Sigma_{\rm int}^{\wp,\F}$
groups \acp{dof} that satisfy Definition~\ref{def:free-well-posed} (i) and
$\Sigma_{\rm ext}^{\wp,\F}$ those that satisfy
Definition~\ref{def:free-well-posed} (ii). We prove next two lemmas that, along
with Lemma~\ref{lem:l2-norm-cut-bound}, allow one to compute a lower bound of the
$L^2(\Omega)$-norm of functions in $\V_h^\ag$ by the Euclidean norm of \acp{dof}
in $\Sigma_{\rm ext}^{\wp,\F}$. Letting $\Sigma_f$ denote the set of local
\acp{dof} in $f \in \mathcal{F}_\cell$, we show, in the first lemma, that for any
$\sigma \in \Sigma_{\rm ext}^{\wp,\F}$, located atop a coarse \ac{vef} $f_{\rm
C}$, we can find a hanging \ac{vef} $f_{\rm H}$ \emph{of a well-posed cell}, with
the same dimension of and owned by $f_{\rm C}$.
\begin{lemma}\label{lemma:hanging-to-owner}
  Given $\sigma \in \Sigma_{\rm ext}^{\wp,\F}$, there exists $\sigma' \in
  \Sigma^{\wp,\H}$, such that $\sigma \in \M_{\sigma'}^\H$, and there exist
  \acp{vef} $f_{\rm C} \in \cell$ and $f_{\rm H} \in \cell'$, with $\cell \in
  \T_h^\ip$ and $\cell' \in \T_h^\wp$, such that $\sigma \in
  \Sigma_{\overline{f_{\rm C}}}$, $\sigma' \in \Sigma_{\overline{f_{\rm H}}}$
  and $\mathrm{dim}(f_{\rm C}) = \mathrm{dim}(f_{\rm H})$.
\end{lemma}
\begin{proof}
  We use Figure~\ref{fig:lemma_3} to illustrate the proof. Given $\sigma \in
  \Sigma_{\rm ext}^{\wp,\F}$, by Definition~\ref{def:free-well-posed} (ii), there
  exists $\sigma' \in \Sigma^{\wp,\H}$, such that $\sigma \in \M_{\sigma'}^\H$.
  Note that $\sigma$ and $\sigma'$ are related by a nontrivial constraint, by
  definition of $\M_{\sigma'}^\H$ (see Section~\ref{sec:std-fe-spaces}). In
  addition, by recalling how hanging \acp{dof} and its constraining \acp{dof}
  are related~\cite[Proposition 3.6]{Badia2019b}, there exist $f_{\rm C} \in
  \mathcal{F}_\cell$, with $\cell \in \T_h^\ip$, and $f_{\rm H} \in
  \mathcal{F}_{\cell'}$, with $\cell' \in \T_h^\wp$, such that $f_{\rm C}$ is the
  owner \ac{vef} of $f_{\rm H}$, $\sigma \in \Sigma_{\overline{f_{\rm C}}}$ and
  $\sigma' \in \Sigma_{\overline{f_{\rm H}}}$. If $\mathrm{dim}(f_{\rm H}) =
  \mathrm{dim}(f_{\rm C})$, the result follows immediately. Otherwise, let us
  denote by $\T_h^{\mathrm{rr}(\cell)}$ the mesh resulting from applying the
  1:$2^d$ isotropic refinement rule once to $\cell$. $\cell$ and $\cell'$ only
  differ by one level of refinement, by the 2:1 balance assumption, and
  $\T_h^{\mathrm{rr}(\cell)} \cup \cell'$ forms a conforming mesh, by the
  construction of the refinement rule. It follows that $f_{\rm H}$ is one of the
  \acp{vef} on the boundary of $\T_h^{\mathrm{rr}(\cell)}$. As $f_{\rm C}$ is the
  only \ac{vef} of $\cell$ that contains $f_{\rm H}$ and the refinement rule
  implies a nontrivial partition of all \acp{vef} of $\cell$, there exists
  $f_{\rm H}' \subsetneq f_{\rm C}$, such that $\mathrm{dim}(f_{\rm H}') =
  \mathrm{dim}(f_{\rm C})$, $f_{\rm H} \subsetneq \overline{f_{\rm H}'}$ and
  $f_{\rm H}' \in \mathcal{F}_{\cell'}$. Clearly, $f_{\rm H}'$ is also a hanging
  \ac{vef} of $\T_h$, with $f_{\rm C}$ as owner and $\sigma' \in
  \Sigma_{\overline{f_{\rm H}'}}$.
\end{proof}
\begin{figure}[ht!]
  \centering
  \includegraphics[width=0.36\textwidth, trim=0.0cm 0.02cm 0.0cm 0.05cm, clip]{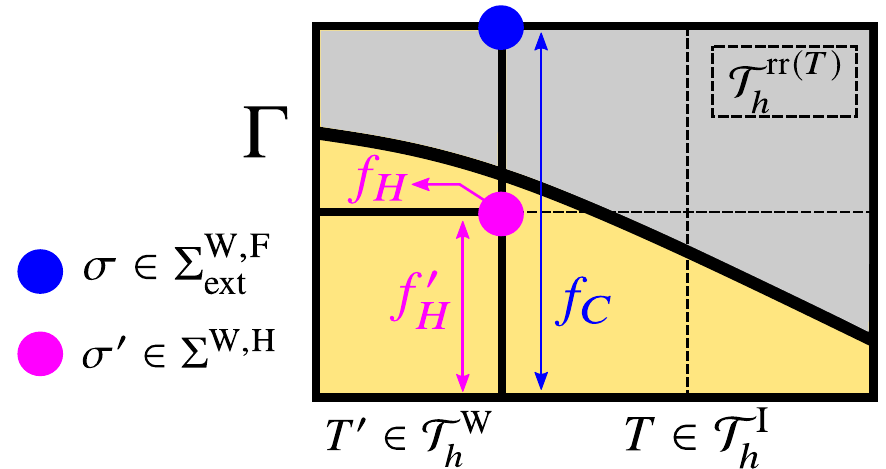}
  \caption{A 2D example to illustrate the proof of 
  Lemma~\ref{lemma:hanging-to-owner}.}
  \label{fig:lemma_3}
\end{figure}
The fact that $f_{\rm C}$ and $f_{\rm H}$ in Lemma~\ref{lemma:hanging-to-owner}
have the same dimension is key to prove the following bound.
\begin{lemma}\label{lemma:l2-hanging-to-l2-coarse-dofs}
  Given $\sigma \in \Sigma_{\rm ext}^{\wp,\F}$, atop a \ac{vef} $f_{\rm C} \in
  \mathcal{F}_\cell$, with $\cell \in \T_h^\ip$, we have the bound
  \[
    \| u_h \|_{L^2(f_{\rm H})}^2 \gtrsim  h_{f_{\rm C}}^{d_{f_{\rm C}}} 
    \| \u_{f_{\rm C}} \|_2^2, \quad \text{for} \ u_h \in \V_h^\ag,
  \]
  where $f_{\rm H}$ is a hanging \ac{vef}, owned by $f_{\rm C}$, $f_{\rm H} \in
  \mathcal{F}_{\cell}$, $\cell \in \T_h^\wp$, such that $\mathrm{dim}(f_{\rm C}) =
  \mathrm{dim}(f_{\rm H})$.
\end{lemma}
\begin{proof}
  First we see that, by Lemma~\ref{lemma:hanging-to-owner}, we can find $f_{\rm
  H}$ satisfying the hypotheses. As seen in \eqref{eq:std-fe-space} of
  Section~\ref{sec:std-fe-spaces}, we have that hanging \ac{dof} linear
  constraints, defined for \acp{dof} in $\overline{f_{\rm H}}$, lead to the
  relation $\u_{f_{\rm H}} = \mathbf{C} \u_{f_{\rm C}}$, where the coefficients
  of $\mathbf{C}$ are given by $\phi^{\sigma'} (\x^\sigma)$ with $\sigma \in
  \Sigma_{\overline{f_{\rm H}}}$ and $\sigma' \in \Sigma_{\overline{f_{\rm
  C}}}$. Coefficients $\phi^{\sigma'} (\x^\sigma)$ of $\mathbf{C}$ can be
  computed in a reference cell $\hat{\cell}$, by generating the mesh
  $\T_h^{\mathrm{rr}(\cell)}$ and evaluating the shape functions of $\hat{\cell}$
  at its nodes. Since shape functions are pointwise bounded, $\left|
  \phi^{\sigma'} (\x^\sigma) \right|$ is bounded above, independently of mesh
  size and cuts. Using the above relation, we have that
  \[
    \| u_h \|_{L^2(f_{\rm H})}^2 = \u_{f_{\rm H}}^T \mathbf{M}_{f_{\rm H}} 
    \u_{f_{\rm H}} = \u_{f_{\rm C}}^T \mathbf{C}^{T} \mathbf{M}_{f_{\rm H}} 
    \mathbf{C} \u_{f_{\rm C}} = \lambda \ \u_{f_{\rm C}}^T 
    \mathbf{M}_{f_{\rm C}} \u_{f_{\rm C}},
  \]
  where $\mathbf{M}_f$ denotes the local \ac{fe} mass matrix on \ac{vef} $f$
  and, in the last equality, we consider the generalized eigenvalue problem
  $\mathbf{C}^{T} \mathbf{M}_{f_{\rm H}} \mathbf{C} \u_{f_{\rm C}} =
  \lambda \ \mathbf{M}_{f_{\rm C}} \u_{f_{\rm C}}$. Since $\mathbf{C}^{T} 
  \mathbf{M}_{f_{\rm H}}
  \mathbf{C}$, $\mathbf{M}_{f_{\rm C}}$ are symmetric and $\mathbf{M}_{f_{\rm
  C}}$ is positive definite (due to \eqref{eq:mass-spectre}), the
  eigenvalues of the above problem are real. Moreover, the same argument in the
  proof of Lemma~\ref{lem:l2-norm-cut-bound} ensures that, if a polynomial
  vanishes in $f_{\rm H}$, it must also vanish in $f_{\rm C}$. Therefore, we
  have that the smallest eigenvalue must be strictly positive,
  i.e.~$\lambda_{\min} > 0$. It suffices to combine this result with
  \eqref{eq:mass-spectre} applied on $\mathbf{M}_{f_{\rm C}}$ to see
  that $\| u_h \|_{L^2(f_{\rm H})}^2 \gtrsim \lambda_{\min} \ 
    h_{f_{\rm C}}^{d_{f_{\rm C}}} \u_{f_{\rm C}}^T \u_{f_{\rm C}} > 0$.
\end{proof}

We need now some auxiliary definitions: Given $\sigma \in \Sigma^{\wp,\F}$, we
let $\mathcal{S}_\sigma \doteq \{ \sigma' \in \Sigma^\C : \sigma \in
\M_{\sigma'} \}$ denote the set of \acp{dof} constrained by $\sigma$ (either by
mesh nonconformity or aggregation), the global shape function associated to
$\sigma$, after solving constraints, is given by $\tilde{\phi}^\sigma \doteq
\phi^\sigma + \sum_{\sigma' \in \mathcal{S}_\sigma} C_{\sigma'\sigma}
\phi^{\sigma'}$ and we let $\T_h^\sigma \doteq \{ \cell \in \T_h : \mathrm{supp} 
(\tilde{\phi}^\sigma) \cap \cell \neq 0 \}$ denote the set of cells where
$\tilde{\phi}^\sigma$ has local support. We observe that 
\begin{equation}\label{eq:bounds-K}
  0 < K_{\min} \leq \left| C_{\sigma'\sigma} \right| \leq K_{\max},
\end{equation}
where the bounds are independent of the size of the physical cell $h_\cell$ or
cut location; this result has already been argued in
Lemma~\ref{lemma:l2-hanging-to-l2-coarse-dofs} for hanging \ac{dof} constraints
and~\cite[Lemma 5.1]{Badia2018} for aggregation \ac{dof} constraints. Apart from
that, we define $h_\sigma \doteq \max_{\cell \in \T_h^\sigma} h_\cell$. We note
that the $h_\cell$ in the definition of $h_\sigma$ differ by a bounded value,
that depends on the 2:1 0-balance restriction and the maximum aggregation
distance, i.e.~$h_\sigma = C(\cell) h_\cell$, for any $\cell \in \T_h^\sigma$.

We are now in position to show the sought-after equivalence between the $L^2$
norm of functions in $\V_h^\ag$ and the Euclidean norm of its nodal values.
\begin{proposition}\label{prop:bound-mass-matrix}
  Given $u_h \in \V_h^\ag$, the following bound holds:
  \[ \| \u \|_\sigma^2 \lesssim \| u_h \|_{L^2(\Omega)}^2 \lesssim \| \u
  \|_\sigma^2, \] 
  for $\| \u \|_\sigma^2 \doteq \sum_{\sigma \in \Sigma^{\wp,\F}} h_\sigma^d
  u_\sigma^2$, with $u_\sigma$ the nodal value of $\sigma \in \Sigma^{\wp,\F}$.
\end{proposition}
\begin{proof}
  The upper bound straightforwardly follows from considering triangular
  inequality repeatedly and the fact that $\left| C_{\sigma'\sigma} \right|$ is
  bounded above (see~\eqref{eq:bounds-K}). For the lower bound, we use the
  results above. First, we see that, by Lemma~\ref{lem:l2-norm-cut-bound},
  \[
    \| u_h \|_{L^2(\Omega)}^2 \geq \| u_h \|_{L^2(\Omega^\wp)}^2 = \sum_{\cell 
    \in \T_h^\wp} \| u_h \|_{L^2(\Omega \cap \cell)}^2 \gtrsim 
    \sum_{\cell \in \T_h^\wp} \| u_h \|_{L^2(\cell)}^2.
  \]
  Then, by \eqref{eq:mass-spectre}, we have the bound for $\Sigma_{\rm
  int}^{\wp,\F}$, that is,
  \[
    \sum_{\cell \in \T_h^\wp} \| u_h \|_{L^2(\cell)}^2 \gtrsim
    \sum_{\cell \in \T_h^\wp} h_\cell^d \| \u_\cell \|_2^2 \geq \sum_{\sigma 
    \in \Sigma_{\rm int}^{\wp,\F}} h_\sigma^d u_\sigma^2.
  \]
  On the other hand, we let $\mathcal{F}_{\rm C}$ denote all the set of
  \acp{vef} $f_{\rm C}$, that contain at least one \ac{dof} of $\Sigma_{\rm
  ext}^{\wp,\F}$. Using Lemma~\ref{lemma:hanging-to-owner}, we pick for each
  $f_{\rm C} \in \mathcal{F}_{\rm C}$ a hanging \ac{vef} $f_{\rm H}$ of the same
  dimension of $f_{\rm C}$, with $f_{\rm H}$ touching a well-posed cell. We
  denote the set of all $f_{\rm H}$ by $\mathcal{F}_{\rm H}$. By
  \eqref{eq:l2-f-bound} and
  Lemma~\ref{lemma:l2-hanging-to-l2-coarse-dofs}, we obtain a bound for
  $\Sigma_{\rm ext}^{\wp,\F}$:
  \[
    \sum_{\cell \in \T_h^\wp} \| u_h \|_{L^2(\cell)}^2 \gtrsim 
    \sum_{f_{\rm H} \in \mathcal{F}_{\rm H}} h_{f_{\rm H}}^{d-d_{f_{\rm H}}} 
    \| u_h \|_{L^2(f_H)}^2 \gtrsim \sum_{f_{\rm C} \in \mathcal{F}_{\rm C}} 
    h_{f_{\rm C}}^d \| \u_{f_{\rm C}} \|_2^2 \gtrsim \sum_{\sigma \in 
    \Sigma_{\rm ext}^{\wp,\F}} h_\sigma^d u_\sigma^2.
  \]
  Combining the two bounds together, we get
  \[
    \| u_h \|_{L^2(\Omega)}^2 \gtrsim \sum_{\sigma \in \Sigma^{\wp,\F}} 
    h_\sigma^d u_\sigma^2 \gtrsim \| \u \|_\sigma^2.
  \]
\end{proof}
Note that the constants in Proposition~\ref{prop:bound-mass-matrix} depend on
the well-posedness threshold via Lemma~\ref{lem:l2-norm-cut-bound}, but are
independent on the cut location.  The following result is a direct consequence of
Proposition~\ref{prop:bound-mass-matrix}.
\begin{corollary}\label{corol:cn-mass-matrix}
  The mass matrix $\mathbf{M}$ related to the aggregated \ac{fe} space $\V_h^\ag$
  is bounded by $k(\mathbf{M}) \leq C$, for a positive constant $C > 0$
  independent on cut location.
\end{corollary}

\subsection{Well-posedness of the unfitted \ac{fe} Problem~\eqref{eq:weak-PoissonEq}}
\label{appendix:proofs-wp}

Our goal now is to prove coercivity and continuity of the bilinear form
in~\eqref{eq:weak-PoissonEq}. To this end, let us assume that we bound the
maximum level of refinement for any triangulation $\T_h$ built recursively as a
forest-of-trees; this is the case in practice, since available memory is
limited. Hence, there exists $h_{\min} > 0$ such that $\min_{T \in \T_h} h_\cell
\geq h_{\min} > 0$. We begin with a trace
inequality that is key to prove coercivity:

Given $\cell \in \T_h^\ip$ and $\cell_1, \ldots, \cell_{m_\cell} \in \T_h^\wp$,
$m_\cell \geq 1$, the set of constraining well-posed cells (i.e.~those
constraining at least one \ac{dof} of $\cell$), we let
\[
  \Omega_\cell^\act \doteq \left( \cell \cup \bigcup_{i=1}^{m_\cell} 
  \cell_i \right) \ \text{and} \ \Omega_\cell \doteq \Omega \cap 
  \Omega_\cell^\act.
\]
Note that $m_\cell$ is bounded, due to the 2:1 0-balance restriction and the
fact that the number of neighbour cells is bounded. In case that $\cell \in
\T_h^\wp$, the definitions above become $\Omega_\cell^\act = \cell$ and
$\Omega_\cell = \Omega \cap \cell$.
\begin{lemma}\label{lem:trace-inv-grads}
  Given $u_h \in \V_h^\ag$ and $\cell \in \T_h^\act$, there exists $C(\eta_0) > 0$,
  such that 
  \[
    \| \boldsymbol{n} \cdot \nabla u_h \|_{L^2( \Gamma_{\rm D} \cap \cell )}^2
  \leq C(\eta_0) h_\cell^{-1} \| \nabla u_h \|_{L^2( \Omega_\cell )}^2.
  \]
\end{lemma}
\begin{proof}
  We note first that $|\Gamma \cap \cell| |\cell|^{-1} \leq C h_\cell^{-1}$; it
  can be proven for piecewise smooth boundaries for a constant that depends on
  the curvature of the surface patches and the maximum number of patches
  intersecting a cell. Combining this bound with the fact that constraints are
  bounded (cf.~\eqref{eq:bounds-K}), we can readily use the ideas of the proof
  in~\cite[Lemma 5.6]{Badia2018}, followed by Lemma~\ref{lem:l2-norm-cut-bound},
  to prove the result:
  \[
    \| \boldsymbol{n} \cdot \nabla u_h \|_{L^2( \Gamma_{\rm D} \cap \cell 
    )}^2 \lesssim h_\cell^{-1} \| \nabla u_h \|_{L^2( \Omega_\cell^\act )}^2 
    \lesssim C(\eta_0) h_\cell^{-1} \| \nabla u_h \|_{L^2( \Omega_\cell )}^2.
  \]
\end{proof}
We let now $\V(h) \doteq \V_h^\ag + H^2(\Omega) \cap H_0^1(\Omega)$ and define the
following mesh dependent norms for $v \in \V(h)$:
\[
  \begin{aligned}
    \opnormh{v}^2 &\doteq \| \nabla v \|_{L^2(\Omega)}^2 + \sum_{\cell \in 
    \T_h^\act} \beta_\cell h_\cell^{-1} \| v \|_{L^2(\Gamma_{\rm D} \cap 
    \cell)}^2, \\
    \opnormvh{v}^2 &\doteq \opnormh{v}^2 + \sum_{\cell \in 
    \T_h^\act} h_\cell \| \boldsymbol{n} \cdot \nabla{v} \|_{L^2(\Gamma_{\rm D} 
    \cap \cell)}^2.
  \end{aligned}
\]
\begin{remark}
  By Lemma~\ref{lem:trace-inv-grads}, norms $\opnormh{\cdot}$ and 
  $\opnormvh{\cdot}$ are equivalent in $\V_h^\ag$.
  \label{rem:norms-equiv}
\end{remark}
\noindent{In what follows, we assume that $\Omega$ has smoothing properties.
Then we have the Discrete Poincaré-type inequality (see, e.g.~\cite[Lemma
5.8]{Badia2018})}
\begin{equation}\label{eq:norm-is-norm}
  \| v \|_{L^2(\Omega)} \lesssim \opnormh{v}, \quad \text{for} \ \text{any} 
  \ v \in \V(h).
\end{equation}
\begin{theorem}\label{th:wp}
  The aggregated unfitted \ac{fe} problem in~\eqref{eq:weak-PoissonEq} satisfies
  the following bounds:
  \begin{itemize}
    \item[i)] Coercivity: 
      \begin{equation}\label{eq:th-coercivity}
        a(u_h,u_h) \gtrsim \opnormh{u_h}^2, \quad \text{for any} \ u_h \in \V_h^\ag,
      \end{equation}
    \item [ii)] Continuity:
      \begin{equation}\label{eq:th-continuity}
        a(u,v) \lesssim \opnormvh{u} \opnormvh{v}, \quad \text{for any} \ u, 
        v \in \V(h),
      \end{equation}
  \end{itemize} 
  if $\beta_\cell > C(\eta_0)$, for some positive constant $C(\eta_0)$. In this 
  case, there exists one and only one solution of~\eqref{eq:weak-PoissonEq}.
\end{theorem}
\begin{proof}
  The proof is analogous to~\cite[Theorem 5.7]{Badia2018}. Hence, we omit
  details. In order to show coercivity, given $u_h \in \V_h^\ag$, since we have
  that 
  \[
    a(u_h,u_h) = \opnormh{u_h}^2 - 2 \int_{\Gamma_{\rm D}} u_h ( \boldsymbol{n}
    \cdot \nabla u_h ) {\rm d}\Gamma,
  \]
  it suffices to show that $2 \int_{\Gamma_{\rm D}} u_h ( \boldsymbol{n} \cdot
  \nabla u_h ) {\rm d}\Gamma \lesssim \opnormh{u_h}^2$. For a (well- or
  ill-posed) cut cell $\cell$, usage of the Cauchy-Schwarz inequality, Young's
  inequality and Lemma~\ref{lem:trace-inv-grads} leads to
  \[
    2 \int_{\Gamma_{\rm D} \cap \cell} u_h ( \boldsymbol{n} \cdot \nabla u_h ) 
    {\rm d}\Gamma \leq \alpha_\cell C(\eta_0) h_\cell^{-1} \| u_h \|_{L^2( 
    \Gamma_{\rm D} \cap \cell)}^2 + \alpha_\cell^{-1} \| \nabla u_h \|_{L^2( 
    \Omega_\cell )}^2
  \]
 For tree-based meshes, the number of neighbouring cells is bounded and
  the cell sizes $h_\cell$ of $\cell \in \Omega_\cell$ differ by a bounded
  value, that depends on the 2:1 0-balance restriction and the maximum
  aggregation distance, one can take $\alpha_\cell > 0$ large enough, but
  uniform with respect to $h_\cell$ and cut location, such that:
  \[
    2 \int_{\Gamma_{\rm D}} u_h ( \boldsymbol{n} \cdot \nabla u_h ) 
    {\rm d}\Gamma \leq \sum_{\cell \in \T_h^\act} \alpha_\cell C(\eta_0) 
    h_\cell^{-1} \| u_h \|_{L^2(\Gamma_{\rm D} \cap \cell)}^2 + \frac{1}{2}
    \| \nabla u_h \|_{L^2( \Omega )}^2
  \]
  Therefore,
  \[
    a(u_h,u_h) \geq \frac{1}{2} \| \nabla u_h \|_{L^2( \Omega )}^2 + 
    \sum_{\cell \in \T_h^\act} ( \beta_\cell - \alpha_\cell C(\eta_0) ) h_\cell^{-1} 
    \| u_h \|_{L^2(\Gamma_{\rm D} \cap \cell)}^2
  \]
  For, e.g.~$\beta_\cell > 2 \alpha_\cell C(\eta_0)$, $a(u_h,u_h)$ is a norm. By
  construction, the lower bound for $\beta_\cell$ is independent of the
  $h_\cell$ and the intersection of $\Gamma_{\rm D}$ and $\T_h^\act$, but it
  depends on the well-posedness threshold $\eta_0$, which is a user-defined
  value. It proves the coercivity property in~\eqref{eq:th-coercivity}. Thus,
  the bilinear form is non-singular. The continuity in~\eqref{eq:th-continuity}
  can be readily proved by repeated use of the Cauchy-Schwarz inequality. Since
  the problem is finite-dimensional and the corresponding linear system matrix
  is non-singular, there exists one and only one solution of this problem.
\end{proof}
The linear system matrix that arises from problem~\eqref{eq:weak-PoissonEq} can
be defined as
\begin{equation}\label{eq:system-mat}
  A_{\sigma\sigma'} \doteq a(\tilde{\phi}^\sigma,\tilde{\phi}^{\sigma'}), \quad 
  \text{for} \ \sigma,\sigma' \in \Sigma^{\wp,\F},
\end{equation}
and we have that $\u \cdot \mathbf{A} \u = a(u_h,u_h)$, for any $u_h \in
\V_h^\ag$.
% We can now use the standard inverse inequality, see, 
% e.g.~\cite[p.111]{brenner_mathematical_2010},
% \begin{equation}\label{eq:std-inv-ineq}
%   \| \nabla u_h \|_{L^2(\Omega^\act)} \lesssim h_{\min}^{-1} \| u_h 
%   \|_{L^2(\Omega^\act)}, \quad \text{for any} \ u_h \in \V_h^\std,
% \end{equation}
% together with 
We can now use Proposition~\ref{prop:bound-mass-matrix} and Theorem~\ref{th:wp}
to show that we have the same bound as the body fitted problem for the linear
system matrix. This comes as a consequence of the following:
\begin{proposition}\label{prop:bound-system-matrix}
  Given $u_h \in \V_h^\ag$, the following bound holds:
  \[
    \| \u \|_\sigma^2 \lesssim a(u_h,u_h) \lesssim h_{\min}^{-2} \| \u 
    \|_\sigma^2.
  \]
\end{proposition}
\begin{proof}
  The lower bound readily follows from coercivity in 
  \eqref{eq:th-coercivity}, \eqref{eq:norm-is-norm} and the
  lower bound of Proposition~\ref{prop:bound-mass-matrix}:
  \[
    a(u_h,u_h) \gtrsim \opnormh{u_h}^2 \gtrsim \| u_h \|_{L^2(\Omega)}^2 
    \gtrsim \| \u \|_\sigma^2
  \]
  For the upper bound, we first see that the boundary term of $\opnormh{\cdot}$
  is bounded by $\| \u \|_\sigma^2$. Indeed, by scaling arguments and the
  equivalence of norms for finite-dimensional spaces, we have that
  \[
    \| u_h \|_{L^2(\Gamma_{\rm D} \cap \cell)}^2 \lesssim h_\cell^{d-1} 
    \| \u_\cell \|_2^2,
  \]
  where $\u_\cell$ gathers both free and constrained \acp{dof}. Adding up for
  all cells, invoking the fact that the number of neighbour cells and constraint
  coefficients are bounded (see~\eqref{eq:bounds-K}), we obtain:
  \begin{equation}\label{eq:lbound-1}
    \sum_{\cell \in \T_h^\act} \beta_\cell h_\cell^{-1} \| u_h \|_{L^2(\Gamma_{\rm 
    D} \cap \cell)}^2 \lesssim h_{\min}^{-2} \| \u \|_\sigma^2.
  \end{equation}
  % On the other hand, using \eqref{eq:std-inv-ineq} and the upper bound
  % of Proposition~\ref{prop:bound-mass-matrix}, which also holds for
  % $\Omega^\act$, we get
  On the other hand, using a standard inverse inequality and the upper bound of
  Proposition~\ref{prop:bound-mass-matrix}, which also holds for $\Omega^\act$,
  we get
  \begin{equation}\label{eq:lbound-2}
    \| \nabla u_h \|_{L^2(\Omega)}^2 \leq \| \nabla u_h \|_{L^2(\Omega^\act)}^2
    \lesssim h_{\min}^{-2} \| u_h \|_{L^2(\Omega^\act)}^2 \lesssim h_{\min}^{-2} 
    \| \u \|_\sigma^2.
  \end{equation}
  Combining continuity of the bilinear form~\eqref{eq:weak-PoissonEq}, in
  \eqref{eq:th-continuity}, with
  Remark~\ref{rem:norms-equiv},~\eqref{eq:lbound-1} and~\eqref{eq:lbound-2}, we
  get the sought-after upper bound:
  \[
    a(u_h,u_h) \lesssim \opnormvh{u_h}^2 \lesssim \opnormh{u_h}^2 
    \lesssim h_{\min}^{-2} \| \u \|_\sigma^2.
  \]
\end{proof}
By recalling Corollary~\ref{corol:cn-mass-matrix}, we obtain the following 
condition number bound.
\begin{corollary}
  The condition number of the linear system matrix $\mathbf{A}$, associated to
  Problem~\eqref{eq:weak-PoissonEq}, preconditioned by the mass matrix
  $\mathbf{M}$, related to the aggregated \ac{fe} space $\V_h^\ag$, satisfies the
  bound $k(\mathbf{M}^{-1}\mathbf{A}) \le C h_{\min}^{-2}$, for a positive
  constant $C > 0$ independent on cut location.
\end{corollary}

A priori error estimates can be proved following the same steps in~\cite[Section
5.6]{Badia2018} and, for conciseness, are not covered here. The key arguments,
leading to the estimates, are standard \ac{fe} arguments, the results above and
the fact that the nodal interpolator of a continuous function $u$ in
$\mathcal{C}^0(\overline{\Omega})$, defined as $\mathcal{I}_h(u) \doteq
\sum_{\sigma \in \Sigma^{\wp,\F}} u(\x_\sigma) \tilde{\phi}^\sigma$, is bounded
above, since constraints are also bounded above (see~\eqref{eq:bounds-K}).

\begin{small}

\section*{Acknowledgements}

\thethanks

\setlength{\bibsep}{0.0ex plus 0.00ex}
\bibliographystyle{myabbrvnat}
\bibliography{art042}

\end{small}

\end{document}